\documentclass{article}

\usepackage{arxiv}

\usepackage[utf8]{inputenc}
\usepackage[T1]{fontenc}
\usepackage[english]{babel} 

\usepackage{amsmath,amssymb,amsthm,mathtools}
\usepackage{mathrsfs}
\usepackage{bm}
\usepackage{dsfont}
\usepackage{bbm}

\usepackage{graphicx}
\usepackage{epsfig}

\usepackage{xcolor}        
\usepackage{hyperref}
\usepackage{cleveref}      
\usepackage{url}
\usepackage{booktabs}

\usepackage{microtype}
\usepackage{enumitem}      
\usepackage{comment}
\usepackage{csquotes}      

\usepackage{fouriernc}     

\hypersetup{
	colorlinks=true,
	linkcolor=blue,
	citecolor=blue,
	urlcolor=blue,
}

\newtheorem{theorem}{Theorem}[section]
\newtheorem{lemma}[theorem]{Lemma}
\newtheorem{definition}[theorem]{Definition}
\theoremstyle{remark} \newtheorem{remark}{Remark}

\title{Singular Perturbation in Multiscale Stochastic Control Problems with Domain Restriction in the Slow Variable}

  \author{
  	Anderson O. Calixto\thanks{The author would like to thank FAPERJ grant E-26/202.535/2022} \\
  	Departamento de Estatística \\
  	Universidade Federal Fluminense \\
  	\texttt{acalixto@id.uff.br}
   \And
  Bernardo Freitas Paulo da Costa \\
  School of Applied Mathematics  \\
  Getulio Vargas Foundation\\
  \texttt{bernardo.paulo@fgv.br} \\
  \And
  Glauco Valle \thanks{The author would like to thank FAPERJ grant E-26/202.636/2019 and CNPq grants 307938/2022-0 and 403423/20236}
   \\
  Instituto de Matemática \\
  Universidade Federal do Rio de Janeiro\\
  \texttt{glauco.valle@im.ufrj.br} \\
}
\begin{document}
\maketitle
\begin{abstract}
  We study a multiscale stochastic optimal control problem subject to state constraints on the slow variable. To address this class of problems, we develop a rigorous theoretical framework based on singular perturbation analysis, tailored to settings with constrained dynamics. Our approach relies on the theory of viscosity solutions for degenerate Hamilton–Jacobi–Bellman equations with Neumann-type boundary conditions. We also establish the convergence of the multiscale value functions in the infinite-horizon regime. Finally, we present two illustrative examples that highlight the applicability and effectiveness of the proposed framework.
\end{abstract}
\vspace{1em}
\noindent\textbf{Keywords:} Multiscale Stochastic Optimal Control, Multiscale Hamilton–Jacobi–Bellman equation, Degenerate Elliptic PDEs, Viscosity Solutions, State Constraints, Neumann Boundary Condition.
\section{Introduction}
\label{sec:introduction}
Multiscale Stochastic Optimal Control (MSOC) problems have been extensively studied, particularly in connection with singular perturbation methods and homogenization techniques. These approaches aim to characterize the effective behavior of systems governed by dynamics at distinct time scales, such as slow and fast variables. Foundational contributions in this direction include the works of Alvarez and Bardi \cite{AlvarezBardi01,AlvarezBardi02}, who employed viscosity solution techniques to analyze singular perturbations in deterministic and stochastic control. Their results established general convergence principles for the associated Hamilton--Jacobi--Bellman (HJB) equations under perturbation regimes.

In contrast, Kushner’s monograph \cite{Kushner} develops a complementary probabilistic approach based on weak convergence methods. This framework rigorously addresses singularly perturbed stochastic control and filtering problems by focusing on probabilistic limit theorems and convergence of stochastic processes. Thus, Kushner’s methods provide a different, yet foundational, perspective that contrasts with the analytical PDE-based techniques predominant in the viscosity solution literature.

In a related line of research, Alvarez, Bardi, and Marchi \cite{AlvarezBardiMarchi} investigated homogenization problems for second-order degenerate HJB equations in multiscale settings, while Bardi and Cesaroni \cite{BardiCesaroni} developed a multiscale framework where the fast dynamics are modeled by ergodic diffusion processes. These works provided rigorous foundations for deriving effective equations that govern the long-term behavior of multiscale control systems. The comprehensive book by Alvarez and Bardi \cite{AlvarezBardi} further elaborates on ergodicity, stabilization, and singular perturbations for Bellman–Isaacs equations, deepening the theoretical underpinnings of such problems.

The theory has also been applied to finance. For instance, Bardi, Cesaroni, and Manca \cite{BardiCesaroniManca} studied singular perturbation problems in stochastic volatility models, employing viscosity methods to analyze convergence in financial settings. More recently, Bardi and Kouhkouh \cite{BardiKouhkouh01,BardiKouhkouh02} addressed control problems inspired by deep learning and optimization, particularly focusing on degenerate stochastic systems with unbounded data and relaxation dynamics. 

While these contributions address multiscale behavior, degenerate dynamics, and complex control structures, the issue of state constraints --- particularly in combination with multiscale dynamics --- has remained largely unexplored. State constraints introduce significant mathematical challenges, especially when dealing with degenerate HJB equations and Neumann-type boundary conditions.

A crucial step toward addressing this gap was taken in the companion paper \cite{CalixtoCostaValle2025}, where a rigorous framework was developed for Stochastic Optimal Control (SOC) problems with state constraints in a non-multiscale setting. That work dealt with fully nonlinear degenerate elliptic HJB equations, where the diffusion is control-dependent and potentially degenerate, and the constraint is enforced through a nontrivial Neumann boundary condition. The authors established existence and uniqueness of viscosity solutions in compact domains and presented an illustrative application.

The present work builds directly upon that foundation, extending the theory to multiscale settings with singular perturbations and time-scale separation. In particular, we investigate convergence of value functions in the infinite-horizon regime for constrained systems where the restriction is imposed on the slow variable. Our contribution complements and generalizes existing literature by unifying multiscale stochastic control with boundary-constrained formulations in a viscosity solution framework.

This article is organized as follows:

In Section~\S\ref{sec:multiscale_stochastic_optimal_control}, we introduce the constrained MSOC problem, where the state constraint is imposed on the slow variable. We rigorously formulate the problem and establish its well-posedness—namely, the existence and uniqueness of the value function—by applying the foundational results developed in~\cite[Section~2]{CalixtoCostaValle2025}, which address the state-constrained SOC setting in a single-scale framework.

In Section~\S\ref{subsec:effective_stochastic_optimal_control}, we derive the effective (or homogenized) control problem that captures the asymptotic behavior of the original multiscale system as the time-scale separation parameter tends to zero. This effective problem retains the state constraint structure and is analyzed independently to ensure well-posedness under appropriate regularity and ergodicity conditions.

In Section~\S\ref{sec:convergence_multiscale_optimal_value_function_effective_optimal_value_function} contains the main result of the paper. We prove that the value function associated with the multiscale problem converges uniformly to the value function of the effective problem, thereby justifying the limiting procedure and validating the use of the effective model as a reliable approximation in the asymptotic regime.

In Section~\S\ref{sec:examples}, we present two examples that illustrate the applicability of the proposed theoretical framework. Subsection~\S\ref{subsec:control_problem_semilinear_hjb_equation} addresses an MSOC problem with constraints on the slow variable and no control in the diffusion matrix, leading to a semilinear, degenerate elliptic HJB equation with boundary conditions. Assuming the existence of classical solutions, we examine the structure of the corresponding Markovian control.

In Subsection~\S\ref{subsec:control_problem_fully_nonlinear_HJB_equation}, we analyze a more complex case where the control appears in the diffusion of the slow subsystem. This results in a fully nonlinear, degenerate HJB equation that encapsulates the full scope of the framework and, to the best of our knowledge, lies beyond the reach of existing methodologies.

Finally, in Appendix \ref{sec:appendix_A}, we study the differentiability of the Fokker–Planck equation with periodic boundary conditions with respect to the parameters of the slow scale.

\section{Multiscale Stochastic Optimal Control}
  \label{sec:multiscale_stochastic_optimal_control}
  We consider a stochastic basis \(\big(\Omega, \mathcal{O}, (\mathcal{F}_t)_{t \geq 0}, \mathbb{P}\big)\) satisfying the usual conditions; that is, the filtration \((\mathcal{F}_t)_{t \geq 0}\) is right-continuous and contains all \(\mathbb{P}\)-null sets. Furthermore, \((W(t))_{t \geq 0}\) denotes a standard \(d_W\)-dimensional Brownian motion defined on this basis. In the context of strong solutions to Stochastic Differential Equations (SDEs), the filtration \((\mathcal{F}_t)_{t \geq 0}\) is taken to be the one generated by the Brownian motion, augmented with the \(\mathbb{P}\)-null sets.
  
  We also fix a compact set \(\mathbb{X} \subset \mathbb{R}^d\), which plays a central role in the formulation of the state constraint. Specifically, we assume the existence of a function \(\phi \in \text{C}^2_{\text{b}}\left(\mathbb{R}^d\right)\) such that:
  \begin{itemize}
  	\item \(\mathbb{X} = \left\{x \in \mathbb{R}^d : \phi\big(x\big) \leqslant 0\right\}\);
  	\item \(\text{Int}\left(\mathbb{X}\right) = \left\{x \in \mathbb{R}^d : \phi\big(x\big) < 0\right\}\);
  	\item \(\partial \mathbb{X} = \left\{x \in \mathbb{R}^d : \phi\big(x\big) = 0\right\}\), with \(\|D_x\phi\big(x\big)\| = 1\) for all \(x \in \partial \mathbb{X}\).
  \end{itemize}
  
  We assume that the fast variable in the multiscale system is periodic. To formalize this assumption, we introduce the following definition:
  \begin{definition}[$\mathbb{Y}$-Periodic Function]
  	\label{def:function_Y_periodic}
  	Consider the set $\mathbb{Y}:=\big[0,1\big]^{d_Y}$ and $\big\{e_1,...,e_{d_Y}\big\}$ the vectors of the canonical basis of $\mathbb{R}^{d_Y}$. A function $g:\mathbb{R}^{d_Y}\to\mathbb{R}$ such that 
  	\begin{equation}
  		\label{eq:function_Y_periodic}
  		g\big(y+ke_i\big)=g\big(y\big)\hspace{0.15cm}\forall y\in\mathbb{R}^{d_Y},\hspace{0.10cm}\forall k\in\mathbb{Z}\hspace{0.15cm}\text{and}\hspace{0.15cm} \forall i\in \big\{1,...,\emph{d}_Y\big\}
  	\end{equation} 
  	is called $\mathbb{Y}$-\emph{Periodic}.
  \end{definition}
   We will also consider the following function spaces:
  \begin{itemize}
  	\item $\text{C}^{\text{k}}_{\text{per}}\big(\mathbb{Y}\big)$ defined as the space of functions $\text{C}^{\text{k}}\big(\mathbb{R}^{d_Y}\big)$ that satisfy the condition \eqref{eq:function_Y_periodic}. Similarly, we also define the space $\text{C}^{\infty}_{\text{per}}\big(\mathbb{Y}\big)$. By identifying with the torus $\mathbb{T}^{d_Y}$ we will simply write $\text{C}^{\text{k}}\big(\mathbb{T}^{d_Y}\big)$ and $\text{C}^{\infty}\big(\mathbb{T}^{d_Y}\big)$.
  	\item $\text{C}_{\text{per}}^{k_x,k_y}\big(\mathbb{X}\times\mathbb{Y}\big)$ defined as the space of functions $\text{C}^{k_x,k_y}\big(\mathbb{X}\times\mathbb{R}^{d_Y}\big)$ that satisfy the condition \eqref{eq:function_Y_periodic} in the second variable. By identifying with the torus $\mathbb{T}^{d_Y}$ we will simply write $\text{C}^{k_x,k_y}\big(\mathbb{X}\times\mathbb{T}^{d_Y}\big)$.
  \end{itemize}
  
   Given a scale factor of $0<\epsilon\ll 1$, we want to control the following stochastic dynamical system on two scales:
  \begin{subequations}
  	\label{subeq:multiscale_stochastic_dynamic_system_macroscale_variable_constraint}
  	\begin{align}
  		\label{eq:macroscale_stochastic_dynamical_system}
  		dX^{\epsilon}_x(t)=&\mu_{X}\big(X^{\epsilon}_x(t),Y^{\epsilon}_y(t),u(t)\big)\,dt+\sigma_{X}\big(X^{\epsilon}_x(t),Y^{\epsilon}_y(t),u(t)\big)\,dW(t)-D_x\phi\big(X^{\epsilon}_x(t)\big)\,dl^{\epsilon}_{x,y}(t),\\
  		\label{eq:microscale_stochastic_dynamical_system}
  		dY^{\epsilon}_y(t)=&\frac{1}{\epsilon}\mu_{Y}\big(X^{\epsilon}_x(t),Y^{\epsilon}_y(t)\big)dt+\frac{1}{\sqrt{\epsilon}}\sigma_{Y}\big(X^{\epsilon}_x(t),Y^{\epsilon}_y(t)\big)\,dW(t),
  	\end{align}
  \end{subequations}
  in which we have the elements:
  \begin{itemize}
  	\item The processes \(\big(X^{\epsilon}_x(t)\big)_{t \geq 0}\) and \(\big(Y^{\epsilon}_y(t)\big)_{t \geq 0}\) are continuous and adapted to the filtration \(\big(\mathcal{F}_t\big)_{t \geq 0}\). The former is referred to as the \textit{slow variable}, and the latter as the \textit{fast variable}.
  	
  	\item The process \(\big(l^{\epsilon}_{x,y}(t)\big)_{t \geq 0}\) is continuous, non-decreasing, and adapted to \(\big(\mathcal{F}_t\big)_{t \geq 0}\), with \(l^{\epsilon}_{x,y}(0) = 0\), and it satisfies the condition:
  	\[
  	l^{\epsilon}_{x,y}(t) = \int_0^t \mathds{1}_{\partial \mathbb{X}}\big(X^{\epsilon}_x(s)\big)\,dl^{\epsilon}_{x,y}(s) \quad \mathbb{P}\text{-a.s.}
  	\]
  	\item The process \(\big(u(t)\big)_{t \geq 0}\) is progressively measurable with respect to \(\big(\mathcal{F}_t\big)_{t \geq 0}\), taking values in a convex and compact set \(\mathbb{U} \subset \mathbb{R}^m\). The set of such controls is denoted by \(\mathcal{U}\), and its elements are referred to as admissible controls.
  \end{itemize}
  
  \begin{remark}
  	The intuitive interpretation of \(-D_x\phi\big(X^{\epsilon}_x(t)\big)\) is as follows: this gradient acts as a force field that reflects the process \(\big(X^{\epsilon}_x(t)\big)_{t \geq 0}\) back into the domain \(\mathbb{X}\). The reflection is instantaneously triggered by the process \(\big(l^{\epsilon}_{x,y}(t)\big)_{t \geq 0}\) whenever the state \(X^{\epsilon}_x(t)\) reaches the boundary \(\partial \mathbb{X}\).
  \end{remark}
  
  We now introduce the following hypothesis for the drift and dispersion fields:
  \begin{itemize}
  	\item[I] The drift $\mu_{X}$ and dispersion $\sigma_{X}$ of the slow variable are continuous in all variables and Lipschitz continuous in the spatial variables, uniformly with respect to the control variable. Additionally, these fields admit the following decomposition:
  	 \begin{equation*}
  	 	\mu_{X}\big(x,y,u\big)=\mu_{\text{SF}}\big(x,y\big)+\mu_{\text{SC}}\big(x,u\big)\quad\text{and}\quad\sigma_{X}\big(x,y,u\big)=\sigma_{\text{SF}}\big(x,y\big)+\sigma_{\text{SC}}\big(x,u\big)
  	 \end{equation*}
  	where SF denotes slow and fast variables, and SC refers to slow and control variables. Furthermore, we assume that these fields are periodic in the fast variable, as per condition \eqref{eq:function_Y_periodic}, uniformly with respect to both the slow and control variables. Finally, we assume that the following relationship holds:
  	  \begin{equation}
  	 	\label{eq:noncorrelation_between_parts_multiscale_dispersion}
  	 	\sigma_{\text{SC}}\big(x,u\big)\sigma^{\top}_{\text{SF}}\big(x,y\big)+\sigma_{\text{SF}}\big(x,y\big)\sigma^{\top}_{\text{SC}}\big(x,u\big)=0.
  	 \end{equation}

  	 \item[II] Initially, we assume that the drift  $\mu_{Y}$ and dispersion $\sigma_{Y}$ of the fast variable belong to the space $\text{C}^{0,2}\big(\mathbb{X}\times\mathbb{T}^{d_Y}\big)$. Additionally, they satisfy the following property: the diffusion matrix $\big[\sigma_{Y}\sigma^{\top}_{Y}\big]$ is positive definite uniformly over $\mathbb{X}\times\mathbb{T}^{d_Y}$. In other words, there exists a constant $c_0>0$ such that 
  	 	\begin{equation*}
  	 		\frac{1}{c_0}\|\xi\|^2\geqslant\bigg\langle \xi,\big[\sigma_{Y}\sigma^{\top}_{Y}\big]\big(x,y\big)\xi\bigg\rangle \geqslant c_0\|\xi\|^2\quad \forall \xi \in \mathbb{R}^{d_Y}\hspace{0.10cm}\text{and}\hspace{0.15cm}\forall (x,y)\in \mathbb{X}\times\mathbb{T}^{d_Y}.
  	 	\end{equation*}
  \end{itemize}
  In order to control the system \eqref{subeq:multiscale_stochastic_dynamic_system_macroscale_variable_constraint}, we must define the criteria for selecting the process $\big(u(t)\big)_{t\geqslant0}$ within the class of admissible processes $\mathcal{U}$.  Furthermore, we assume that:
  \begin{itemize}
  	\item[III] The operation cost $L$ is continuous in all variables and Lipschitz continuous in the spatial variables, uniformly with respect to the control variable. Additionally, the cost function admits the following decomposition:
  	\begin{equation*}
  		L\big(x,y,u\big)=L_{\text{SF}}\big(x,y\big)+L_{\text{SC}}\big(x,u\big)
  	\end{equation*}
  	and is periodic (condition \eqref{eq:function_Y_periodic}) in the fast variable $y$.
  	\item[IV] The preventive cost on the boundary $h\in \text{C}\big(\partial \mathbb{X}\big)$. 
  \end{itemize}
  
 The next hypothesis we introduce is the most delicate, and therefore requires a more detailed discussion. To establish that both the multiscale and effective value functions are viscosity solutions of their respective HJB equations, we rely on Theorem 2.11 from the companion paper~\cite{CalixtoCostaValle2025}. As is standard in the viscosity solution framework, we must verify separately that the candidate value function is a viscosity subsolution and a viscosity supersolution.
 
 In the proof of the subsolution property, an auxiliary control is employed. To guarantee strong solutions to the controlled stochastic differential equation, this control must be Lipschitz continuous. Such a control is classically synthesized as
 \begin{equation*}
 	u^{*}\big(x\big) \in \arg\min_{u \in \mathbb{U}} \bigg\{ \mathcal{L}^{u}_X \varphi\big(x\big) + L\big(x,u\big) \bigg\},
 \end{equation*}
 where $\varphi \in \text{C}^{\infty}\big(\mathbb{X}\big)$ is a smooth test function and $\mathcal{L}^{u}_X$ denotes the second-order differential operator associated with the controlled dynamics.
 
 To ensure the Lipschitz regularity of $u^*$, we invoke the following result:
 \begin{lemma}[\protect{\cite[Lemma 2.10]{CalixtoCostaValle2025}}]
 	\label{lm:Lipschitz_controls} 
 	Let $\mathbb{U} \subset \mathbb{R}^m$ and $\theta: \mathbb{X} \times \mathbb{U} \to \mathbb{R}$ be a function satisfying:
 	\begin{itemize}
 		\item $\mathbb{U}$ is a convex and compact set;
 		\item For every $x \in \mathbb{X}$, the mapping $u \mapsto \theta\big(x,u\big)$ belongs to $\emph{C}^2\big(\mathbb{U}\big)$;
 		\item The gradient $D_u \theta\big(\cdot, u\big)$ is Lipschitz continuous uniformly in $u \in \mathbb{U}$;
 		\item The Hessian $D^2_u \theta$ is positive definite uniformly on $(x,u) \in \mathbb{X} \times \mathbb{U}$.
 	\end{itemize}
 	Then, the mapping
 	\[
 	u^*\big(x\big) \in \arg\min_{u \in \mathbb{U}} \theta\big(x,u\big)
 	\]
 	is Lipschitz continuous.
 \end{lemma}
 
 In our setting, Lemma~\ref{lm:Lipschitz_controls} is applied to the function
 \begin{equation}
 	\label{eq:function_theta_varphi}
 	\theta_\varphi\big(x,u\big) := \mathcal{L}^{u}_X \varphi\big(x\big) + L\big(x,u\big).
 \end{equation}
 
 A concrete scenario in which conditions (1)–(4) of Lemma~\ref{lm:Lipschitz_controls} are satisfied for $\theta_\varphi$ arises under the following structural assumptions:
 \begin{itemize}
 	\item[5.1] For each $x \in \mathbb{X}$, the functions $\mu_X\big(x,\cdot\big)$, $\sigma_X\big(x,\cdot\big)$, and $L\big(x,\cdot\big)$ are of class $\text{C}^2\big(\mathbb{U}\big)$, and the mixed derivatives $D_{xu} \mu_X$ and $D_{xu} \sigma_X$ are continuous on $\mathbb{X} \times \mathbb{U}$;
 	\item[5.2] The following hold: $D^2_u \mu_X \equiv 0$, $D^2_u\big[\sigma_X \sigma_X^\top\big] \equiv 0$, and $D^2_u L \succ 0$.
 \end{itemize}
 
 It is worth noting that the examples provided in Section~\S\ref{sec:examples} satisfy both Assumptions 5.1 and 5.2. We emphasize, however, that these conditions are only sufficient to guarantee the applicability of Lemma~\ref{lm:Lipschitz_controls} to the function $\theta_\varphi$ defined in~\eqref{eq:function_theta_varphi}. 
 
  Finally, we remark that in the multiscale case, the function $\theta_\varphi$ depends on the parameter $\epsilon>0$. However, this does not pose any difficulty, as Theorem~2.11 is applied for each fixed $\epsilon > 0$, and uniformity in $\epsilon$ is not required.
 
 Accordingly, in~\cite[Theorem 2.11]{CalixtoCostaValle2025}, we adopt the following general hypothesis:
 
 \begin{itemize}
 	\item[V] Structural assumptions on the drift, diffusion, and running cost functions ensuring that Lemma~\ref{lm:Lipschitz_controls} holds for the function defined in~\eqref{eq:function_theta_varphi}.
 \end{itemize}
 
   Thus, given the initial conditions $(x,y)\in \mathbb{X}\times\mathbb{T}^{d_Y}$ and a discount rate $\beta>0$, we define the cost functional $J^{\beta,\epsilon}_{x,y}:\mathcal{U}\to\mathbb{R}$ by
  \begin{equation}
  	\label{eq:functional_cost_multiscale_operation}
  	J^{\beta,\epsilon}_{x,y}\big(u\big):=\mathbb{E}\Bigg[\int_{0}^{+\infty}e^{-\beta s}L\big(X^{\epsilon}_x(s),Y^{\epsilon}_y(s),u(s)\big)\,ds+\int_{0}^{+\infty}e^{-\beta s}h\big(X^{\epsilon}_x(s)\big)\,dl^{\epsilon}_{x,y}(s)\Bigg].
  \end{equation}
  We define the Multiscale Optimal Value Function:
  \begin{equation}
  	\label{eq:multiscale_optimal_value_function}
  	v^{\beta,\epsilon}\big(x,y\big):=\inf_{u\in\mathcal{U}}J^{\beta,\epsilon}_{x,y}\big(u\big).
  \end{equation}
  Just as we did in \cite[Subsection 2.2]{CalixtoCostaValle2025}, we associate \eqref{eq:multiscale_optimal_value_function} with the following identities from the Dynamic Programming Principle (DPP):
  \begin{subequations}
  	\label{subeq:multiscale_dynamic_programming_principle}
  	\begin{align}
  		\label{eq:multiscale_dynamic_programming_principle_inf_inf }
  		v^{\beta,\epsilon}\big(x,y\big)=\inf_{u\in \mathcal{U}}\inf_{\tau\in \mathcal{T}}\mathbb{E}\Bigg[&\int_{0}^{\tau}e^{-\beta s}L\big(X^{\epsilon}_x(s),Y^{\epsilon}_y(s),u(s)\big)\,ds+\int_{0}^{\tau}e^{-\beta s}h\big(X^{\epsilon}_x(s)\big)\,dl^{\epsilon}_{x,y}(s)+e^{-\beta\tau}v^{\beta,\epsilon}\big(X^{\epsilon}_x(\tau),Y^{\epsilon}_y(\tau)\big)\Bigg]\\
  		\label{eq:multiscale_dynamic_programming_principle_inf_sup }
  		=\inf_{u\in \mathcal{U}}\sup_{\tau\in \mathcal{T}}\mathbb{E}\Bigg[&\int_{0}^{\tau}e^{-\beta s}L\big(X^{\epsilon}_x(s),Y^{\epsilon}_y(s),u(s)\big)\,ds+\int_{0}^{\tau}e^{-\beta s}h\big(X^{\epsilon}_x(s)\big)\,dl^{\epsilon}_{x,y}(s)+e^{-\beta\tau}v^{\beta,\epsilon}\big(X^{\epsilon}_x(\tau),Y^{\epsilon}_y(\tau)\big)\Bigg]
  	\end{align}
  \end{subequations}
  where $\mathcal{T}$ is the set of stopping times with respect to $\big(\mathcal{F}_t\big)_{t\geq0}$. Similarly to what we did in \cite[Subsection 2.3]{CalixtoCostaValle2025}, we associate the control problem \eqref{subeq:multiscale_stochastic_dynamic_system_macroscale_variable_constraint}-\eqref{eq:multiscale_optimal_value_function} with the following multiscale HJB equation:
  \begin{subequations}
  	\label{subeq:multiscale_HJB_Von_Neumann_boundary}
  	\begin{align}
  		\label{eq:multiscale_HJB_equation}
  		\beta v^{\beta,\epsilon}-\mathcal{H}\Bigg(x,y,D_{x}v^{\beta,\epsilon},\frac{D_{y}v^{\beta,\epsilon}}{\epsilon},D^2_{x}v^{\beta,\epsilon},\frac{D^2_{y}v^{\beta,\epsilon}}{\epsilon},\frac{D^2_{x,y}v^{\beta,\epsilon}}{\sqrt{\epsilon}}\Bigg)&=0\hspace{0.20cm}\forall (x,y)\in \mathbb{X}\times \mathbb{T}^{d_Y}\\
  		\label{eq:Von_Neumann_boundary_condition}
  		\bigg\langle D_{x}v^{\beta,\epsilon},D_x\phi\bigg\rangle&=h\hspace{0.20cm}\forall (x,y)\in \partial \mathbb{X}\times\mathbb{T}^{d_Y}
  	\end{align}
  \end{subequations}
  where the function $\mathcal{H}: \mathbb{X}\times\mathbb{T}^{d_Y}\times\mathbb{R}^{d_X}\times\mathbb{R}^{d_Y}\times\mathbb{S}^{d_X}\big(\mathbb{R}\big)\times \mathbb{S}^{d_Y}\big(\mathbb{R}\big)\times \mathbb{M}^{d_X\times d_Y}\to\mathbb{R}$ is defined by
 \begin{equation}
 	\label{eq:multiscale_Hamiltonian}
 	\mathcal{H}\Big(x,y,g_x,g_y,H_x,H_y,H_{xy}\Big):=\min_{u\in \mathbb{U}}\Bigg\{\mathcal{G}_{X}^{u}\big(x,g_x,H_x\big|y\big)+\mathcal{G}_{Y}\big(y,g_y,H_y\big|x\big)+L\big(x,y,u\big)+\text{Tr}\Big(H_{xy}\big[\sigma_{X}\sigma^{\top}_{Y}\big]\big(x,y,u\big)\Big)\Bigg\}
 \end{equation}
 and is called the Hamiltonian of the system. The function $\mathcal{H}$ is composed of the elements:
 \begin{itemize}
 	\item For each $y\in\mathbb{T}^{d_Y}$ and $u\in \mathbb{U}$, the function $\mathcal{G}_{X}^{u}:\mathbb{X}\times \mathbb{R}^{d_X}\times \mathbb{S}^{d_X}(\mathbb{R})\to\mathbb{R}$ defined by
    \begin{equation*}
 		\mathcal{G}_{X}^{u}\big(x,g_x,H_x\big|y\big):=\big\langle \mu_{X}\big(x,y,u\big),g_x\big\rangle +\frac{1}{2}\text{Tr}\Big(H_x\big[\sigma_{X}\sigma_{X}^{\top}\big]\big(x,y,u\big)\Big)
 	\end{equation*}
 	is the infinitesimal generator of the SDE \eqref{eq:macroscale_stochastic_dynamical_system}.
 	\item For each $x\in \mathbb{X}$, the function $\mathcal{G}_{Y}:\mathbb{T}^{d_Y}\times \mathbb{R}^{d_Y}\times \mathbb{S}^{d_Y}(\mathbb{R})\to\mathbb{R}$ defined by
 	\begin{equation*}
 		\mathcal{G}_{Y}\big(y,g_y,H_y\big|x\big):=\big\langle \mu_{Y}\big(x,y\big),g_y\big\rangle +\frac{1}{2}\text{Tr}\Big(H_y\big[\sigma_{Y}\sigma_{Y}^{\top}\big]\big(x,y\big)\Big)
 	\end{equation*}
 	is the infinitesimal generator of the SDE \eqref{eq:microscale_stochastic_dynamical_system}
 \end{itemize}
  
Finally, the definition of viscosity solutions adopted in this paper is the same as that used in the companion paper~\cite{CalixtoCostaValle2025}, namely:
 
 \begin{definition}[Viscosity Solutions]
 	\label{def:viscosity_solutions}
 	\text{ }
 	\begin{itemize}
 		\item $w\in \emph{USC}\big(\mathbb{X}\big)$ is an \emph{viscosity subsolution}, if for every $\varphi\in \emph{C}^{\infty}\big(\mathbb{X}\big)$ such that, the application $x\mapsto\big(w-\varphi\big)\big(x\big)$ has a local maximum at the point $x\in \mathbb{X}$ with $w\big(x\big)=\varphi\big(x\big)$, the following inequalities hold:
 		\begin{subequations}
 			\label{subeq:viscosity_subsolution}
 			\begin{align}
 				\label{eq:inner_boundary_viscosity_subsolution}
 				F^{\beta}\Big(x,\varphi\big(x\big),D_{x}\varphi\big(x\big),D^2_{x}\varphi\big(x\big)\Big)&\leqslant0\quad x\in \mathbb{X},\\
 				\label{eq:boundary_viscosity_subsolution}
 				F^{\beta}\Big(x,\varphi\big(x\big),D_x\varphi\big(x\big), D^2_x\varphi\big(x\big)\Big)\wedge\Gamma\Big(x,D_x\varphi\big(x\big)\Big)&\leqslant0\quad x\in\partial \mathbb{X}.
 			\end{align}
 		\end{subequations}	
 		\item $w\in \emph{LSC}\big(\mathbb{X}\big)$ is an \emph{viscosity supersolution}, if for every $\varphi\in \emph{C}^{\infty}\big(\mathbb{X}\big)$ such that, the application $x\mapsto\big(w-\varphi\big)\big(x\big)$ has a local minimum at the point $x\in \mathbb{X}$ with $w\big(x\big)=\varphi\big(x\big)$, the following inequalities hold:
 		\begin{subequations}
 			\label{subeq:viscosity_supersolution}
 			\begin{align}
 				\label{eq:inner_boundary_viscosity_supersolution}
 				F^{\beta}\Big(x,\varphi\big(x\big),D_x\varphi\big(x\big), D^2_x\varphi\big(x\big)\Big)&\geqslant0\quad x\in \mathbb{X},\\
 				\label{eq:boundary_viscosity_supersolution}
 				F^{\beta}\Big(x,\varphi\big(x\big),D_x\varphi\big(x\big),
 				D^2_x\varphi\big(x\big)\Big)\lor\Gamma\Big(x,D_x\varphi\big(x\big)\Big)&\geqslant0\quad x\in\partial \mathbb{X}.
 			\end{align}
 		\end{subequations}
 		\item $w\in \emph{C}\big(\mathbb{X}\big)$ is an \emph{viscosity solution} if it is at the same time a subsolution and a supersolution. 
 	\end{itemize}
 \end{definition}
 \subsection{Multiscale Optimal Value Function as a Viscosity Solution}
 \label{subsec:multiscale_optimal_value_function_viscosity_solution}
 Consider the following SDE with control and reflection on the boundary of the compact set $\widetilde{\mathbb{X}}:=\mathbb{X}\times\mathbb{T}^{d_Y}$:
 \begin{equation}
 	\label{eq:auxiliary_multiscale_SDE_fast_slow_system}
 	dZ^{\epsilon}_z(t)=\widetilde{\mu}^{\epsilon}_{Z}\big(Z^{\epsilon}_z(t),u(t)\big)\,dt+\widetilde{\sigma}^{\epsilon}_{Z}\big(Z^{\epsilon}_z(t),u(t)\big)\,dW(t)-D_z\widetilde{\phi}\big(Z^{\epsilon}_z(t)\big)\,dl^{\epsilon}_z(t),
 \end{equation}
 with $Z^{\epsilon}_z(t):=\big(X^{\epsilon}_x(t),Y^{\epsilon}_y(t)\big)$. In addition, we also have the following elements:
 \begin{itemize}
 	\item $\widetilde{\phi}:\mathbb{R}^{d_X}\times\mathbb{T}^{d_Y}\to\mathbb{R}$ is a function defined by:
    \begin{equation*}
 	  \widetilde{\phi}\big(x,y\big): =\phi\big(x\big)
    \end{equation*}
    \item The drift $\widetilde{\mu}^{\epsilon}_{Z}$ and dispersion $\widetilde{\sigma}^{\epsilon}_{Z}$ fields are defined, respectively, by 
    \begin{equation*}
 	\widetilde{\mu}^{\epsilon}_{Z}\big(z,u\big):=\begin{bmatrix}
 		\mu_{X}\big(z,u\big)\\
 		\frac{1}{\epsilon}\mu_{Y}\big(z\big)
 	\end{bmatrix}\quad\text{and}\quad\widetilde{\sigma}^{\epsilon}_{Z}\big(z,u\big):=\begin{bmatrix}
 		\sigma_{X}\big(z,u\big)\\
 		\frac{1}{\sqrt{\epsilon}}\sigma_{Y}\big(z\big)
 	   \end{bmatrix}.
    \end{equation*}
 \end{itemize}
 From items I and II, together with the definition of the fields $\widetilde{\mu}^{\epsilon}_{Z}$ and $\widetilde{\sigma}^{\epsilon}_{Z}$, we get from \cite{LionsSznitman}, \cite{PardouxRascanu}, and \cite{Pilipenko} that there exists a unique pair of processes $\big(Z^{\epsilon}_z(t),l^{\epsilon}_z(t)\big)_{t\geq0}$ which is a strong solution of the SDE \eqref{eq:auxiliary_multiscale_SDE_fast_slow_system}. Thus, there exists a unique triple of processes $\big(X^{\epsilon}_x(t),Y^{\epsilon}_y(t),l^{\epsilon}_{x,y}(t)\big)_{t\geq0}$ which is a strong solution of the SDE \eqref{subeq:multiscale_stochastic_dynamic_system_macroscale_variable_constraint}. We now define the cost functional for the dynamical system \eqref{eq:auxiliary_multiscale_SDE_fast_slow_system}:
 \begin{equation}
 	\label{Auxiliary_operation_cost_function}
 	J^{\beta,\epsilon}_z\big(u\big):=\mathbb{E}\Bigg[\int_{0}^{+\infty}e^{-\beta s}\widetilde{L}\big(Z^{\epsilon}_z(s),u(s)\big)\,ds+\int_{0}^{+\infty}e^{-\beta s}\widetilde{h}\big(Z^{\epsilon}_z(s)\big)\,dl^{\epsilon}_z(s)\Bigg]
 \end{equation}
 where $\widetilde{L}$ is the same operation cost we used in the multiscale problem and $\widetilde{h}$ is defined by $\widetilde{h}\big(x,y\big):=h\big(x\big)$ . From items III and IV, it follows from \cite[Lemma 2.2]{CalixtoCostaValle2025} that the functional \eqref{Auxiliary_operation_cost_function} is well defined and consequently the functional \eqref{eq:functional_cost_multiscale_operation}
 is also well defined.  In addition, items I, II, III and IV guarantee, by \cite[Lemma 2.6]{CalixtoCostaValle2025}, that the family $\big(z\mapsto J^{\beta,\epsilon}_z\big(u\big)\big)_{u\in\mathcal{U}}$ is equicontinuous and thus the application $\big((x,y)\mapsto J^{\beta,\epsilon}_{x,y}\big(u\big)\big)_{u\in\mathcal{U}}$ is also equicontinuous. Finally, we define the optimal value function  
 \begin{equation}
 	\label{eq:auxiliary_optimal_value_function}
 	v^{\beta,\epsilon}\big(z\big):=\inf_{u\in\mathcal{U}}J^{\beta,\epsilon}_{z}\big(u\big)
 \end{equation}
 and, by \cite[Lemma 2.7]{CalixtoCostaValle2025}, this function is continuous, which implies that the optimal value function \eqref{eq:multiscale_optimal_value_function} is also continuous.
 
 From items I, II, III and IV, we see that the DPP \cite[Theorem 2.8]{CalixtoCostaValle2025} holds for the optimal value function \eqref{eq:auxiliary_optimal_value_function}. This implies that this principle (rewritten as in \eqref{subeq:multiscale_dynamic_programming_principle}) also holds for the optimal value function \eqref{eq:multiscale_optimal_value_function}. Thus, together with the Itô formula, we heuristically obtain the following HJB equation
 \begin{subequations}
 	\label{subeq:auxiliary_HJB_Von_Neumann_boundary}
 	\begin{align}
 		\label{eq:auxiliary_HJB_equation}
 		\beta v^{\beta,\epsilon}\big(z\big)-\mathcal{\widetilde{H}}^{\epsilon}\Big(z,D_zv^{\beta,\epsilon}\big(z\big), D^2_zv^{\beta,\epsilon}\big(z\big)\Big)&=0\qquad\hspace{0.13cm} \forall z\in \widetilde{\mathbb{X}},\\
 		\label{eq:auxiliary_von_Neumann_boundary_condition}
 		\big\langle D_{z}v^{\beta,\epsilon}\big(z\big),D_z\widetilde{\phi}\big(z\big)\big\rangle&=\widetilde{h}\big(z\big)\quad \forall z\in \partial\widetilde{\mathbb{X}}
 	\end{align}
 \end{subequations}
 where the Hamiltonian $\mathcal{\widetilde{H}}^{\epsilon}$ is given by
 \begin{equation}
 	\label{eq:auxiliary_Hamiltonian}
 	\mathcal{\widetilde{H}}^{\epsilon}\Big(z,g_z,H_z\Big):=\min_{u\in \mathbb{U}}\Bigg\{\mathcal{G}_{Z}^{\epsilon,u}\big(z,g_z,H_z\big)+L\big(z,u\big)\Bigg\}.
 \end{equation}
 To finish rewriting the problem \eqref{subeq:multiscale_stochastic_dynamic_system_macroscale_variable_constraint}-\eqref{eq:multiscale_optimal_value_function} in the form \eqref{eq:auxiliary_multiscale_SDE_fast_slow_system}-\eqref{eq:auxiliary_optimal_value_function}, we need to show that the domain $\widetilde{\mathbb{X}}$ satisfies the same conditions as the domain $\mathbb{X}$. To see this, consider the following set
   \begin{equation*}
	\mathfrak{X}:=\Bigg\{(x,y)\in \mathbb{R}^{d_X}\times\mathbb{R}^{d_Y}:\widetilde{\phi}\big(x,y\big)\leqslant 0\Bigg\}.
\end{equation*}
  On the one hand, given $(x,y)\in\widetilde{\mathbb{X}}$, it follows that $0\geqslant\phi\big(x\big)=\widetilde{\phi}\big(x,y\big)$, which implies $(x,y)\in \mathfrak{X}$. On the other hand, if $(x,y)\in \mathfrak{X}$, then $y\in \mathbb{T}^{d_Y}$ and $\phi\big(x\big)\leqslant 0$, which implies $x\in\mathbb{X}$ and therefore $(x,y)\in\widetilde{\mathbb{X}}$.  Furthermore, by the definition of $\widetilde{\phi}$, this function is of class $\text{C}^2_{\text{b}}\big(\mathbb{R}^{d_X}\times\mathbb{T}^{d_Y}\big)$. Now, observe that
 \begin{equation*}
 	\partial\widetilde{\mathbb{X}}=\partial\big(\mathbb{X}\times\mathbb{T}^{d_Y}\big)=\big(\partial \mathbb{X}\times\mathbb{T}^{d_Y}\big)\cup \big(\mathbb{X}\times\partial \mathbb{T}^{d_Y}\big)=\partial \mathbb{X}\times\mathbb{T}^{d_Y}
 \end{equation*}
 because $\partial \mathbb{T}^{d_Y}=\emptyset$. Since, for all $x\in\partial \mathbb{X}$, we have $\phi\big(x\big)=0$, it follows that
 \begin{equation*}
 	\partial\widetilde{\mathbb{X}}=\Bigg\{(x,y)\in \mathbb{R}^{d_X}\times\mathbb{R}^{d_Y}:\widetilde{\phi}\big(x,y\big)=0\Bigg\}.
 \end{equation*}
 On the other hand,
 \begin{equation*}
 	\text{Int}\big(\widetilde{\mathbb{X}}\big)=\widetilde{\mathbb{X}}-\partial \widetilde{\mathbb{X}}=\mathbb{X}\times\mathbb{T}^{d_Y}-\partial \mathbb{X}\times \mathbb{T}^{d_Y}=\big(\mathbb{X}-\partial \mathbb{X}\big)\times\mathbb{T}^{d_Y}=\text{Int}\big(\mathbb{X}\big)\times\mathbb{T}^{d_Y}.
 \end{equation*}
 Since, for all $x\in\text{Int}\big(\mathbb{X}\big)$, we have $\phi\big(x\big)<0$, it follows that
 \begin{equation*}
 	\text{Int}\big(\widetilde{\mathbb{X}}\big)=\Bigg\{(x,y)\in \mathbb{R}^{d_X}\times\mathbb{R}^{d_Y}:\widetilde{\phi}\big(x,y\big)<0\Bigg\}.
 \end{equation*}
 
Finally, for all $z=(x,y)\in \partial\widetilde{\mathbb{X}}$ we have $\|D_z\widetilde{\phi}\big(z\big)\|=\|D_x\phi\big(x\big)\|=1$. Therefore, the set $\widetilde{\mathbb{X}}$ satisfies the same properties as the domain $\mathbb{X}$. Applying \cite[Theorems 2.11 and 2.12]{CalixtoCostaValle2025} to the formulation \eqref{subeq:auxiliary_HJB_Von_Neumann_boundary}, we obtain that the optimal value function \eqref{eq:auxiliary_optimal_value_function} is the unique viscosity solution of problem \eqref{subeq:auxiliary_HJB_Von_Neumann_boundary} and consequently the optimal value function \eqref{eq:multiscale_optimal_value_function} is the unique viscosity solution of problem \eqref{subeq:multiscale_HJB_Von_Neumann_boundary}. 
\section{Effective Stochastic Optimal Control}
  \label{subsec:effective_stochastic_optimal_control}
 In simple terms, the goal is to derive a new Hamiltonian for the multiscale phenomenon we are modeling, with the following characteristic: in the new Hamiltonian, the dynamics of the slow variable continue to be explicitly represented, while the dynamics of the fast variable are implicitly embedded. Therefore, the objective is to define a new problem, called the effective problem, given by:
  \begin{subequations}
  	\label{subeq:effective_HJB_Von_Neumann_boundary}
  	\begin{align}
  		\label{eq:effective_HJB_equation}
  		\beta\overline{v}^{\beta}\big(x\big)-\mathcal{\overline{H}}\Big(x,D_{x}\overline{v}^{\beta}\big(x\big),D^2_{x}\overline{v}^{\beta}\big(x\big)\Big)&=0\qquad\hspace{0.15cm}\forall x\in \mathbb{X},\\
  		\label{eq:effective_von_Neumann_boundary_condition}
  		\bigg\langle D_{x}\overline{v}^{\beta}\big(x\big),D_x\phi\big(x\big)\bigg\rangle&=h\big(x\big)\quad \forall x\in \partial \mathbb{X}.
  	\end{align}
  \end{subequations}
  The problems \eqref{subeq:multiscale_HJB_Von_Neumann_boundary} and \eqref{subeq:effective_HJB_Von_Neumann_boundary} are related in two key steps. The first step is to derive a formula that describes the effective Hamiltonian as a function of the multiscale Hamiltonian. The second step is to demonstrate that the family  $\big(v^{\beta,\epsilon}\big)_{\epsilon>0}$ converges uniformly in $\mathbb{X}\times\mathbb{T}^{d_Y}$ to the effective solution $\overline{v}^{\beta}$. 
  
  Our phenomenon of interest is modeled as a SOC problem rather than a second-order nonlinear PDE. In control theory, the HJB equation is primarily a tool for deriving controls within a specific class of Markov controls. Hence, the effective problem \eqref{subeq:effective_HJB_Von_Neumann_boundary} only becomes relevant if it can be ``disassembled'' back into an SOC problem. To achieve this, the effective Hamiltonian $\mathcal{\overline{H}}$ must possess the following form:
  \begin{equation}
  	\label{eq:explicit_effective_Hamiltonian}
  	\mathcal{\overline{H}}\bigg(x,g_x,H_x\bigg)=\min_{u\in \mathbb{U}}\Bigg\{\bigg\langle \overline{\mu}_{\overline{X}}\big(x\big)+\mu_{\text{SC}}\big(x,u\big),g_x\bigg\rangle+\frac{1}{2}\text{Tr}\Big(H_x\big[\big(\overline{\sigma}_{\overline{X}}+\sigma_{\text{SC}}\big)\big(\overline{\sigma}_{\overline{X}}+\sigma_{\text{SC}}\big)^{\top}\big]\big(x,u\big)\Big)+\overline{L}\big(x\big)+L_{\text{SC}}\big(x,u\big)\Bigg\}.
  \end{equation}
  When this representation is possible, we get an effective SOC problem that takes the following form:
  \begin{subequations}
  	\label{subeq:effective_stochastic_optimal_control_problem}
  	\begin{align}
  		\label{eq:effective_optimum_value_function}
  		\overline{v}^{\beta}\big(x\big):=&\inf_{u\in\mathcal{U}}\mathbb{E}\Bigg[\int_{0}^{+\infty}e^{-\beta s}\big[\overline{L}\big(\overline{X}_x(s)\big)+L_{\text{SC}}\big(\overline{X}_x(s),u(s)\big)\big]\,ds+\int_{0}^{+\infty}e^{-\beta s}h\big(\overline{X}_x(s)\big)\,d\bar{l}_x(s)\Bigg],\\
  		\label{eq:effective_stochastic_dynamic_system}
  		d\overline{X}_x(t)=&\bigg[\overline{\mu}_{\overline{X}}\big(\overline{X}_x(t)\big)+\mu_{\text{SC}}\big(\overline{X}_x(t),u(t)\big)\bigg]\,dt+\bigg[\overline{\sigma}_{\overline{X}}\big(\overline{X}_x(t)\big)+\sigma_{\text{SC}}\big(\overline{X}_x(t),u(t)\big)\bigg]\,dW(t)-D_x\phi\big(\overline{X}_x(t)\big)\,d\overline{l}_x(t).
  	\end{align}
  \end{subequations}
  
  Since the control problem \eqref{subeq:effective_stochastic_optimal_control_problem} operates on a single scale, it is not necessary to consider any form of decomposition of the effective drift and diffusion fields, nor of the effective running cost. Below, we present a list of assumptions, as introduced in \cite[Section 1 and Subsection 2.3]{CalixtoCostaValle2025}, under which a SOC problem with state constraints, control-dependent diffusion matrix, and a positive semidefinite diffusion matrix can be addressed.

  \begin{itemize}
  	\item[\(\textbf{H}_1\)] The drift \(\mu_{X}\) and the dispersion \(\sigma_{X}\) are continuous, bounded, and Lipschitz continuous with respect to the spatial variable, uniformly with respect to the control variable. Specifically, there exists a constant \(C > 0\) such that for every \(x, y \in \mathbb{R}^d\) and every \(u \in \mathbb{U}\), the following inequalities hold:
  	\[
  	\|\mu_{X}\big(x,u\big)-\mu_{X}\big(y,u\big)\|+\|\sigma_{X}\big(x,u\big)-\sigma_{X}\big(y,u\big)\| \leqslant C\|x-y\| \quad \text{and} \quad \|\mu_{X}\big(x,u\big)\| + \|\sigma_{X}\big(x,u\big)\| \leqslant C.
  	\]
  	\item[\(\textbf{H}_2\)] The function \(L: \mathbb{R}^d \times \mathbb{U} \to \mathbb{R}\) is continuous and bounded in \((x, u)\), and Lipschitz continuous in the spatial variable, uniformly in the control. That is, there exists a constant \(C > 0\) such that for every \(x, y \in \mathbb{R}^d\) and every \(u \in \mathbb{U}\),
  	\[
  	\big|L\big(x,u\big)-L\big(y,u\big)\big| \leqslant C\|x - y\| \quad \text{and} \quad \big|L\big(x, u\big)\big| \leqslant C.
  	\]
  	\item[\(\textbf{H}_3\)] The function \(h: \partial \mathbb{X} \to \mathbb{R}\) is continuous. 
  	\item[$\textbf{H}_4$] Structural assumptions on the drift, dispersion, and running cost functions that guarantee the validity of Lemma \ref{lm:Lipschitz_controls} for the function defined in \eqref{eq:function_theta_varphi}.
  \end{itemize}
  \subsection{Introduction to Cell-$t$-Problem}
  \label{subsec:cell_t_problem}
  In this subsection, we present an auxiliary PDE that allows us to obtain the effective Hamiltonian $\mathcal{\overline{H}}$. To do this, we adapt the theory presented in \cite[Chapters 2 and 3]{AlvarezBardi} and in \cite{BardiCesaroni,BardiCesaroniManca} to our case. Fixing $\bar{x}\in \mathbb{X}$ and taking $\epsilon=1$ in the SDE \eqref{eq:microscale_stochastic_dynamical_system}, we obtain the following stochastic dynamical system
  \begin{equation}
  	\label{eq:microscale_stochastic_dynamical_system_frozen_macro_variable}
  	dY_y(t)=\mu_{Y}\big(\bar{x},Y_y(t)\big)\,dt+\sigma_{Y}\big(\bar{x},Y_y(t)\big)\,dW(t).
  \end{equation} 
  Fixed $p:=\big(\bar{x},g_{\bar{x}},H_{\bar{x}}\big)\in \mathbb{X}\times\mathbb{R}^{d_X}\times\mathbb{S}^{d_X}\big(\mathbb{R}\big)$ (the parameters of the slow scale), consider the function $ \mathfrak{H}_p: \mathbb{T}^{d_Y}\to\mathbb{R}$ defined by 
  \begin{equation}
  	\label{eq:microscale_operation_cost}
  	\mathfrak{H}_p\big(y\big):=\min_{u\in \mathbb{U}}\Bigg\{\mathcal{G}_{X}^{u}\big(\bar{x},g_{\bar{x}},H_{\bar{x}}\big|y\big)+L\big(\bar{x},y,u\big)\Bigg\}.
  \end{equation}
  We now define the Hamiltonian $\mathcal{H}_p:\mathbb{T}^{d_Y}\times\mathbb{R}^{d_Y}\times\mathbb{S}^{d_Y}\big(\mathbb{R}\big)\to\mathbb{R}$ by
  \begin{equation}
  	\label{eq:hamiltonian_cell_t_problem}
  	\mathcal{H}_p\Big(y,g_y,H_y\Big):=\bigg\langle\mu_{Y}\big(\bar{x},y\big),g_y\bigg\rangle +\frac{1}{2} \text{Tr}\Big(H_y\big[\sigma_{Y}\sigma_{Y}^{\top}\big]\big(\bar{x},y\big)\Big)+\mathfrak{H}_p\big(y\big).
  \end{equation}
  Observe that the Hamiltonian \eqref{eq:hamiltonian_cell_t_problem} relates to the previous multiscale Hamiltonian $\mathcal{H}$ as follows
  \begin{equation*}
  	\mathcal{H}_p\bigg(y,g_y,H_y\bigg)=\mathcal{H}\bigg(\bar{x},y,g_{\bar{x}},g_y,H_{\bar{x}},H_y,0\bigg).
  \end{equation*}
  The intuition behind the above equality is that, with respect to the fast variable $\big(Y(t)\big)_{t\geq0}$, the slow variable $\big(X(t)\big)_{t\geq0}$ is ``frozen'', and so we can treat it as constant when looking only at the dynamics of the system on the fast scale. We call the following second-order parabolic linear PDE the \emph{Cell-t-Problem}:
  \begin{subequations}
  	\label{subeq:PDE_cell_t_problem}
  	\begin{align}
  		\label{eq:PDE_cell_t_problem}
  		-\partial_{t}w_p\big(t,y\big)+\mathcal{H}_p\Big(y,D_yw_p\big(t,y\big),D^2_yw_p\big(t,y\big)\Big)&=0\quad \forall (t,y)\in\mathbb{R}_{+}\times\mathbb{T}^{d_Y},\\
  		\label{eq:initial_condition_cell_t_problem}
  		w_p\big(0,y\big)&=0\quad \forall y\in \mathbb{T}^{d_Y}.
  	\end{align}
  \end{subequations} 
  From hypotheses I, II, and III it follows from \cite[Theorem 2]{IlyinKalashnikovOleynik} that the PDE \eqref{subeq:PDE_cell_t_problem} admits a unique classical solution given by 
  	\begin{equation}
  	\label{eq:cell_t_problem_solution}
  	w_p\big(t,y\big):=\int_{0}^{t}\mathbb{E}\big[\mathfrak{H}_p\big(Y_y(s)\big)\big]\,ds.
  \end{equation}
  
  A fundamental part of the theory presented in \cite{AlvarezBardi} is that all effective functions must be obtained as functions of the properties of the dynamical system \eqref{eq:microscale_stochastic_dynamical_system_frozen_macro_variable}. The next definition is an adaptation of the concept of \emph{Ergodic Uniqueness} presented in \cite[Chapter 2]{AlvarezBardi} to our case.
  \begin{definition}[Ergodic Uniqueness]
  	\label{df:ergodic_uniqueness_cell_t_problem}
  	The stochastic dynamical system \eqref{eq:microscale_stochastic_dynamical_system_frozen_macro_variable} is \emph{Uniquely Ergodic} if, for every $\varphi\in \emph{C}\big(\mathbb{T}^{d_Y}\big)$, there exists a constant $c_{\varphi}\in\mathbb{R}$ such that
  	\begin{equation*}
  		\lim_{t\to +\infty}\frac{1}{t}\int_{0}^{t}\mathbb{E}\big[\varphi\big(Y_y(s)\big)\big]\,ds=c_{\varphi}\quad \forall y\in \mathbb{T}^{d_Y}.
  	\end{equation*}
  \end{definition}
  \begin{definition}[Effective Hamiltonian]
  	\label{df:effective_hamiltonian_cell_t_problem}
  		Assume that the stochastic dynamical system \eqref{eq:microscale_stochastic_dynamical_system_frozen_macro_variable} is uniquely ergodic. Given $p:=\big(\bar{x},g_{\bar{x}},H_{\bar{x}}\big)\in \mathbb{X}\times\mathbb{R}^{d_X}\times\mathbb{S}^{d_X}\big(\mathbb{R}\big)$, consider $\mathfrak{H}_p: \mathbb{T}^{d_Y}\to\mathbb{R}$ the function defined in \eqref{eq:microscale_operation_cost}. We call $\mathcal{\overline{H}}:\mathbb{X}\times\mathbb{R}^{d_X}\times\mathbb{S}^{d_X}\big(\mathbb{R}\big)\to\mathbb{R}$ defined by 
  		\begin{equation}
  			\label{eq:uniquely_ergodic_effective_Hamiltonian}
  			\mathcal{\overline{H}}\Big(\bar{x},g_{\bar{x}},H_{\bar{x}}\Big):=	\lim_{t\to +\infty}\frac{1}{t}\int_{0}^{t}\mathbb{E}\big[\mathfrak{H}_p\big(Y_y(s)\big)\big]\,ds\quad \forall y\in\mathbb{T}^{d_Y}
  		\end{equation}
  		of \emph{Effective Hamiltonian}.
  \end{definition}
   In this case, the question is: how can we prove that a given system is uniquely ergodic according to Definition \ref{df:ergodic_uniqueness_cell_t_problem}? This answer will be given in the next subsection, where we discuss ergodic theory for Markov processes.
  \subsection{Ergodic theory of Markov Processes}
  \label{subsubsec:ergodic_theory_Markov_processes}
  The case we are interested in is when the dynamics of the fast scale does not depend on the control variable, i.e. when we have \eqref{eq:microscale_stochastic_dynamical_system_frozen_macro_variable}.    
  The advantage of not controlling the dynamics of the fast variable directly is that we can use a series of analytical tools that make it easier to study the ergodic behavior of this system. To see this, fixed $\bar{x}\in \mathbb{X}$, consider the dynamical system \eqref{eq:microscale_stochastic_dynamical_system_frozen_macro_variable}.
  As the drift $\mu_{Y}$ and dispersion $\sigma_{Y}$ are deterministic, we have from  \cite[Theorem 7.2.4]{Oksendal} that the SDE \eqref{eq:microscale_stochastic_dynamical_system_frozen_macro_variable} defines a Markov process. On the other hand, the infinitesimal generator of this SDE is given by 
  \begin{equation*}
  	\mathcal{L}^{\bar{x}}f\big(y\big):=\bigg\langle\mu_{Y}\big(\bar{x},y\big),D_yf\big(y\big)\bigg\rangle+\frac{1}{2}\text{Tr}\Big(\big[\sigma_{Y}\sigma^{\top}_{Y}\big]\big(\bar{x},y\big)D^2_yf\big(y\big)\Big)\quad \forall y\in\mathbb{T}^{d_Y}
  \end{equation*}
  where $f\in \text{C}^2\big(\mathbb{T}^{d_Y}\big)$. We can formally define the adjoint operator of $\mathcal{L}^{\bar{x}}$ (in the sense of $\text{L}^2$) as 
  \begin{equation*}
  	\mathcal{L}^{\bar{x},*}\rho\big(y\big):=-\sum_{i=1}^{d_Y}\partial_{y_i}\big(\mu_{Y,i}\big(\bar{x},y\big)\rho\big(y\big)\big)+\frac{1}{2}\sum_{i,j}^{d_Y}\partial^2_{y_i,y_j}\bigg(\big[\sigma_{Y}\sigma^{\top}_{Y}\big]_{ij}\big(\bar{x},y\big)\rho\big(y\big)\bigg)\quad \forall y\in\mathbb{T}^{d_Y}
  \end{equation*}
  where $\rho$ satisfies:
  \begin{equation*}
  	\inf_{y\in \mathbb{T}^{d_Y}}\rho\big(y\big)>0\quad\text{and}\quad \int_{\mathbb{T}^{d_Y}}\rho\big(y\big)dy=1.
  \end{equation*}
  With these objects, we define the so-called stationary \emph{Fokker-Planck} PDE
  \begin{equation}
  	\label{eq:fokker_planck_PDE}
  	\mathcal{L}^{\bar{x},*}\rho\big(y\big)=0\quad \forall y\in \mathbb{T}^{d_Y}.\\
  \end{equation}
  \begin{definition}[Ergodic Markov Processes]
  	\label{df:ergodic_markov_processes}
  	The Markov process $\big(Y(t)\big)_{t\geq0}$ generated by the SDE \eqref{eq:microscale_stochastic_dynamical_system_frozen_macro_variable} is \emph{Ergodic} if the PDE \eqref{eq:fokker_planck_PDE} admits a unique classical solution and for every $\varphi \in \emph{C}\big(\mathbb{T}^{d_Y}\big)$ the following limit holds
  	\begin{equation}
  		\label{eq:microscale_ergodic_limit}
  		\lim_{t\to+\infty}\frac{1}{t}\int_{0}^{t}\mathbb{E}\big[\varphi\big(Y_y(s)\big)\big]\,ds=\int_{\mathbb{T}^{d_Y}}\varphi\big(y\big)\rho_{\bar{x}}\big(y\big)\,dy\quad \forall y\in \mathbb{T}^{d_Y}.
  	\end{equation}
  \end{definition}
  \begin{theorem}[\protect{\cite[Theorem 6.16]{PavliotisStuart}}]
  	\label{tm:ergodicity_markov_process}
  	Assume that hypothesis II holds. Then the Markov process $\big(Y(t)\big)_{t\geq0}$ generated by the SDE \eqref{eq:microscale_stochastic_dynamical_system_frozen_macro_variable} is ergodic.
  \end{theorem}
  
  A relevant remark is that such results can be obtained in more general contexts than the torus $\mathbb{T}^{d_Y}$, for example, for general compact spaces. It is even possible to obtain results for non-compact spaces, but in this case, it is necessary to study the vanishing of the tail of invariant densities at infinity. See  \cite[Chapter 4]{Pavliotis} for some interesting examples in this context and also  \cite[Chapter 4]{Khasminskii} for more general results involving only non absolutely continuous invariant probability distributions.
  
  Now consider, the equations \eqref{eq:fokker_planck_PDE} and \eqref{eq:microscale_ergodic_limit}. First, note that if these equations hold, then the stochastic dynamical system \eqref{eq:microscale_stochastic_dynamical_system_frozen_macro_variable} is uniquely ergodic according to Definition \ref{df:ergodic_uniqueness_cell_t_problem} with constant $c_{\varphi}\in\mathbb{R}$ given by
  \begin{equation*}
  	c_{\varphi}=\int_{\mathbb{T}^{d_Y}}\varphi\big(y\big)\rho_{\bar{x}}\big(y\big)\,dy.
  \end{equation*}
  This implies, by the equation \eqref{eq:uniquely_ergodic_effective_Hamiltonian} (in the definition of the effective Hamiltonian), that
  \begin{equation}
  	\label{eq:effective_Hamiltonian}
  	\mathcal{\overline{H}}\Big(\bar{x},g_{\bar{x}},H_{\bar{x}}\Big)=\int_{\mathbb{T}^{d_Y}}\mathfrak{H}_p\big(y\big)\rho_{\bar{x}}\big(y\big)\,dy.
  \end{equation}
  We want to rewrite the effective Hamiltonian \eqref{eq:effective_Hamiltonian} in the form \eqref{eq:explicit_effective_Hamiltonian}. To do this, we rewrite appropriately $\mathfrak{H}_p$ defined in \eqref{eq:microscale_operation_cost},
  \begin{align*}
  	\mathcal{G}_{X}^{u}\big(\bar{x},g_{\bar{x}},H_{\bar{x}}\big|y\big)=\bigg\langle \mu_{\text{SC}}\big(\bar{x},u\big),g_{\bar{x}}\bigg\rangle+\bigg\langle\mu_{\text{SF}}\big(\bar{x},y\big),g_{\bar{x}}\bigg\rangle+&\frac{1}{2}\text{Tr}\Big(H_{\bar{x}}\big[\sigma_{\text{SC}}\sigma^{\top}_{\text{SC}}\big]\big(\bar{x},u\big)\Big)+\frac{1}{2}\text{Tr}\Big(H_{\bar{x}}\big[\sigma_{\text{SF}}\sigma^{\top}_{\text{SF}}\big]\big(\bar{x},y\big)\Big)\\
  	+&\frac{1}{2}\text{Tr}\Big(H_{\bar{x}}\big[\sigma_{\text{SC}}\big(\bar{x},u\big)\sigma^{\top}_{\text{SF}}\big(\bar{x},y\big)+\sigma_{\text{SF}}\big(\bar{x},y\big)\sigma^{\top}_{\text{SC}}\big(\bar{x},u\big)\big]\Big).
  \end{align*}
  By equality
  \eqref{eq:noncorrelation_between_parts_multiscale_dispersion} in hypothesis I of Section~\S\ref{sec:multiscale_stochastic_optimal_control}, it follows that
  \begin{equation*}
  	\mathcal{G}_{X}^{u}\big(\bar{x},g_{\bar{x}},H_{\bar{x}}\big|y\big)=\bigg\langle \mu_{\text{SC}}\big(\bar{x},u\big),g_{\bar{x}}\bigg\rangle+\bigg\langle\mu_{\text{SF}}\big(\bar{x},y\big),g_{\bar{x}}\bigg\rangle+\frac{1}{2}\text{Tr}\Big(H_{\bar{x}}\big[\sigma_{\text{SC}}\sigma^{\top}_{\text{SC}}\big]\big(\bar{x},u\big)\Big)+\frac{1}{2}\text{Tr}\Big(H_{\bar{x}}\big[\sigma_{\text{SF}}\sigma^{\top}_{\text{SF}}\big]\big(\bar{x},y\big)\Big).	
  \end{equation*}
  Thus, the function $\mathfrak{H}_p$ can be rewritten as,
  \begin{align}
  	\nonumber
  	\mathfrak{H}_p\big(y\big)=\min_{u\in \mathbb{U}}\Bigg\{&\bigg\langle \mu_{\text{SC}}\big(\bar{x},u\big),g_{\bar{x}}\bigg\rangle+\frac{1}{2}\text{Tr}\Big(H_{\bar{x}}\big[\sigma_{\text{SC}}\sigma^{\top}_{\text{SC}}\big]\big(\bar{x},u\big)\Big)+L_{\text{SC}}\big(\bar{x},u\big)\Bigg\}\\
  	\label{eq:microscale_operation_cost_02}
  	+&\bigg\langle\mu_{\text{SF}}\big(\bar{x},y\big),g_{\bar{x}}\bigg\rangle+\frac{1}{2}\text{Tr}\Big(H_{\bar{x}}\big[\sigma_{\text{SF}}\sigma^{\top}_{\text{SF}}\big]\big(\bar{x},y\big)\Big)+L_{\text{SF}}\big(\bar{x},y\big).
  \end{align}
  Combining \eqref{eq:effective_Hamiltonian} with \eqref{eq:microscale_operation_cost_02}, it follows that
  \begin{align*}	
  	\mathcal{\overline{H}}\Big(\bar{x},g_{\bar{x}},H_{\bar{x}}\Big)=&\min_{u\in \mathbb{U}}\Bigg\{\bigg\langle \mu_{\text{SC}}\big(\bar{x},u\big),g_{\bar{x}}\bigg\rangle+\frac{1}{2}\text{Tr}\Big(H_{\bar{x}}\big[\sigma_{\text{SC}}\sigma^{\top}_{\text{SC}}\big]\big(\bar{x},u\big)\Big)+L_{\text{SC}}\big(\bar{x},u\big)\Bigg\}\\
  	+&\Bigg\langle\int_{\mathbb{T}^{d_Y}}\mu_{\text{SF}}\big(\bar{x},y\big)\rho_{\bar{x}}\big(y\big)dy,g_{\bar{x}}\Bigg\rangle+\frac{1}{2}\text{Tr}\Big(H_{\bar{x}}\int_{\mathbb{T}^{d_Y}}\big[\sigma_{\text{SF}}\sigma^{\top}_{\text{SF}}\big]\big(\bar{x},y\big)\rho_{\bar{x}}\big(y\big)dy\Big)+\int_{\mathbb{T}^{d_Y}}L_{\text{SF}}\big(\bar{x},y\big)\rho_{\bar{x}}\big(y\big)dy.
  \end{align*}
  Defining the effective functions:
  \begin{subequations}
  	\label{subeq:effective_functions}
  	\begin{align}
  		\label{eq:effective_drift}
  		\overline{\mu}_{\overline{X}}\big(\bar{x}\big):=&\int_{\mathbb{T}^{d_Y}}\mu_{\text{SF}}\big(\bar{x},y\big)\rho_{\bar{x}}\big(y\big)dy,\\
  		\label{eq:effective_diffusion}	
  		\overline{\sigma}_{\overline{X}}\big(\bar{x}\big)\overline{\sigma}_{\overline{X}}\big(\bar{x}\big)^{\top}=&\int_{\mathbb{T}^{d_Y}}\big[\sigma_{\text{SF}}\sigma^{\top}_{\text{SF}}\big]\big(\bar{x},y\big)\rho_{\bar{x}}\big(y\big)dy,\\
  		\label{eq:effective_operation_cost}		
  		\overline{L}\big(\bar{x}\big):=&\int_{\mathbb{T}^{d_Y}}L_{\text{SF}}\big(\bar{x},y\big)\rho_{\bar{x}}\big(y\big)dy,
  	\end{align}
  \end{subequations}
  where, in the equation \eqref{eq:effective_diffusion} we are assuming that there is a matrix $\overline{\sigma}_{\overline{X}}:\mathbb{X}\to \mathbb{M}^{d_X\times d_W}$ such that the equality in \eqref{eq:effective_diffusion} holds. Using these functions, we obtain the effective Hamiltonian presented in the formula  \eqref{eq:explicit_effective_Hamiltonian}. 
  
  To apply the results obtained in \cite[Section 2]{CalixtoCostaValle2025} to problem \eqref{subeq:effective_stochastic_optimal_control_problem}, we need to ensure that the effective functions $\overline{\mu}_{\overline{X}}$, $\overline{\sigma}_{\overline{X}}$ and $\overline{L}$ fulfill hypotheses $\textbf{H}_1$ and $\textbf{H}_2$. Note that it is sufficient to show that the functions defined in \eqref{subeq:effective_functions} belong to the space $\text{Lip}\big(\mathbb{X}\big)$. The simple case in which this occurs is when the drift $\mu_{Y}$ and dispersion $\sigma_{Y}$
  do not depend on the slow variable. In fact, in this case the invariant density $\rho$ does not depend on the parameter $\bar{x}$.
  
  When the invariant density $\rho$ depends on the parameter $\bar{x}$ it is necessary to investigate under what conditions this dependence allows Lipschitz continuity of the functions defined in \eqref{subeq:effective_functions}. We claim that it is sufficient for the application $D_x\rho$ to exist and be bounded in $\mathbb{X}\times\mathbb{T}^{d_Y}$. In fact, since $D_x\rho$ is bounded in $\mathbb{X}\times\mathbb{T}^{d_Y}$, there exists a constant $C>0$ such that 
  \begin{equation*}
  	\big|\rho\big(x_1,y\big)-\rho\big(x_2,y\big)\big|\leqslant C\|x_1-x_2\|\quad \forall y\in \mathbb{T}^{d_Y}.
  \end{equation*}
  In other words, the density $\rho$ is Lipschitz continuous with respect to the parameter. 
    
  On the other hand, since $D_x\rho$ exists for all $(x,y)\in \mathbb{X}\times \mathbb{T}^{d_Y}$, it follows that $\rho:\mathbb{X}\times \mathbb{T}^{d_Y}\to \mathbb{R}_{+}$ is bounded. So we have a continuous and bounded Lipschitz function. Since the fields $\mu_{\text{SF}}$, $\sigma_{\text{SF}}$ and $L_{\text{SF}}$ are also Lipschitz continuous and bounded, we obtain that the applications:
  \begin{equation*}
  	x\mapsto \mu_{\text{SF}}\big(x,y\big)\rho_x\big(y\big),\quad
  	x\mapsto \sigma_{\text{SF}}\big(x,y\big)\rho_x\big(y\big)\quad\text{and}\quad
  	x\mapsto L_{\text{SF}}\big(x,y\big)\rho_x\big(y\big)
  \end{equation*}
   are, uniformly in $y$, Lipschitz continuous.  This implies, for example, for the function $\overline{\mu}_{\overline{X}}:\mathbb{X}\to \mathbb{R}^{d_X}$ that
   \begin{align*}
   	\|\overline{\mu}_{\overline{X}}\big(x_1\big)-\overline{\mu}_{\overline{X}}\big(x_2\big)\|^2&\leqslant \int_{\mathbb{T}^{d_Y}}\|\mu_{\text{SF}}\big(x_1,y\big)\rho_{x_1}\big(y\big)-\mu_{\text{SF}}\big(x_2,y\big)\rho_{x_2}\big(y\big)\|^2\,dy\leqslant C^2\lambda\big(\mathbb{T}^{d_Y}\big)\|x_1-x_2\|^2.
   \end{align*}
   That is,
   \begin{equation*}
   	\|\overline{\mu}_{\overline{X}}\big(x_1\big)-\overline{\mu}_{\overline{X}}\big(x_2\big)\|\leqslant C\sqrt{\lambda\big(\mathbb{T}^{d_Y}\big)}\|x_1-x_2\|
   \end{equation*}
   where $\lambda:\mathbb{T}^{d_Y}\to \mathbb{R}_{+}$ is the Lebesgue measure.  
    
   Before proceeding to demonstrate the main result of this subsection, we will make the following extra hypothesis.
   \begin{itemize}
   	\item[VI] The drift $\mu_{Y}$ and dispersion $\sigma_{Y}$  of the fast variable belong to the space $\text{C}^{1,m+2}\big(\mathbb{X}\times\mathbb{T}^{d_Y}\big)$ with $m>d_Y/2$ and are such that the following inequality holds
   	\begin{equation}
   		\label{eq:maximum_principle_condition}
   		\sum_{i,j}^{d_Y}\frac{1}{2}\partial^2_{y_i,y_j}\big[\sigma_{Y}\sigma^{T}_{Y}\big]_{ij}\big(x,y\big)-\sum_{i=1}^{d_Y}\partial_{y_i}\mu_{Y,i}\big(x,y\big)\leqslant 0\quad\forall (x,y)\in \mathbb{X}\times\mathbb{T}^{d_Y}.
   	\end{equation}
   	In addition, we consider that the parameter $p\in[1,+\infty)$ (of the spaces $\text{L}^{\text{p}}$) is $p>d_Y$.
   \end{itemize}
   \begin{lemma}[Differentiability of Fokker-Plack \eqref{eq:fokker_planck_PDE} with respect to parameters -- \protect{Appendix \ref{sec:appendix_A}}]
   	\label{lm:differentiability_fokker_plack_parameters}
   	Assume that hypotheses II and VI hold. Then the classical solution of the PDE \eqref{eq:fokker_planck_PDE} is differentiable with respect to the slow scale parameter.
   \end{lemma}
  \begin{theorem}[Effective Optimal Value function as Viscosity Solution]
  	\label{tm:effective_optimal_value_function_viscosity_solution} The effective optimal value function \eqref{eq:effective_optimum_value_function} is the unique viscosity solution of the effective HJB equation \eqref{subeq:effective_HJB_Von_Neumann_boundary}.
  \end{theorem}
  \begin{proof}
  		The functions, defined in \eqref{subeq:effective_functions}, belong to the space $\text{Lip}\big(\mathbb{X}\big)$. In this case, the drift $\overline{\mu}_{\overline{X}}+\mu_{\text{SC}}$ and dispersion $\overline{\sigma}_{\overline{X}}+\sigma_{\text{SC}}$ satisfy hypothese $\textbf{H}_1$, which implies that there exists a unique pair of processes $\big(\overline{X}(t),\bar{l}(t)\big)_{t\geq0}$ which is a strong solution of the effective SDE \eqref{eq:effective_stochastic_dynamic_system}.  On the other hand, the effective operation cost $\overline{L}+L_{\text{SC}}$ satisfies hypothese $\textbf{H}_2$. Thus, by \cite[Lemma 2.6]{CalixtoCostaValle2025}, the family $\big(\bar{x}\mapsto \overline{J}_{\bar{x}}^{\beta}\big(\overline{u}\big)\big)_{\overline{u}\in\mathcal{U}}$ is equicontinuous. Furthermore, the effective optimal value function \eqref{eq:effective_optimum_value_function} is continuous by \cite[Theorem 2.7]{CalixtoCostaValle2025}. Finally, we have that the functions $\overline{\mu}_{\overline{X}}+\mu_{\text{SC}}$, $\overline{\sigma}_{\overline{X}}+\sigma_{\text{SC}}$ and $\overline{L}+L_{\text{SC}}$ fulfill hypothese $\textbf{H}_4$, therefore, it follows by \cite[Theorem 2.11]{CalixtoCostaValle2025}, that the effective optimal value function \eqref{eq:effective_optimum_value_function} is a viscosity solution of the effective HJB equation \eqref{subeq:effective_HJB_Von_Neumann_boundary} and, by \cite[Theorem 2.12]{CalixtoCostaValle2025}, this solution is unique.
  \end{proof}
  \section{Convergence of the Multiscale Optimal Value Function to the Effective Optimal Value Function}
  \label{sec:convergence_multiscale_optimal_value_function_effective_optimal_value_function}
  In this section, we prove the main result of the paper, namely that the family $\big(v^{\beta,\epsilon}\big)_{\epsilon>0}$ of multiscale optimal value functions converges uniformly to the effective optimal value function $\bar{v}^{\beta}$. To do this, we need some auxiliary results, starting with the following lemma.
  
  \begin{lemma}[Multiscale Optimal Value functions are equibounded]
  	\label{lm:multiscale_optimal_value_function_equibounded}
  	Consider items III and IV. Then the family $\big(v^{\beta,\epsilon}\big)_{\epsilon>0}$ of multiscale optimal value functions is equibounded.
  \end{lemma}
  \begin{proof}
  	From items III and IV, there exists a constant $C>0$ such that:
  	\begin{subequations}
  		\begin{align}
  			\label{eq:auxiliary_inequality_20}
  			\big|L\big(x,y,u\big)\big|&<C\hspace{0.15cm}\forall (x,y,u)\in \mathbb{X}\times\mathbb{T}^{d_Y}\times \mathbb{U},\\
  			\label{eq:auxiliary_inequality_21}
  			\big|h\big(x\big)\big|&<C\hspace{0.15cm}\forall x\in \partial\mathbb{X}.
  		\end{align}
  	\end{subequations}
  	On the other hand, given $0<\epsilon\ll 1$, it follows from Definition \eqref{eq:multiscale_optimal_value_function} that
  	\begin{equation*}
  		\big|v^{\beta,\epsilon}\big(x,y\big)\big|\leqslant\mathbb{E}\Bigg[\int_{0}^{+\infty}e^{-\beta s}\big|L\big(X^{\epsilon}_x(s),Y^{\epsilon}_y(s),u(s)\big)\big|\,ds\Bigg]+\mathbb{E}\Bigg[\bigg|\int_{0}^{+\infty}e^{-\beta s}h\big(X^{\epsilon}_x(s)\big)\,dl^{\epsilon}_{x,y}(s)\bigg|\Bigg]\quad \forall u\in\mathcal{U}.
  	\end{equation*}
  		Let's analyze the terms on the right-hand side of the above inequality. From the inequality \eqref{eq:auxiliary_inequality_20}, we have for the first term that
  		\begin{equation}
  			\label{eq:auxiliary_inequality_22}
  			\mathbb{E}\Bigg[\int_{0}^{+\infty}e^{-\beta s}\big|L\big(X^{\epsilon}_x(s),Y^{\epsilon}_y(s),u(s)\big)\big|\,ds\Bigg]\\
  			\leqslant\int_{0}^{+\infty}Ce^{-\beta s}ds\leqslant \frac{C}{\beta}.
  		\end{equation}
  	    Now, from the inequality \eqref{eq:auxiliary_inequality_21} and \cite[Inequality 15]{CalixtoCostaValle2025}, we have, for the second term, that
  	    \begin{equation*}
  	    	\mathbb{E}\Bigg[\bigg|\int_{0}^{+\infty}e^{-\beta s}h\big(X^{\epsilon}_x(s)\big)\,dl^{\epsilon}_{x,y}(s)\bigg|^2\Bigg]<\gamma
  	    \end{equation*}
  	    where $\gamma>0$ depends on the constants $\beta$ and $C$ as well as on the Lipschitz constant $N>0$ of the drift and dispersion of the slow variable. Using Cauchy-Schwarz's inequality, we get
  	    \begin{equation}
  	    \label{eq:auxiliary_inequality_23}
  	    	\mathbb{E}\Bigg[\bigg|\int_{0}^{+\infty}e^{-\beta s}h\big(X^{\epsilon}_x(s)\big)\,dl^{\epsilon}_{x,y}(s)\bigg|\Bigg]<\sqrt{\gamma}.
  	    \end{equation}
  	    By combining the inequalities \eqref{eq:auxiliary_inequality_22} and \eqref{eq:auxiliary_inequality_23}, we conclude that
  	    \begin{equation*}
  	    	\big|v^{\beta,\epsilon}\big(x,y\big)\big|\leqslant\frac{C}{\beta}+\sqrt{\gamma}\quad \forall(x,y)\in \mathbb{X}\times\mathbb{T}^{d_Y}\quad \text{and}\quad \forall \epsilon>0.
  	    \end{equation*}
  	    Therefore, the family $\big(v^{\beta,\epsilon}\big)_{\epsilon>0}$ is equibounded.
  \end{proof}
  \begin{theorem}[Uniform Convergence of $\big(v^{\beta,\epsilon}\big)_{\epsilon>0}$ to $\overline{v}^{\beta}$]
  	\label{tm:convergence_function_multiscale_value_effective}
  	Assume that items I through VI hold. Then the family $\big(v^{\beta,\epsilon}\big)_{\epsilon>0}$ of multiscale optimal value functions converges uniformly to the effective optimal value function $\overline{v}^{\beta}$.  
  \end{theorem}
  To demonstrate this result, we follow more closely part of the strategy adopted in \cite{BardiCesaroni}, \cite{BardiCesaroniManca} and \cite{AlvarezBardi01}. As well as some of the strategies presented in \cite[Chapter 2]{AlvarezBardi} and \cite[Chapter 7]{FlemingSoner}. The main tool to be used is the relaxed semi-limits:
  \begin{subequations}
  	\label{subeq:relaxed_semilimits}
  	\begin{align}
  		\label{eq:lower_relaxed_semilimits}
  		v_{\text{inf}}^{\beta}\big(x_0\big)&:=\liminf_{\epsilon\to 0, x\to x_0}\inf_{y\in \mathbb{T}^{d_Y}}v^{\beta,\epsilon}\big(x,y\big),\\
  		\label{eq:upper_relaxed_semilimits}
  		v_{\text{sup}}^{\beta}\big(x_0\big)&:=\limsup_{\epsilon\to 0, x\to x_0}\sup_{y\in\mathbb{T}^{d_Y}}v^{\beta,\epsilon}\big(x,y\big).
  	\end{align}
  \end{subequations}
  These semilimits are finite thanks to Lemma \ref{lm:multiscale_optimal_value_function_equibounded}. The strategy adopted consists of demonstrating that the semilimits \eqref{eq:lower_relaxed_semilimits} and \eqref{eq:upper_relaxed_semilimits} are, respectively, viscosity super- and subsolutions of the PDE \eqref{subeq:effective_HJB_Von_Neumann_boundary}. As the inequality
  \begin{equation*}
  	v_{\text{inf}}^{\beta}\big(x_0\big)=\liminf_{\epsilon\to 0, x\to x_0}\inf_{y\in \mathbb{T}^{d_Y}}v^{\beta,\epsilon}\big(x,y\big)\leqslant \limsup_{\epsilon\to 0, x\to x_0}\sup_{y\in\mathbb{T}^{d_Y}}v^{\beta,\epsilon}\big(x,y\big)=v_{\text{sup}}^{\beta}\big(x_0\big)
  \end{equation*}
  always holds, we apply \cite[Theorem 2.12]{CalixtoCostaValle2025} (which encapsulates the comparison principle), in order to obtain the opposite inequality. In this case, the function $\widetilde{v}^{\beta}:\mathbb{X}\to\mathbb{R}$ defined by 
  \begin{equation*}
  	\widetilde{v}^{\beta}\big(x_0\big):=v_{\text{inf}}^{\beta}\big(x_0\big)=v_{\text{sup}}^{\beta}\big(x_0\big)
  \end{equation*}
  is the unique viscosity solution of the PDE \eqref{subeq:effective_HJB_Von_Neumann_boundary}. On the other hand, in Theorem \ref{tm:effective_optimal_value_function_viscosity_solution}, we prove that the effective optimal value function \eqref{eq:effective_optimum_value_function} is the unique viscosity solution of the PDE \eqref{subeq:effective_HJB_Von_Neumann_boundary}. Therefore, we obtain that $\overline{v}^{\beta}=\widetilde{v}^{\beta}$.  Finally, since the family $\big(v^{\beta,\epsilon}\big)_{\epsilon>0}$ is equibounded by Lemma \ref{lm:multiscale_optimal_value_function_equibounded}, it follows from \cite[Lemma 1.9 of Chapter 5]{BardiDolcetta}, that the family $\big(v^{\beta,\epsilon}\big)_{\epsilon>0}$ converges uniformly to $\overline{v}^{\beta}$.
  \begin{proof}
  		We prove the result for the points $x\in \partial\mathbb{X}$, because the proof for interior points is analogous to that done in the literature for unbounded domains, see, for example, \cite[Theorem 5]{BardiCesaroniManca} or \cite[Theorem 3.2]{BardiCesaroni}. We will divide the proof into several steps.
  		\begin{itemize}
  			\item[Step 1] (Multiscale test function). In the first step, we construct, from the test function of the effective case, a test function that can be used in the multiscale case. We will do this using a method known as \emph{Pertubation of the Test Function}.
  			
  			Given $x_0\in \partial\mathbb{X}$, consider a test function $\varphi\in \text{C}^{\infty}\big(\mathbb{X}\big)$ such that
  			\begin{equation*}
  				\max_{x\in \overline{B\left(x_0,\gamma_{x_0}\right)}\cap \mathbb{X}}\big(v^{\beta}_{\text{sup}}\big(x\big)-\varphi\big(x\big)\big)=v^{\beta}_{\text{sup}}\big(x_0\big)-\varphi\big(x_0\big)=0
  			\end{equation*}
  			for some $\gamma_{x_0}>0$, with the maximum being strict. We want to prove that
  			\begin{equation}
  				\label{eq:auxiliary_effective_boundary_condition}
  				\overline{F}^{\beta}\Big(x_0,\varphi\big(x_0\big),D_x\varphi\big(x_0\big),D^2_x\varphi\big(x_0\big)\Big)\wedge \overline{\Gamma}\Big(x_0,D_x\varphi\big(x_0\big)\Big)\leqslant0.
  			\end{equation}
  			To do this, set $p:=\big(x_0,D_x\varphi\big(x_0\big),D^2_x\varphi \big(x_0\big)\big)$,
  			items I, II and III imply, by  \cite[Theorem 2]{IlyinKalashnikovOleynik}, that there exists $w_p\in \text{C}^{1,2}\big(\mathbb{R}_{+}\times\mathbb{T}^{d_Y}\big)$, the unique classical solution of the PDE \eqref{subeq:PDE_cell_t_problem}.
  			Given $\varsigma>0$, there exists $t_0>0$ (by the definition of $\mathcal{\overline{H}}$), such that 
  		   \begin{equation}
  				\label{eq:auxiliary_inequality_18}
  				\Bigg|\frac{1}{t_0}w_p\big(t_0,y\big)-\mathcal{\overline{H}}\Big(x_0,D_x\varphi\big(x_0\big),D^2_x\varphi\big(x_0\big)\Big)\Bigg|<\varsigma\quad \forall y\in \mathbb{T}^{d_Y}.
  			\end{equation}
  			Given $\epsilon>0$, consider $\varphi^{\epsilon}:\mathbb{X}\times\mathbb{T}^{d_Y}\to \mathbb{R}$ defined by
  			\begin{equation}
  				\label{eq:perturbed_test_function}
  				\varphi^{\epsilon}\big(x,y\big):=\varphi\big(x\big)+ \frac{\epsilon}{t_0}\int_{0}^{t_0}w_p\big(s,y\big)\,ds
  			\end{equation}
  			a perturbation of the test function $\varphi$. We want to show that \eqref{eq:perturbed_test_function} is a multiscale test function that we can use to work with the viscosity solution $v^{\beta,\epsilon}$. To do this, we have to prove the following:
  			\begin{itemize}
  				\item There exists sequences $\epsilon_n\downarrow 0$, $x^{t_0}_n\to x_0$ and $y^{t_0}_n\to y^{t_0}$ such that the pair $\big(x^{t_0}_n, y^{t_0}_n\big)$ is a local maximum of the application $(x,y)\mapsto\big(v^{\beta,\epsilon_n}-\varphi^{\epsilon_n}\big)\big(x,y\big)$ and
  				\begin{equation*}
  					\lim_{n\to \infty}\big(v^{\beta,\epsilon_n}-\varphi^{\epsilon_n}\big)\big(x^{t_0}_n,y^{t_0}_n\big)=v^{\beta}_{\text{sup}}\big(x_0\big)-\varphi\big(x_0\big)=0.
  				\end{equation*}
  			\end{itemize}
  			First of all, note that the family $\big(\varphi^{\epsilon}\big)_{\epsilon>0}$ converges, uniformly in $\mathbb{X}\times\mathbb{T}^{d_Y}$, to $\varphi$ and thus,
  		    \begin{equation*}
  			\limsup_{\epsilon\to 0, x\to x_0}\sup_{y\in\mathbb{T}^{d_Y}}\big(v^{\beta,\epsilon}-	\varphi^{\epsilon}\big)\big(x,y\big)=v^{\beta}_{\text{sup}}\big(x_0\big)-\varphi\big(x_0\big)=0.
  			\end{equation*}
  			We claim that: there are sequences $\epsilon_n\downarrow 0$ and $x^{t_0}_{\epsilon_n}\in \overline{B\left(x_0,\gamma_{x_0}\right)}\cap \mathbb{X}$ such that: 
  			\begin{itemize}
  				\item[1] $x^{t_0}_{\epsilon_n}\to x_0$ when $n\to +\infty$,
  				\item[2] $\sup_{y\in\mathbb{T}^{d_Y}}\big(v^{\beta,\epsilon_n}-\varphi^{\epsilon_n}\big)\big(x^{t_0}_{\epsilon_n},y\big)\to v^{\beta}_{\text{sup}}\big(x_0\big)-\varphi\big(x_0\big)$ when $n\to+\infty$,
  				\item[3] $x^{t_0}_{\epsilon_n}$ is a maximum of the application $x\mapsto \sup_{y\in\mathbb{T}^{d_Y}}\big(v^{\beta,\epsilon_n}-\varphi^{\epsilon_n}\big)\big(x,y\big)$ in $\overline{B\left(x_0,\gamma_{x_0}\right)}\cap\mathbb{X}$.
  			\end{itemize}
  			In fact, consider $x^{t_0}_{\epsilon}\in \overline{B\left(x_0,\gamma_{x_0}\right)}\cap\mathbb{X}$ such that
  			\begin{equation*}
  				\max_{x\in\overline{B\left(x_0,\gamma_{x_0}\right)}\cap\mathbb{X}}\sup_{y\in\mathbb{T}^{d_Y}}\big(v^{\beta,\epsilon}-\varphi^{\epsilon}\big)\big(x,y\big)=\sup_{y\in\mathbb{T}^{d_Y}}\big(v^{\beta,\epsilon}-\varphi^{\epsilon}\big)\big(x^{t_0}_{\epsilon},y\big).
  			\end{equation*}
  			By definition of $\limsup$, there exists sequences $\epsilon_n\downarrow 0$ and $\widetilde{x}^{t_0}_{\epsilon_n}\to x_0$ such that 
  			\begin{equation*}
  				\lim_{n\to \infty}\sup_{y\in\mathbb{T}^{d_Y}}\big(v^{\beta,\epsilon_n}-\varphi^{\epsilon_n}\big)\big(\widetilde{x}^{t_0}_{\epsilon_n},y\big)=v^{\beta}_{\text{sup}}\big(x_0\big)-\varphi\big(x_0\big)=0.
  			\end{equation*}
  			On the other hand, since $\big(x^{t_0}_{\epsilon_n}\big)_{n\geqslant0}$ is a sequence in $\overline{B\left(x_0,\gamma_{x_0}\right)}\cap\mathbb{X}$, there exists a subsequence $\big(x^{t_0}_{\epsilon_{n_i}}\big)_{i\geqslant0}$ and a point $\widetilde{x}^{t_0}\in \overline{B\left(x_0, \gamma_{x_0}\right)}\cap\mathbb{X}$ such that $\lim_{i\to+\infty}x^{t_0}_{\epsilon_{n_i}}=\widetilde{x}^{t_0}$. This implies that
  			\begin{align*}
  				0=v^{\beta}_{\text{sup}}\big(x_0\big)-\varphi\big(x_0\big)&=\lim_{i\to+\infty}\sup_{y\in\mathbb{T}^{d_Y}}\big(v^{\beta,\epsilon_{n_i}}-\varphi^{\epsilon_{n_i}}\big)\big(\widetilde{x}^{t_0}_{\epsilon_{n_i}},y\big)\\
  				&\leqslant\liminf_{i\to+\infty}\sup_{y\in\mathbb{T}^{d_Y}}\big(v^{\beta,\epsilon_{n_i}}-\varphi^{\epsilon_{n_i}}\big)\big(x^{t_0}_{\epsilon_{n_i}},y\big)\\
  				&\leqslant\limsup_{i\to+\infty}\sup_{y\in\mathbb{T}^{d_Y}}\big(v^{\beta,\epsilon_{n_i}}-\varphi^{\epsilon_{n_i}}\big)\big(x^{t_0}_{\epsilon_{n_i}},y\big)=v^{\beta}_{\text{sup}}\big(\widetilde{x}^{t_0}\big)-\varphi\big(\widetilde{x}^{t_0}\big).
  			\end{align*}
  			That is,
  			\begin{equation*}
  				0=v^{\beta}_{\text{sup}}\big(x_0\big)-\varphi\big(x_0\big)\leqslant v^{\beta}_{\text{sup}}\big(\widetilde{x}^{t_0}\big)-\varphi\big(\widetilde{x}^{t_0}\big).
  			\end{equation*}
  			As $x_0$ is a strict maximum point of the application $x\mapsto \big(v^{\beta}_{\text{sup}}-\varphi\big)\big(x\big)$, we get $\widetilde{x}^{t_0}=x_0$ and therefore
  			\begin{equation*}
  				\lim_{i\to+\infty}\sup_{y\in\mathbb{T}^{d_Y}}\big(v^{\beta,\epsilon_{n_i}}-\varphi^{\epsilon_{n_i}}\big)\big(x^{t_0}_{\epsilon_{n_i}},y\big)=v^{\beta}_{\text{sup}}\big(x_0\big)-\varphi\big(x_0\big)=0.
  			\end{equation*}
  			Thus, we have established items 1, 2 and 3.
  			
  			For notational reasons, we will continue to write $\epsilon_n$ instead of $\epsilon_{n_i}$ and $x^{t_0}_n$ instead of $x^{t_0}_{\epsilon_{n_i}}$. Furthermore, taking another subsequence if necessary (and keeping the notation), we can consider that $x^{t_0}_n\in \text{Int}\big(\mathbb{X}\big)$ or $x^{t_0}_n\in \partial\mathbb{X}$ for all $n\in\mathbb{N}$.
  			
  			For each term of the sequence $\big(x^{t_0}_n\big)_{n\geqslant0}$, it follows by the compactness of the torus $\mathbb{T}^{d_Y}$ and continuity of the application $y\mapsto\big(v^{\beta, \epsilon_n}-\varphi^{\epsilon_n}\big)\big(x^{t_0}_n,y\big)$, that there exists a sequence $\big(y^{t_0}_n\big)_{n\geqslant0}$ and a point $y^{t_0}\in\mathbb{T}^{d_Y}$ such that $\lim_{n\to \infty}y^{t_0}_n=y^{t_0}$.
  			
  			\item[Step 2] (Solution of the effective problem when $(x^{t_0}_n,y^{t_0}_n)\in\text{Int}\big(\mathbb{X}\big)\times\mathbb{T}^{d_Y}$) In this step, we use the test function \eqref{eq:perturbed_test_function} and the sequences established in step 1, to prove that: when we have $(x^{t_0}_n,y^{t_0}_n)\in\text{Int}\big(\mathbb{X}\big)\times\mathbb{T}^{d_Y}$, then \eqref{eq:auxiliary_effective_boundary_condition} holds.
  			
  			By step 1, there are sequences $\epsilon_n\downarrow 0$, $x^{t_0}_n\to x_0$ and $y^{t_0}_n\to y^{t_0}$, such that, the pair $(x^{t_0}_n,y^{t_0}_n)$ is a local maximum of the application $(x,y)\mapsto\big(v^{\beta,\epsilon_n}-\varphi^{\epsilon_n}\big)\big(x,y\big)$ and
  			\begin{equation*}
  				\lim_{n\to \infty}\big(v^{\beta,\epsilon_n}-\varphi^{\epsilon_n}\big)\big(x^{t_0}_n,y^{t_0}_n\big)=v^{\beta}_{\text{sup}}\big(x_0\big)-\varphi\big(x_0\big)=0.
  			\end{equation*}
  			Since every viscosity solution is also a viscosity subsolution, we have, by Definition \ref{def:viscosity_solutions}, that
  			\begin{subequations}
  				\begin{align}
  					\label{eq:first_possibility_01}
  					F^{\beta,\epsilon_n}\Big(x^{t_0}_n,y^{t_0}_n,\varphi^{\epsilon_n},D\varphi^{\epsilon_n},D^2\varphi^{\epsilon_n}\Big)&\leqslant 0\quad
  					(x^{t_0}_n,y^{t_0}_n)\in\text{Int}\big(\mathbb{X}\big)\times\mathbb{T}^{d_Y},\\
  					\label{eq:second_possibility_01}
  					F^{\beta,\epsilon_n}\Big(x^{t_0}_n,y^{t_0}_n,\varphi^{\epsilon_n},D\varphi^{\epsilon_n},D^2\varphi^{\epsilon_n}\Big)\wedge \Gamma^{\epsilon_n}\Big(x^{t_0}_n,y^{t_0}_n,D\varphi^{\epsilon_n}\Big)&\leqslant 0\quad
  					(x^{t_0}_n,y^{t_0}_n)\in \partial\mathbb{X}\times\mathbb{T}^{d_Y},
  				\end{align}
  			\end{subequations}
  			where $F^{\beta,\epsilon_n}$ and $\Gamma^{\epsilon_n}$ appear when we rewrite the multiscale HJB equation \eqref{subeq:multiscale_HJB_Von_Neumann_boundary} in the form of \cite[Equation 29]{CalixtoCostaValle2025}.
  			
  			We deal with the case \eqref{eq:first_possibility_01}. By definition of the function $F^{\beta,\epsilon_n}$, we have that
  			\begin{equation*}
  				\beta \varphi^{\epsilon_n}-\mathcal{H}\Bigg(x^{t_0}_n,y^{t_0}_n,D_{x}\varphi^{\epsilon_n},\frac{D_{y}\varphi^{\epsilon_n}}{\epsilon_n},D^2_{x}\varphi^{\epsilon_n},\frac{D^2_{y}\varphi^{\epsilon_n}}{\epsilon_n},\frac{D^2_{x,y}\varphi^{\epsilon_n}}{\sqrt{\epsilon_n}}\Bigg)\leqslant 0.
  			\end{equation*}
  			By definition of $\mathcal{H}$, we can write the above inequality as
  			\begin{align*}
  				\beta \varphi^{\epsilon_n}&-\min_{u\in \mathbb{U}}\Biggl\{\mathcal{G}_{X}^{u}\big(x^{t_0}_n,D_{x}\varphi^{\epsilon_n},D^2_{x}\varphi^{\epsilon_n}\big|y^{t_0}_n\big)+L\big(x^{t_0}_n,y^{t_0}_n,u\big)+\text{Tr}\Big(\frac{D^2_{x,y}\varphi^{\epsilon_n}}{\sqrt{\epsilon_n}}\big[\sigma_{X}\sigma^{\top}_{Y}\big]\big(x^{t_0}_n,y^{t_0}_n,u\big)\Big)\Biggr\}\\
  				&-\mathcal{G}_{Y}\Bigg(y^{t_0}_n,\frac{D_{y}\varphi^{\epsilon_n}}{\epsilon_n},\frac{D^2_{y}\varphi^{\epsilon_n}}{\epsilon_n}\Bigg|x^{t_0}_n\Bigg)\leqslant 0.
  			\end{align*}
  			From the definition of $\varphi^{\epsilon}$ we obtain, by Leibniz's formula, that
  			\begin{equation*}
  				D_y\varphi^{\epsilon_n}\big(x,y\big)=\frac{\epsilon_n}{t_0}\int_{0}^{t_0}D_yw_p\big(s,y\big)\,ds\quad\text{and}\quad 	D^2_y\varphi^{\epsilon_n}\big(x,y\big)=\frac{\epsilon_n}{t_0}\int_{0}^{t_0}D^2_yw_p\big(s,y\big)\,ds
  			\end{equation*}
  			From the above formulas and the linearity of $\mathcal{G}_{Y}$, we have that
  			\begin{align}
  				\label{eq:auxiliary_inequality_19}
  				\beta \varphi^{\epsilon_n}&-\min_{u\in \mathbb{U}}\Bigg\{\mathcal{G}_{X}^{u}\big(x^{t_0}_n,D_{x}\varphi,D^2_{x}\varphi\big|y^{t_0}_n\big)+L\big(x^{t_0}_n,y^{t_0}_n,u\big)\Bigg\}\\
  				\nonumber
  				&-\frac{1}{t_0}\int_{0}^{t_0}\mathcal{G}_{Y}\Big(y^{t_0}_n,D_{y}w_p\big(s,y^{t_0}_n\big),D^2_{y}w_p\big(s,y^{t_0}_n\big)\bigg|x^{t_0}_n\Big)\,ds\leqslant 0.
  			\end{align}
  			Since $w_p$ is a classical solution of the Cell-$t$-Problem \eqref{subeq:PDE_cell_t_problem}, it follows that
  			\begin{equation*}
  				-\partial_t w_p\big(s,y^{t_0}_n\big)+\mathcal{G}_{Y}\Bigg(y^{t_0}_n,D_yw_p\big(s,y^{t_0}_n\big),D^2_yw_p\big(s,y^{t_0}_n\big)\bigg|x^{t_0}_n\Bigg)+\mathfrak{H}_p\big(y^{t_0}_n\big)=0.
  			\end{equation*}
  		    Integrating up to time $t_0>0$ and then dividing by that same value, we get
  		    \begin{equation}
  		       \label{eq:auxiliary_equality_20}
  		    	- \frac{1}{t_0}w_p\big(t_0,y^{t_0}_n\big)+\frac{1}{t_0}\int_{0}^{t_0}\mathcal{G}_{Y}\bigg(y^{t_0}_n,D_yw_p\big(s,y^{t_0}_n\big),D^2_yw_p\big(s,y^{t_0}_n\big)\bigg|x^{t_0}_n\bigg)ds+\mathfrak{H}_p\big(y^{t_0}_n\big)=0.
  		    \end{equation}
  		    By combining the inequality \eqref{eq:auxiliary_inequality_19} with the equality \eqref{eq:auxiliary_equality_20}, it follows that
  		    \begin{equation*}
  		    	\beta \Bigg(\varphi\big(x^{t_0}_n\big)+ \frac{\epsilon_n}{t_0}\int_{0}^{t_0}w_p\big(s,y^{t_0}_n\big)\,ds\Bigg)-\min_{u\in \mathbb{U}}\Bigg\{\mathcal{G}_{X}^{u}\big(x^{t_0}_n,D_{x}\varphi,D^2_{x}\varphi\big|y^{t_0}_n\big)+L\big(x^{t_0}_n,y^{t_0}_n,u\big)\Bigg\}+\mathfrak{H}_p\big(y^{t_0}_n\big)-\frac{1}{t_0}w_p\big(t_0,y^{t_0}_n\big)\leqslant 0.
  		    \end{equation*}
  		    By making $n\to+\infty$ it follows, from the continuity of each of the above terms, that
  		    \begin{equation*}
  		    \beta \varphi\big(x_0\big)-\frac{1}{t_0}w_p\big(t_0,y^{t_0}\big)\leqslant 0.
  		    \end{equation*}
  		    From the inequality \eqref{eq:auxiliary_inequality_18}, we have that 
  		    \begin{equation*}
  		    	\beta \varphi\big(x_0\big)-\mathcal{\overline{H}}\Big(x_0,D_x\varphi\big(x_0\big),D^2_x\varphi \big(x_0\big)\Big)\leqslant \varsigma.
  		    \end{equation*}
  		    Making $\varsigma \to 0$ we get
  		    \begin{equation*}
  		    	\beta \varphi\big(x_0\big)-\mathcal{\overline{H}}\Big(x_0,D_x\varphi\big(x_0\big),D^2_x\varphi \big(x_0\big)\Big)\leqslant 0.
  		    \end{equation*}
  		    What implies the inequality \eqref{eq:auxiliary_effective_boundary_condition}
  		    \item[Step 3](Solution of the effective problem when $(x^{t_0}_n,y^{t_0}_n)\in \partial\mathbb{X}\times\mathbb{T}^{d_Y}$) We treat the case \eqref{eq:second_possibility_01}. This case is divided into the following sub-cases:
  			\begin{subequations}
  			\begin{align}
  				\label{eq:first_possibility_02}
  				\beta \varphi^{\epsilon_n}-\mathcal{H}\Bigg(x^{t_0}_n,y^{t_0}_n,D_{x}\varphi^{\epsilon_n},\frac{D_{y}\varphi^{\epsilon_n}}{\epsilon_n},D^2_{x}\varphi^{\epsilon_n},\frac{D^2_{y}\varphi^{\epsilon_n}}{\epsilon_n},\frac{D^2_{x,y}\varphi^{\epsilon_n}}{\sqrt{\epsilon_n}}\Bigg)&\leqslant 0\\
  				\nonumber
  				\text{or}\\
  				\label{eq:second_possibility_02}
  				\bigg\langle D\varphi^{\epsilon_n}\big(x^{t_0}_n,y^{t_0}_n\big), D\widetilde{\phi}\big(x^{t_0}_n,y^{t_0}_n\big)\bigg\rangle-h\big(x^{t_0}_n\big)&\leqslant 0.
  			\end{align}
  		\end{subequations}
  		If the inequality \eqref{eq:first_possibility_02} holds, then we can repeat the same procedure as in step 2 and obtain the inequality \eqref{eq:auxiliary_effective_boundary_condition}. Now, if the inequality \eqref{eq:second_possibility_02} holds, then
  		by definitions of $\widetilde{\phi}$ ($\widetilde{\phi}\big(x,y\big):=\phi\big(x\big)$) and $\varphi^{\epsilon_n}$(equation \eqref{eq:perturbed_test_function}), it follows that
  		\begin{equation*}
  			\bigg\langle D_x\varphi\big(x^{t_0}_n\big), D_x\phi\big(x^{t_0}_n\big)\bigg\rangle-h\big(x^{t_0}_n\big)\leqslant 0.
  		\end{equation*}
  		Making $n\to +\infty$ we have, by continuity, that
  		\begin{equation*}
  			\bigg\langle D_x\varphi\big(x_0\big), D_x\phi\big(x_0\big)\bigg\rangle-h\big(x_0\big)\leqslant 0.
  		\end{equation*}
  		Thus, we obtain the inequality \eqref{eq:auxiliary_effective_boundary_condition}.
  		
  		\item[Step 4] In steps 2 and 3, we proved that $v_{\text{sup}}^{\beta}:\mathbb{X}\to\mathbb{R}$ is a viscosity subsolution of the effective HJB equation \eqref{subeq:effective_HJB_Von_Neumann_boundary}.  To demonstrate that the relaxed semilimit \eqref{eq:lower_relaxed_semilimits} is a viscosity supersolution of the effective HJB equation \eqref{subeq:effective_HJB_Von_Neumann_boundary}, we repeat the same arguments we used for the viscosity subsolution. 
  		
  		On the one hand, by definition of the relaxed semilimits \eqref{subeq:relaxed_semilimits}, we have that
  		\begin{align*}
  			v^{\beta}_{\text{inf}}\big(x_0\big)=\liminf_{\epsilon\to 0, x\to x_0}\inf_{y\in \mathbb{T}^{d_Y}}v^{\beta,\epsilon}\big(x,y\big)\leqslant\limsup_{\epsilon\to 0, x\to x_0}\sup_{y\in\mathbb{T}^{d_Y}}v^{\beta,\epsilon}\big(x,y\big)=v^{\beta}_{\text{sup}}\big(x_0\big).
  		\end{align*}
  		On the other hand, by \cite[Theorem 2.12]{CalixtoCostaValle2025} (which encapsulates the comparison principle), we obtain the opposite inequality. Therefore, $v_{\text{sup}}^{\beta}=v_{\text{inf}}^{\beta}$ in $\mathbb{X}$. Defining $\widetilde{v}^{\beta}:\mathbb{X}\to\mathbb{R}$ by
  		\begin{equation}
  			\label{eq:auxiliary_effective_optimum_value_function}
  			\widetilde{v}^{\beta}:=v_{\text{sup}}^{\beta}=v_{\text{inf}}^{\beta},
  		\end{equation}
  		we obtain that $\widetilde{v}^{\beta}$ is the unique viscosity solution of the effective HJB equation \eqref{subeq:effective_HJB_Von_Neumann_boundary}. In addition, the function \eqref{eq:auxiliary_effective_optimum_value_function} is continuous. We also know, from \cite[Theorem 2.11]{CalixtoCostaValle2025}, that the effective optimal value function is a viscosity solution of the effective HJB equation \eqref{subeq:effective_HJB_Von_Neumann_boundary}. Therefore, $\overline{v}^{\beta}=\widetilde{v}^{\beta}$ and we have, from \cite[Lemma 1.9 of Chapter 5]{BardiDolcetta}, that the family $\big(v^{\beta,\epsilon}\big)_{\epsilon>0}$ of multiscale optimal value functions converges uniformly to the effective optimal value function $\overline{v}^{\beta}$.
  		\end{itemize}
  \end{proof}
  \section{Examples}
  \label{sec:examples}
   In this section we present two examples that implement what we have developed in the previous sections. In the first, we have a control problem for which the HBJ equation is semilinear, so we can explicitly calculate the optimal control as a function of the gradient of the optimal value function. In the second, we have a control problem for which the HJB equation is totally nonlinear.
  
   In both cases, we will consider the following space for the slow variable
  \begin{equation}
  	 	\label{eq:domain_D_slow_variable}
  	\mathbb{X}:=\big[-\alpha,\alpha\big].
  \end{equation}
   for $\alpha>0$. Consider the following function $\phi:\mathbb{R}\to\mathbb{R}$ by
     \begin{equation}
   	\label{eq:example_phi_function}
   	\phi\big(x\big):=\frac{1}{2\alpha}e^{\alpha^2-x^2}\bigg(x^2-\alpha^2\bigg).
   \end{equation}
    Let's show that the set $\mathbb{X}$ has the following representation 
  \begin{equation}
  	\label{eq:domain_D_slow_variable_phi}
  	\mathbb{X}=\Bigg\{x\in\mathbb{R}:\phi\big(x\big)\leqslant 0\Bigg\}
  \end{equation}
  in terms of the function \eqref{eq:example_phi_function}. To do this, let's look at some properties of the function \eqref{eq:example_phi_function}. First of all, note that
  \begin{equation*}
  	\lim_{x\to+\infty}\phi\big(x\big)=0\quad\text{and}\quad \lim_{x\to-\infty}\phi\big(x\big)=0.
  \end{equation*}
  Note that \eqref{eq:example_phi_function} is clearly $\text{C}^2\big(\mathbb{R}\big)$ with derivative given by
  \begin{equation}
  	\label{eq:example_phi_function_derivative}
  	\frac{d\phi}{dx}\big(x\big)=\frac{x}{\alpha}e^{\alpha^2-x^2}\big(1+\alpha^2-x^2\big).
  \end{equation}
  At points $-\alpha$ and $\alpha$ we have that
  \begin{equation*}
  	\frac{d\phi}{dx}\big(\alpha\big)=1\quad \text{and}\quad \frac{d\phi}{dx}\big(-\alpha\big)=-1.
  \end{equation*}
  The critical points of $\phi$ are
  \begin{equation*}
  	\frac{d\phi}{dx}\big(0\big)=0,\quad \frac{d\phi}{dx}\big(\sqrt{1+\alpha^2}\big)=0\quad\text{and}\quad \frac{d\phi}{dx}\big(-\sqrt{1+\alpha^2}\big)=0.
  \end{equation*}
  where $0$ is the minimum point and $\sqrt{1+\alpha^2}$ and $-\sqrt{1+\alpha^2}$ are the maximum points.  Then we have that \eqref{eq:example_phi_function} is a bounded function. Furthermore, we have
  \begin{equation*}
  	\phi\big(x\big)<0\quad \forall x\in\big(-\alpha,\alpha\big),\quad \phi\big(-\alpha\big)=0\quad\text{and} \quad \phi\big(\alpha\big)=0.
  \end{equation*}
  We conclude that the set \eqref{eq:domain_D_slow_variable}  admits the representation \eqref{eq:domain_D_slow_variable_phi}.
  
  The above development is generalized to the $d_X$-dimensional case when we consider $\mathbb{X}:=\text{B}\big(0,\alpha\big)$ with $\alpha>0$ and $\phi:\mathbb{R}^{d_X}\to\mathbb{R}$ given by 
  \begin{equation*}
  	\phi\big(x\big):=\frac{1}{2\alpha}e^{\alpha^2-\|x\|^2}\big(\|x\|^2-\alpha^2\big).
  \end{equation*}
  \subsection{Control Problem with Semilinear HJB Equation}
  \label{subsec:control_problem_semilinear_hjb_equation}
  Given $0<\epsilon \ll1$, the two-scale stochastic dynamical system we will be working with is defined by:
  \begin{subequations}
  	\label{subeq:multiscale_dynamic_system_example_01}
  	\begin{align}
  		\label{eq:dynamic_system_slow_scale_exemplo_01}
  		dX^{\epsilon}_x(t)=&\Big[\theta_a X^{\epsilon}_x(t)\sin\big(2\pi Y^{\epsilon}_{1,y}(t)\big)\sin\big(2\pi Y^{\epsilon}_{2,y}(t)\big)-\theta_b u(t)\Big]\,dt+\sigma_{X}X^{\epsilon}_x(t)\,dW(t)-D_x\phi\big(X^{\epsilon}_x(t)\big)\,dl^{\epsilon}_{x,y}(t),\\
  		\label{eq:dinamica_escala_rapida_y1_exemplo_01}
  		dY^{\epsilon}_{1,y}(t)=&\frac{\theta_c}{\epsilon}X^{\epsilon}_x(t)\cos\Big(2\pi\big(Y^{\epsilon}_{1,y}(t)-Y^{\epsilon}_{2,y}(t)\big)\Big)\,dt+\frac{\sigma_{Y}}{\sqrt{\epsilon}}\,dW(t),\\
  		\label{eq:dinamica_escala_rapida_y2_exemplo_01}
  		dY^{\epsilon}_{2,y}(t)=&\frac{\theta_c}{\epsilon}X^{\epsilon}_x(t)\cos\Big(2\pi\big(Y^{\epsilon}_{1,y}(t)-Y^{\epsilon}_{2,y}(t)\big)\Big)\,dt+\frac{\sigma_{Y}}{\sqrt{\epsilon}}\,dW(t),
  	\end{align}
  \end{subequations}
  where $\theta_a,\theta_b,\theta_c\in\mathbb{R}$ are parameters, as well as $\sigma_{X}>0$ and $\sigma_{Y}>0$. The control $\big(u^{\epsilon}(t)\big)_{t\geqslant0}$ has as its state space the interval $\big[u_a,u_b\big]$ with $u_a,u_b\in\mathbb{R}$ and $u_a<u_b$. Furthermore, we consider $\mathcal{U}$ to be the set of progressively measurable processes with state space given by the interval $\big[u_a,u_b\big]$. Finally, the process $\big(l^{\epsilon}_{x,y}(t)\big)_{t\geqslant 0}$ is continuous, non-decreasing, with $l^{\epsilon}_{x,y}\big(0\big)=0$ and satisfies the condition
  \begin{equation*}
  	l^{\epsilon}_{x,y}(t)=\int_{0}^{t}\mathds{1}_{\partial \mathbb{X}}\big(X^{\epsilon}_x(t)\big)\,dl^{\epsilon}_{x,y}(s)\quad \mathbb{P}\text{-a.s}.
  \end{equation*}
  The drift and dispersion associated with the system \eqref{subeq:multiscale_dynamic_system_example_01} are given by:
  \begin{subequations}
  	\label{subeq:fields_drift_dispersion_multiscale_system_example_01}
  	\begin{align}
  		\label{eq:slow_variable_drift_example_01}
  		\mu_{X}\big(x,y_1,y_2,u\big)&:=\theta_a x\sin\big(2\pi y_1\big)\sin\big(2\pi y_2\big)-\theta_b u,\\
  		\label{eq:slow_variable_dispersion_example_01}
  		\sigma_{X}\big(x,y_1,y_2,u\big)&:=\sigma_{X}x,\\
  		\label{eq:fast_variable_drift_example_01}
  		\mu_{Y}\big(x,y_1,y_2\big)&:=\theta_c x\big(\cos\bigg(2\pi\big(y_1-y_2\big)\bigg),\cos\big(2\pi\big(y_1-y_2\big)\big)\big),\\
  		\label{eq:fast_variable_dispersion_example_01}
  		\sigma_{Y}\big(x,y_1,y_2\big)&:=\sigma_{Y}.
  	\end{align}
  \end{subequations} 
  Note that the fields \eqref{eq:slow_variable_drift_example_01} and \eqref{eq:slow_variable_dispersion_example_01} of the slow variable are bounded and Lipschitz continuous in the state variables $(x,y)$ in $\mathbb{X}\times \mathbb{R}^2\times \big[u_a,u_b\big]$ uniformly with respect to the control variable $u$. Furthermore, all fields are $\mathbb{Z}^2$-periodic in the fast variable with the fields of the slow variable admitting the following decomposition:
  \begin{align*}
  	\mu_{\text{SF}}\big(x,y_1,y_2\big):=\theta_a x\sin\big(2\pi y_1\big)\sin\big(2\pi y_2\big)\quad&\text{and}\quad\mu_{\text{SC}}\big(x,u\big):=-\theta_b u,\\
  	\sigma_{\text{SF}}\big(x,y_1,y_2\big):=\sigma_{X}x\quad&\text{and}\quad\sigma_{\text{SC}}\big(x,u\big):=0.
  \end{align*}
  Thus, the hypotheses of item I of Section \S\ref{sec:multiscale_stochastic_optimal_control} are satisfied. 
  
  We now move on to the definitions of the operational and preventive costs. We consider the following function to be the operational cost
  \begin{equation}
  	\label{eq:operational_cost_example_01}
  	L\big(x,y_1,y_2,u\big):=\big(\theta_d-u\big)^2
  \end{equation}
  with $\theta_d>0$ a parameter. This function is bounded and Lipschitz continuous in $\mathbb{X}\times \mathbb{R}^2\times \big[u_a,u_b\big]$ uniformly with respect to the control variable $u$. It is also $\mathbb{Z}^2$-periodic in the fast variable and admits the following decomposition:
  \begin{equation*}
  	L_{\text{SF}}\big(x,y_1,y_2\big):=0\quad \text{and}\quad L_{\text{SC}}\big(x,u\big):=\big(\theta_d-u\big)^2.
  \end{equation*}
  Thus, the hypotheses of item III of Section \S\ref{sec:multiscale_stochastic_optimal_control} are satisfied. For the preventive cost we simply take
  \begin{equation*}
  	h\big(-\alpha\big):=\theta_e=:h\big(\alpha\big)
  \end{equation*}
  where $\theta_e\in\mathbb{R}$ is a parameter. 
  
  Fixed $\beta>0$, we define the cost functional by
  \begin{equation}
  	\label{eq:operational_functional_example_01}
  	J_{x,y}^{\beta,\epsilon}\big(u\big):=\mathbb{E}\Bigg[\int_{0}^{+\infty}e^{-\beta s}\big(\theta_d-u(s)\big)^2\,ds+\int_{0}^{+\infty}e^{-\beta s}\theta_e\,dl^{\epsilon}_{x,y}(s)\Bigg].
  \end{equation}
  Thus, the optimal value function is given by
  \begin{equation}
  	\label{eq:optimal_value_function_example_01}
  	v^{\beta,\epsilon}\big(x,y\big):=\inf_{u\in \mathcal{U}}J_{x,y}^{\beta,\epsilon}\big(u\big)
  \end{equation}
  and by \cite[Theorem 2.7]{CalixtoCostaValle2025}, the function \eqref{eq:optimal_value_function_example_01} is continuous. Using the principle of dynamic programming, proven in \cite[Subsection 2.2]{CalixtoCostaValle2025}, we obtain the following multiscale HJB equation 
  \begin{subequations}
  	\label{subeq:multiscale_HJB_equation_boundary_condition_example_01}
  	\begin{align}
  		\label{eq:multiscale_HJB_equation_exemplo_01}
  		\beta v^{\beta,\epsilon}-\mathcal{H}\Bigg(x,y,\partial_xv^{\beta,\epsilon},\frac{D_yv^{\beta,\epsilon}}{\epsilon},\partial^2_{x^2}v^{\beta,\epsilon},\frac{D^2_yv^{\beta,\epsilon}}{\epsilon},\frac{D^2_{x,y}v^{\beta,\epsilon}}{\sqrt{\epsilon}}\Bigg)&=0\quad\forall (x,y)\in \mathbb{X}\times \mathbb{T}^2,\\
  		\label{eq:boundary_condition_example_01}
  		\partial_{x}v^{\beta,\epsilon}\big(-\alpha,y\big)=-\theta_e\quad \text{and}\quad \partial_{x}v^{\beta,\epsilon}\big(\alpha,y\big)&=\theta_e \quad\forall y\in\mathbb{T}^2.
  	\end{align}
  \end{subequations}
  where the multiscale Hamiltonian is defined by
  \begin{align*}
  	\mathcal{H}\bigg(x,y,g_x,g_y,H_x,H_y,H_{xy}\bigg):=&\frac{\sigma^2_{X}x^2}{2}H_x+\frac{\sigma^2_{Y}}{2}\text{Tr}\Bigg(H_y\begin{bmatrix}
  		1\\
  		1
  	\end{bmatrix}\begin{bmatrix}
  		1\\
  		1
  	\end{bmatrix}^{\top}\Bigg)+\sigma_{Y}\sigma_{X}x\text{Tr}\Bigg(H_{x,y}\begin{bmatrix}
  		1\\
  		1
  	\end{bmatrix}^{\top}\Bigg)+\theta_ax\sin\big(2\pi y_1\big)\sin\big(2\pi y_2\big)g_x\\
  	&+\theta_cx\cos\big(2\pi\big(y_1-y_2\big)\big)\Bigg\langle\begin{bmatrix}
  		1\\
  		1
  	\end{bmatrix},g_y\Bigg\rangle+\min_{u\in [u_a,u_b]}\bigg\{u^2-\big(2\theta_d+\theta_bg_x\big)u\bigg\}+\theta_d^2.
  \end{align*}
  To obtain an explicit formula for the Hamiltonian \eqref{eq:multiscale_HJB_equation_exemplo_01}, we must calculate the minimization 
  \begin{equation}
  	\label{eq:multiscale_control_function_example_01}
  	u_{\eta}^{*}\big(g_x\big):=\arg\min_{u\in [u_a,u_b]}\bigg\{u^2-f_{\eta}\big(g_x\big)u\bigg\}
  \end{equation}
  where $\eta:=\big(\theta_b,\theta_d,u_a,u_b\big)$ and \(f_{\eta}\big(g_x\big):=2\theta_d+\theta_bg_x\).  Furthermore, observe that if \(v^{\beta,\epsilon}:\mathbb{X}\times\mathbb{T}^{d_Y}\mapsto \mathbb{R}\) is at least Lipschitz continuous with respect to the slow variable, then by Rademacher’s Theorem \cite[Theorem 3.2]{EvansGariepy}, the derivative of \(v^{\beta,\epsilon}\) with respect to the slow variable exists almost everywhere with respect to the Lebesgue measure and is measurable at every point where it is differentiable. Thus, by means of equation \eqref{eq:multiscale_control_function_example_01}, we obtain the following Markovian control
 \begin{equation}
 	\label{eq:multiscale_markovian_control_example_01}
 	\widetilde{u}^{\epsilon}\big(x,y\big):=u_{\eta}^{*}\bigg(\partial_xv^{\beta,\epsilon}\big(x,y\big)\bigg).
 \end{equation}
 
Observe that, due to the low regularity of the optimal value function \eqref{eq:optimal_value_function_example_01} --- even in the case where it is Lipschitz continuous --- it is not possible to apply classical control tools such as, for instance, the Verification Theorem to ensure the optimality of \eqref{eq:multiscale_markovian_control_example_01}.
   
 Finally, by \cite[Theorem 2.11]{CalixtoCostaValle2025}, the optimal value function \eqref{eq:optimal_value_function_example_01} is a viscosity solution of the PDE \eqref{eq:multiscale_HJB_equation_exemplo_01}, and by \cite[Theorem 2.12]{CalixtoCostaValle2025}, it is the unique viscosity solution.
 
  We now proceed to the construction of the effective control problem. Fixing $\bar{x}\in \mathbb{X}$ and taking $\epsilon=1$, the SDE associated with the fast system is defined by
 \begin{equation}
 	\label{eq:sde_fast_system_example_01}
 	\begin{bmatrix}
 		dY_1(t)\\
 		dY_2(t)
 	\end{bmatrix}=\theta_c \bar{x}\cos\Big(2\pi\big(Y_1(t)-Y_2(t)\big)\Big)\begin{bmatrix}
 		1\\
 		1
 	\end{bmatrix}\,dt+\sigma_{Y}\begin{bmatrix}
 		1\\
 		1
 	\end{bmatrix}\,dW(t).
 \end{equation}
 The adjoint operator of the infinitesimal generator associated with  \eqref{eq:sde_fast_system_example_01} is given by
 \begin{equation*}
 	\mathcal{L}^{\bar{x},*}\rho_{\bar{x}}\big(y_1,y_2\big)=-\theta_c \bar{x}\Bigg(\frac{\partial}{\partial y_1}\bigg(\cos\big(2\pi \big(y_1-y_2\big) \big)\rho_{\bar{x}}\big(y_1,y_2\big)\bigg)+\frac{\partial}{\partial y_2}\bigg(\cos\big(2\pi \big(y_1-y_2\big) \big)\rho_{\bar{x}}\big(y_1,y_2\big)\bigg)\Bigg)+\frac{\sigma^2_{Y}}{2}\Delta\rho_{\bar{x}}\big(y_1,y_2\big).
 \end{equation*}
 Thus, the Fokker-Planck is given by
 \begin{align}
 	\label{eq:fokker_planck_PDE_example_01}
 	\frac{\sigma^2_{Y}}{2}\Delta\rho_{\bar{x}}\big(y_1,y_2\big)-\theta_c \bar{x}\Bigg(\frac{\partial}{\partial y_1}\bigg(\cos\big(2\pi \big(y_1-y_2\big) \big)\rho_{\bar{x}}\big(y_1,y_2\big)\bigg)+\frac{\partial}{\partial y_2}\bigg(\cos\big(2\pi \big(y_1-y_2\big) \big)\rho_{\bar{x}}\big(y_1,y_2\big)\bigg)\Bigg)=0.
 \end{align}
 From item II of Section \S\ref{sec:multiscale_stochastic_optimal_control}, it follows, by Theorem \ref{tm:ergodicity_markov_process}, that the PDE \eqref{eq:fokker_planck_PDE_example_01} has a unique classical solution on the torus $\mathbb{T}^2$. On the other hand, the fields \eqref{eq:fast_variable_drift_example_01} and \eqref{eq:fast_variable_dispersion_example_01} of the fast variable belong to $\text{C}^{\infty}\big(\mathbb{X}\times\mathbb{T}^2\big)$ and are therefore $\text{C}^{1,m+2}\big(\mathbb{X}\times\mathbb{T}^2\big)$ for all $m>1$. Let's check that the main hypothesis of item VI in Subsection \S\ref{subsubsec:ergodic_theory_Markov_processes} is satisfied. As the diffusion matrix is constant, the condition \eqref{eq:maximum_principle_condition} is given by
 \begin{equation*}
 	\partial_{y_1}\mu_{Y,1}\big(x,y_1,y_2\big)+\partial_{y_2}\mu_{Y,2}\big(x,y_1,y_2\big)\geqslant 0.
 \end{equation*}
 In fact, we have
 \begin{equation*}
 	\partial_{y_1}\mu_{Y,1}\big(x,y_1,y_2\big)+\partial_{y_2}\mu_{Y,2}\big(x,y_1,y_2\big)=-2\pi\theta_cx\sin\big(2\pi\big(y_1-y_2\big)\big)+2\pi\theta_cx\sin\big(2\pi \big(y_1-y_2\big)\big) =0.
 \end{equation*}
 Thus, by Lemma \ref{lm:differentiability_fokker_plack_parameters}, the solution of the Fokker-Planck \eqref{eq:fokker_planck_PDE_example_01} is differentiable with respect to the slow scale parameter. Therefore, using the density $\rho_{\bar{x}}$, the effective drift and dispersion are given by:
 \begin{align*}
 	\mu_{\text{SC}}\big(\bar{x},u\big)&=-\theta_b u \quad \text{and}\quad\overline{\mu}_{\overline{X}}\big(\bar{x}\big)=\theta_a\kappa\big(\bar{x}\big)\bar{x},\\
 	\sigma_{\text{SC}}\big(\bar{x},u\big)&=0\qquad\hspace{0.31cm}\text{and}\quad\overline{\sigma}_{\overline{X}}\big(\bar{x}\big)=\sigma_{X}\bar{x}
 \end{align*}
 where $\kappa:\mathbb{X}\to \mathbb{R}$ is a function defined by 
 \begin{equation*}
 	\kappa\big(\bar{x}\big):=\int_{\mathbb{T}^2}\sin\big(2\pi y_1\big)\sin\big(2\pi y_2\big)\rho_{\bar{x}}\big(y_1,y_2\big)dy_1dy_2.
 \end{equation*}
 An important property of the $\kappa$ function is that it is Lipschitz continuous in function of the arguments presented at the beginning of Subsection \S\ref{subsubsec:ergodic_theory_Markov_processes}. The effective SDE is given by 
 \begin{equation}
 	\label{eq:effective_sde_example_01}
 	d\overline{X}_x(t)=\Big[\theta_a\kappa\big(\overline{X}_x(t)\big) \overline{X}_x(t)-\theta_bu(t)\Big]\,dt+\sigma_{X}\overline{X}_x(t)\,dW(t)-D_x\phi\big(\overline{X}_x(t)\big)\,d\overline{l}_x(t).
 \end{equation}
 The operational and preventive cost functions remain the same. In this case, the effective optimum value function takes the following form
 \begin{equation}
 	\label{eq:effective_optimal_value_function_example_01}
 	\overline{v}^{\beta}\big(x\big):=\inf_{u\in\mathcal{U}}\mathbb{E}\Bigg[\int_{0}^{+\infty}e^{-\beta s}\big(\theta_d-u(s)\big)^2\,ds+\int_{0}^{+\infty}e^{-\beta s}\theta_e\,d\overline{l}_x(s)\Bigg].
 \end{equation}
 The effective HJB equation associated with the \eqref{eq:effective_sde_example_01}-\eqref{eq:effective_optimal_value_function_example_01} problem is
 \begin{equation}
 	\label{eq:effective_HJB_equation_example_01}
 	\beta \overline{v}^{\beta}\big(x\big)-\frac{\sigma^2_{X}x^2}{2}\partial^2_x\overline{v}^{\beta}\big(x\big)-\theta_a \kappa\big(x\big) x\partial_x\overline{v}^{\beta}\big(x\big)-\min_{u\in [u_a,u_b]}\bigg\{u^2-\big(2\theta_d+\theta_b\partial_x\overline{v}^{\beta}\big(x\big)\big)u\bigg\}-\theta^2_d=0
 \end{equation}
 with the Neumann boundary
 \begin{equation*}
 	\partial_x\overline{v}^{\beta}\big(-\alpha\big)=-\theta_e\quad\text{and}\quad  \partial_x\overline{v}^{\beta}\big(\alpha\big)=\theta_e.
 \end{equation*}
 Once again, by \cite[Theorem 2.11]{CalixtoCostaValle2025}, we obtain that the optimal value function \eqref{eq:effective_optimal_value_function_example_01} is a viscosity solution of the PDE \eqref{eq:effective_HJB_equation_example_01} and, by \cite[Theorem 2.12]{CalixtoCostaValle2025}, it is the unique viscosity solution. 
 
As in the multiscale case, to obtain an explicit formula for the PDE \eqref{eq:effective_HJB_equation_example_01}, we compute the minimization
\begin{equation*}
	\overline{u}_{\eta}^{*}\big(g_x\big):=\arg\min_{u\in [u_a,u_b]}\bigg\{u^2 - f_{\eta}\big(g_x\big)u\bigg\}.
\end{equation*}
Thus, by the same observation made in the multiscale case, we obtain the Markovian control
\begin{equation*}
	\widetilde{u}\big(x\big):= \overline{u}_{\eta}^{*}\bigg(\partial_x \overline{v}^{\beta}\big(x\big)\bigg)
\end{equation*}
in the case of the value function \eqref{eq:effective_optimal_value_function_example_01} be Lipschitz continuous.

 \subsection{Control Problem with Fully Nonlinear HJB Equation}
 \label{subsec:control_problem_fully_nonlinear_HJB_equation}
 Given $0<\epsilon \ll1$, the two-scale stochastic dynamical system we will be working with is defined by:
 \begin{subequations}
 	\label{subeq:multiscale_dynamic_system_example_02}
 	\begin{align}
 		\label{eq:dynamic_system_slow_scale_exemplo_02}
 		dX^{\epsilon}_x(t)=&\Big[\theta_a X^{\epsilon}_x(t)\sin\big(2\pi Y^{\epsilon}_{1,y}(t)\big)\sin\big(2\pi Y^{\epsilon}_{2,y}(t)\big)-\theta_b u(t)\Big]\,dt+\sigma_{X} \sqrt{u(t)}X^{\epsilon}_x(t)\,dW(t)-D_x\phi\big(X^{\epsilon}_x(t)\big)\,dl^{\epsilon}_{x,y}(t),\\
 		\label{eq:dinamica_escala_rapida_y1_exemplo_02}
 		dY^{\epsilon}_{1,y}(t)=&\frac{\theta_c}{\epsilon}X^{\epsilon}_x(t)\cos\Big(2\pi\big(Y^{\epsilon}_{1,y}(t)-Y^{\epsilon}_{2,y}(t)\big)\Big)\,dt+\frac{\sigma_{Y}}{\sqrt{\epsilon}}dW(t),\\
 		\label{eq:dinamica_escala_rapida_y2_exemplo_02}
 		dY^{\epsilon}_{2,y}(t)=&\frac{\theta_c}{\epsilon}X^{\epsilon}_x(t)\cos\Big(2\pi\big(Y^{\epsilon}_{1,y}(t)-Y^{\epsilon}_{2,y}(t)\big)\Big)\,dt+\frac{\sigma_{Y}}{\sqrt{\epsilon}}\,dW(t),
 	\end{align}
 \end{subequations}
 The difference between the dynamical system \eqref{subeq:multiscale_dynamic_system_example_01} shown previously and \eqref{subeq:multiscale_dynamic_system_example_02} is that we include the control process in the diffusion term of the slow equation. In this case, the decomposition of the drift and dispersion fields in the slow equation is given by
 \begin{align*}
	\mu_{\text{SF}}\big(x,y_1,y_2\big):=\theta_a x\sin\big(2\pi y_1\big)\sin\big(2\pi y_2\big)\quad&\text{and}\quad\mu_{\text{SC}}\big(x,u\big):=-\theta_b u,\\
	\sigma_{\text{SF}}\big(x,y_1,y_2\big):=0\quad&\text{and}\quad\sigma_{\text{SC}}\big(x,u\big):=\sqrt{u}x.
\end{align*}

Using the same operating cost as in Example \ref{subsec:control_problem_semilinear_hjb_equation} (equation \eqref{eq:operational_cost_example_01}) as well as the same preventive boundary cost, we obtain the following HJB equation
\begin{subequations}
	\label{subeq:multiscale_HJB_equation_boundary_condition_example_02}
	\begin{align}
		\label{eq:multiscale_HJB_equation_exemplo_02}
		\beta v^{\beta,\epsilon}-\mathcal{H}\Bigg(x,y,\partial_xv^{\beta,\epsilon},\frac{D_yv^{\beta,\epsilon}}{\epsilon},\partial^2_{x^2}v^{\beta,\epsilon},\frac{D^2_yv^{\beta,\epsilon}}{\epsilon},\frac{D^2_{x,y}v^{\beta,\epsilon}}{\sqrt{\epsilon}}\Bigg)&=0\quad\forall (x,y)\in \mathbb{X}\times \mathbb{T}^2,\\
		\label{eq:boundary_condition_example_02}
		\partial_{x}v^{\beta,\epsilon}\big(-\alpha,y\big)=-\theta_e\quad \text{and}\quad \partial_{x}v^{\beta,\epsilon}\big(\alpha,y\big)&=\theta_e \quad\forall y\in\mathbb{T}^2.
	\end{align}
\end{subequations}
where the multiscale Hamiltonian is defined by
\begin{align*}
	\mathcal{H}\bigg(x,y,g_x,g_y,H_x,H_y,H_{xy}\bigg):=&\frac{\sigma^2_{Y}}{2}\text{Tr}\Bigg(H_y\begin{bmatrix}
		1\\
		1
	\end{bmatrix}\begin{bmatrix}
		1\\
		1
	\end{bmatrix}^{\top}\Bigg)+\sigma_{X}\sigma_{Y}x\text{Tr}\Bigg(H_{x,y}\begin{bmatrix}
		1\\
		1
	\end{bmatrix}^{\top}\Bigg)+\theta_ax\sin\big(2\pi y_1\big)\sin\big(2\pi y_2\big)g_x\\
	&+\theta_cx\cos\bigg(2\pi\big(y_1-y_2\big)\bigg)\Bigg\langle\begin{bmatrix}
		1\\
		1
	\end{bmatrix},g_y\Bigg\rangle+\min_{u\in [u_a,u_b]}\Bigg\{u^2-\Bigg(2\theta_d+\theta_bg_x-\frac{1}{2}\sigma^2_{X}x^2H_x\Bigg)u\Bigg\}+\theta_d^2.
\end{align*}
To obtain an explicit formula for the Hamiltonian \eqref{eq:multiscale_HJB_equation_exemplo_02}, we must calculate the minimization 
\begin{equation}
	\label{eq:multiscale_control_function_example_02}
	u_{\eta}^{*}\big(x,g_x,H_x\big):=\arg\min_{u\in [u_a,u_b]}\bigg\{u^2- \widetilde{f}_{\eta}\big(x,g_x,H_x\big)u\bigg\}
\end{equation}
where $\eta:=\big(\theta_b,\theta_d,u_a,u_b\big)$ and
\[
\widetilde{f}_{\eta}\big(x,g_x,H_x\big):=f_{\eta}\big(g_x\big)-\frac{1}{2}\sigma^2_{X}x^2H_x.
\]
 Furthermore, when \(x = 0\), the expression in~\eqref{eq:multiscale_control_function_example_02} reduces to that of equation~\eqref{eq:multiscale_control_function_example_01}. 

Unfortunately, the analysis carried out in the semilinear case is no longer applicable, even when the optimal value function is Lipschitz continuous, since it is now also necessary to take the Hessian matrix into account.

  The analysis made on the dynamic system of the fast scale to construct the effective functions is the same as that made for Example \ref{subsec:control_problem_semilinear_hjb_equation} and so  the effective drift and dispersion are given by:
 \begin{align*}
 	\mu_{\text{SC}}\big(\bar{x},u\big)&=-\theta_b u \quad \text{and}\quad\overline{\mu}_{\overline{X}}\big(\bar{x}\big)=\theta_a\kappa\big(\bar{x}\big)\bar{x},\\
 	\sigma_{\text{SC}}\big(\bar{x},u\big)&=\sigma_{X}\sqrt{u}\bar{x}\qquad\hspace{0.31cm}\text{and}\quad\overline{\sigma}_{\overline{X}}\big(\bar{x}\big)=0
 \end{align*}
 where $\kappa:\mathbb{X}\to \mathbb{R}$ is a function defined by 
 \begin{equation*}
 	\kappa\big(\bar{x}\big):=\int_{\mathbb{T}^2}\sin\big(2\pi y_1\big)\sin\big(2\pi y_2\big)\rho_{\bar{x}}\big(y_1,y_2\big)dy_1dy_2.
 \end{equation*}
 An important property of the $\kappa$ function is that it is Lipschitz continuous in function of the arguments presented at the beginning of Subsection \S\ref{subsubsec:ergodic_theory_Markov_processes}. The effective SDE is given by 
 \begin{equation}
 	\label{eq:effective_sde_example_02}
 	d\overline{X}_x(t)=\Big[\theta_a\kappa\big(\overline{X}_x(t)\big) \overline{X}_x(t)-\theta_bu(t)\Big]\,dt+\sigma_{X}\sqrt{u(t)}\overline{X}_x(t)\,dW(t)-D_x\phi\big(\overline{X}_x(t)\big)\,d\overline{l}_x(t).
 \end{equation}
 The operational and preventive cost functions remain the same. In this case, the effective optimum value function takes the following form
 \begin{equation}
 	\label{eq:effective_optimal_value_function_example_02}
 	\overline{v}^{\beta}\big(x\big):=\inf_{u\in\mathcal{U}}\mathbb{E}\Bigg[\int_{0}^{+\infty}e^{-\beta s}\big(\theta_d-u(s)\big)^2\,ds+\int_{0}^{+\infty}e^{-\beta s}\theta_e\,d\overline{l}_x(s)\Bigg].
 \end{equation}
 The effective HJB equation associated with the \eqref{eq:effective_sde_example_02}-\eqref{eq:effective_optimal_value_function_example_02} problem is
 \begin{equation}
 	\label{eq:effective_HJB_equation_example_02}
 	\beta \overline{v}^{\beta}\big(x\big)-\min_{u\in [u_a,u_b]}\Bigg\{\frac{\sigma^2_{X}}{2}ux^2\partial^2_x\overline{v}^{\beta}\big(x\big)+\big(\theta_a \kappa\big(x\big) x-\theta_b u\big)\partial_x\overline{v}^{\beta}\big(x\big)+\big(u-\theta_d\big)^2\Bigg\}=0
 \end{equation}
 with the Neumann boundary
 \begin{equation*}
 	\partial_x\overline{v}^{\beta}\big(-\alpha\big)=-\theta_e\quad\text{and}\quad  \partial_x\overline{v}^{\beta}\big(\alpha\big)=\theta_e.
 \end{equation*}
 Once again, by \cite[Theorem 2.11]{CalixtoCostaValle2025}, we obtain that the optimal value function \eqref{eq:effective_optimal_value_function_example_02} is a viscosity solution of the PDE \eqref{eq:effective_HJB_equation_example_02} and, by \cite[Theorem 2.12]{CalixtoCostaValle2025}, it is the unique viscosity solution. 
 
 As in the multiscale case, we obtain 
 \begin{equation*}
 	\overline{u}_{\eta}^{*}\big(x, g_x, H_x\big):=\arg\min_{u\in [u_a,u_b]}\bigg\{u^2- \widetilde{f}_{\eta}\big(x,g_x,H_x\big)u\bigg\}.
 \end{equation*}
 \appendix
 \section{Regularity of the Fokker-Planck Equation in Relation to Parameters}
   \label{sec:appendix_A}
    In this appendix we will prove that the solutions of the Fokker-Planck \eqref{eq:fokker_planck_PDE} are differentiable with respect to the slow variable. To do this, we need a series of definitions and results in the field of linear second-order elliptic PDEs. We will start with the concept of the \emph{Weak Derivative}.
 \begin{definition}[Weak Derivative]
 	\label{def:generalized_weak_derivative}
 	Given an open set $Q\subset\mathbb{R}^d$, a multi-index $\alpha\in \mathbb{N}^d/\{0\}$ and $p\in [1,+\infty)$, we say that a function $u\in \emph{L}_{\emph{Loc}}^1\big(Q\big)$ admits an \emph{Weak Derivative} of order $\alpha$ in $\emph{L}^{\emph{p}}\big(Q\big)$ if there exists a function $D^{\alpha}u\in \emph{L}^{\emph{p}}\big(Q\big)$ such that  
 	\begin{equation}
 		\label{eq:generalized_weak_derivative_order_alpha}
 		\int_{Q}u\big(x\big) D^{\alpha}\phi\big(x\big)dx=\big(-1\big)^{|\alpha|}\int_{Q}D^{\alpha}u\big(x\big)\phi\big(x\big)dx\quad \forall \phi\in \emph{C}_0^{\infty}\big(Q\big)
 	\end{equation}
 	where $|\alpha|:=\sum_{i=1}^{d}\alpha_i$ and $D^{\alpha}$ is the derivative in classical multi-index notation.
 \end{definition}
 \begin{definition}[Sobolev Space]
 	\label{def:Sobolev_Space}
 	Given $p\in[1,+\infty]$, $k\in \mathbb{N}$ and an open set $Q\subset \mathbb{R}^d$, we define the \emph{Sobolev Space} $\emph{W}^{\emph{p,k}}\big(Q\big)\subset \emph{L}^{\emph{p}}\big(Q\big)$ as 
 	\begin{equation}
 		\label{eq:Sobolev_Space}
 		\emph{W}^{\emph{p,k}}\big(Q\big):=\Bigg\{u\in \emph{L}^{\emph{p}}\big(Q\big)\hspace{0.10cm}\bigg|\hspace{0.10cm}\exists \hspace{0.10cm}D^{\alpha}u \in \emph{L}^{\emph{p}}\big(Q\big) \hspace{0.15cm} \forall 0\leqslant |\alpha|\leqslant k\Bigg\}.
 	\end{equation}
 \end{definition} 
 \begin{definition}[Sobolev Norm]
 	\label{def:norm_sobolev}
 	Given $p\in[1,+\infty]$, $k\in \mathbb{N}$ and an open set $Q\subset \mathbb{R}^d$, we define over the space \eqref{eq:Sobolev_Space} the following norm
 	\begin{equation}
 		\label{eq:norm_sobolev}
 		\|u\|_{\emph{W}^{\emph{p,k}}\big(Q\big)}:=\Bigg(\sum_{|\alpha|\leqslant k}\|D^{\alpha}u\|^{\emph{p}}_{\emph{L}^{\emph{p}}\big(Q\big)}\Bigg)^{\frac{1}{p}}.
 	\end{equation}
 \end{definition}
  When $p=2$, we use the notation $\text{H}^{\text{k}}\big(Q\big):=\text{W}^{\text{2,k}}\big(Q\big)$ to indicate that in this case $\text{W}^{\text{2,k}}\big(Q\big)$ is a Hilbert Space and we write the norm \eqref{eq:norm_sobolev} as $\|\cdot\|_{\text{H}^{\text{k}}\big(Q\big)}:=\|\cdot\|_{\text{W}^{\text{2,k}}\big(Q\big)}$. Furthermore, from the classical results of Sobolev Spaces, recall that
 \begin{equation*}
 	\text{W}^{\text{p,k}}\big(Q\big)=\overline{\text{C}^{\infty}\big(Q\big)}^{\|\cdot\|_{\text{W}^{\text{p,k}}\big(Q\big)}}.
 \end{equation*}
 \begin{definition}[Space of Hölder Continuous Functions] 
 	\label{df:holder_functions_space}
 	Let $Q\subset\mathbb{R}^d$ be an open set and $\gamma\in (0,1]$. The \emph{Space of Hölder Continuous Functions} $\emph{C}_{\emph{hol}}^{0,\gamma}\big(\overline{Q}\big)$ is the set of functions $u\in \emph{C}\big(\overline{Q}\big)$ such that 
 	\begin{equation*}
 		\big[u\big]_{\emph{C}_{\emph{hol}}^{0,\gamma}(\overline{Q})}:=\sup_{x\ne y\in \overline{Q}}\frac{\big|u\big(x\big)-u\big(y\big)\big|}{\|x-y\|^{\gamma}}<\infty.
 	\end{equation*}
 	In addition, the \emph{Hölder Norm} is defined by 
 	\begin{equation*}
 		\|u\|_{\emph{C}_{\emph{hol}}^{0,\gamma}\big(\overline{Q}\big)}:=\|u\|_{\emph{C}\big(\overline{Q}\big)}+\big[u\big]_{\emph{C}_{\emph{hol}}^{0,\gamma}\big(\overline{Q}\big)}.
 	\end{equation*}
 	where 
 	\begin{equation*}
 		\|u\|_{\emph{C}\big(\overline{Q}\big)}:=\sup_{x\in \overline{Q}}\big|u\big(x\big)\big|.
 	\end{equation*}
 \end{definition}
 \begin{definition}[Norms of the Spaces $\text{C}^k\big(\overline{Q}\big)$ and $\text{C}_{\text{hol}}^{k,\gamma}\big(\overline{Q}\big)$]
 	Given an open set $Q\subset\mathbb{R}^d$ and a multi-index $\alpha\in \mathbb{N}^d/\{0\}$, we define the \emph{Norm} of the space $\emph{C}^k\big(\overline{Q}\big)$ by
 	\begin{equation*}
 		\|u\|_{\emph{C}^k\big(\overline{Q}\big)}:= \|u\|_{\emph{C}\big(\overline{Q}\big)}+\sum_{|\alpha|=1}^{k}\|D^{\alpha}u\|_{\emph{C}\big(\overline{Q}\big)}.
 	\end{equation*}
 	In addition, the \emph{Norm} of space $\emph{C}_{\emph{hol}}^{k,\gamma}\big(\overline{Q}\big)$ with $\gamma\in(0,1]$ of the functions $u\in \emph{C}^k\big(\overline{Q}\big)$ such that $\big[D^ku\big]_{\emph{C}_{\emph{hol}}^{0,\gamma}\big(\overline{Q}\big)}<\infty$ is defined by
 	\begin{equation*}
 		\|u\|_{\emph{C}_{\emph{hol}}^{k,\gamma}\big(\overline{Q}\big)}:=\|u\|_{\emph{C}^k\big(\overline{Q}\big)}+\big[D^ku\big]_{\text{C}_{\emph{hol}}^{0,\gamma}\big(\overline{Q}\big)}.
 	\end{equation*}
 \end{definition}
 \begin{theorem}[\protect{\cite[Theorem 10.4.10]{Krylov}}]
 	\label{tm:rellich_kondrachov}
 	Let $Q\subset \mathbb{R}^d$ be a domain with boundary of class $\emph{C}^k$ with $k\in \mathbb{N}$. Assume that
 	\begin{equation*}
 		p>1,\quad k-\frac{d}{p}>0\quad\text{and}\quad k-\frac{d}{p} \text{ is non-integer}.
 	\end{equation*}
 	Then
 	\begin{equation*}
 		\emph{W}^{\emph{p,k}}\big(Q\big)\subset \emph{C}_{\emph{hol}}^{0,k-\frac{d}{p}}\big(Q\big)
 	\end{equation*}
 	in the sense that any $u\in \emph{W}^{\emph{p,k}}\big(Q\big)$ has a continuous modification which we will denote again by $u$ and such that $u\in \emph{C}_{\emph{hol}}^{0,k-\frac{d}{p}}\big(\overline{Q}\big)$ and
 	\begin{equation*}
 		\|u\|_{\emph{C}_{\emph{hol}}^{0,k-\frac{d}{p}}\big(\overline{Q}\big)}\leqslant C \|u\|_{\emph{W}^{\emph{p,k}}\big(Q\big)}
 	\end{equation*}
 	with $C>0$ a constant independent of $u$.
 \end{theorem}
 
  When $k=1$ the Theorem \ref{tm:rellich_kondrachov} is known in the literature as \emph{Rellich-Kondrachov}. 
 \begin{theorem}[\protect{\cite[Theorem 7.19]{Salsa}}]
 	\label{tm:embedded_H_m_C_k}
 	Let $Q\subset \mathbb{R}^d$ be a bounded domain with Lipschitz boundary and $\widetilde{m}>d/2$. Then 
 	\begin{equation*}
 		\emph{H}^{\widetilde{m}}\big(Q\big)\hookrightarrow \emph{C}^k\big(\overline{Q}\big)\quad 0\leqslant k<\widetilde{m}-\frac{d}{2}
 	\end{equation*}
 	with compact embedding. Moreover
 	\begin{equation}
 		\label{eq:embedding_H_m_C_k}
 		\|u\|_{\emph{C}^k\big(\overline{Q}\big)}\leqslant \widetilde{\Lambda}_0\big(d,k,Q\big)\|u\|_{\emph{H}^{\widetilde{m}}\big(Q\big)}\quad 0\leqslant k<\widetilde{m}-\frac{d}{2}.
 	\end{equation}
 \end{theorem}
 \begin{theorem}[Arzelá-Ascoli for the space $\text{C}_{\text{hol}}^{k.\gamma}\big(\overline{Q}\big)$ \protect{\cite[Hypothesis H8 together with Theorem 1.7]{RealOton}}]
 	\label{tm:arzela_ascoli_espaco_holder}
 	Given $\gamma\in(0,1]$, let $Q\subset\mathbb{R}^d$ be an open set and consider $\big(u_n\big)_{n\geqslant0}$ a sequence of functions, satisfying
 	\begin{equation*}
 		\|u_n\|_{\emph{C}_{\emph{hol}}^{\emph{k},\gamma}\big(\overline{Q}\big)}\leqslant C_0
 	\end{equation*}
 	where $C_0>0$ is a constant independent of $n$. Then there exists a subsequence $\big(u_{n_i}\big)_{i\geqslant0}$ and a function $u\in \emph{C}_{\emph{hol}}^{k,\gamma}\big(\overline{Q}\big)$, such that this sequence converges uniformly to $u$ in $\emph{C}^k\big(\overline{Q}\big)$.
 \end{theorem}
 
  For the results we are going to prove, we need to use various theorems and lemmas that use the second-order elliptic differential operator in different forms which, due to the regularity we impose on the coefficients, are all equivalent. The first one is in the divergent (div) form 
 \begin{equation}
 	\label{eq:differential_operator_second_order_divergent_form}
 	\mathcal{L}_{\text{div}} w \big(y\big) := - \sum_{i=1}^{d} \partial_{y_i} \big( b_i \big(y\big) w \big(y\big) \big) + \frac{1}{2} \sum_{i,j=1}^{d} \partial_{y_i,y_j}^2 \big( a_{ij} \big(y\big) w \big(y\big) \big).
 \end{equation}
 
 For some theorems that we will use, the divergent form \eqref{eq:differential_operator_second_order_divergent_form} is not the most appropriate, and so we will also consider one other equivalent forms for this operator, namely
 \begin{equation}
 	\label{eq:differential_operator_second_order_non_divergent_form}
 	\mathcal{L}_{\text{ndiv}}w\big(y\big):=\frac{1}{2}\sum_{i,j}^{d}\widetilde{a}_{ij}\big(y\big)\partial^2_{y_i,y_j}w\big(y\big)+\sum_{i=1}^{d}\widetilde{b}_i\big(y\big)\partial_{y_i}w\big(y\big)+\widetilde{c}\big(y\big)w\big(y\big)
 \end{equation}
where
 \begin{equation*}
 	\widetilde{a}_{ij}\big(y\big):=a_{ij}\big(y\big),\quad\widetilde{b}_i\big(y\big):=\sum_{j=1}^{d}\partial_{y_i}a_{ij}\big(y\big)-b_i\big(y\big)\quad\text{and}\quad
 	\widetilde{c}\big(y\big):=\frac{1}{2}\sum_{i,j}^{d}\partial^2_{y_i,y_j}a_{ij}\big(y\big)-\sum_{i=1}^d\partial_{y_i}b_i\big(y\big).
 \end{equation*}
 
  We will also consider an auxiliary elliptic PDE that will sometimes be defined on the torus and sometimes on open domains
  \begin{equation}
  	\label{eq:pde_second_order_elliptic_divergent_form}
  	-\mathcal{L}_{\text{div}}w\big(y\big)=f\big(y\big).
  \end{equation}
  
  When we work with the Sobolev Space $\text{H}^1$, the suitable way to define the notion of weak solution, for example, for PDE \eqref{eq:pde_second_order_elliptic_divergent_form}, is as follows. 
 
 Consider the bilinear form $a:\text{H}^1\big(Q\big)\times\text{H}^1\big(Q\big)\to \mathbb{R}$ defined by 
 \begin{equation}
 	\label{eq:bilinear_form_PDE}
 	a\big(u,v\big):=\int_{Q}\Big(\biggl\langle \overline{A}\big(y)\nabla u\big(y\big),\nabla v\big(y\big)\biggr\rangle+ \bigg\langle \overline{b}\big(y\big),\nabla u\big(y\big)\bigg\rangle v\big(y\big)+\overline{c}\big(y\big)u\big(y\big)v\big(y\big)\Big)\,dy\quad\forall u,v\in \text{H}^1\big(Q\big)
 \end{equation}
 where
 \begin{equation*}
 	\overline{A}\big(y\big):=\bigg[\widetilde{a}_{ij}\big(y\big)\bigg],\quad \overline{b}\big(y\big):=\Bigg[\frac{1}{2}\sum_{j=1}^{d}\partial_{y_i}a_{ij}\big(y\big)-\widetilde{b}_i\big(y\big)\Bigg]\quad\text{and}\quad \overline{c}\big(y\big):=-\widetilde{c}\big(y\big).
 \end{equation*}
 
 The bilinear form $a\big(\cdot,\cdot\big)$ is associated with the operator $	-\mathcal{L}_{\text{div}}$ by integration by parts. Thus, we can rewrite the PDE \eqref{eq:pde_second_order_elliptic_divergent_form} as follows 
 \begin{equation*}
 	a\big(u,v\big)=F\big(v\big)
 \end{equation*}
 where the functional $F:\text{H}^1\big(Q\big)\to \mathbb{R}$ is defined by 
 \begin{equation*}
 	F\big(v\big):=\int_{Q}f\big(y\big)v\big(y\big)dy
 \end{equation*}
 with $f\in \text{L}^2\big(Q\big)$. 
 \begin{definition}[Weak Solution in $\text{H}^{\text{1}}\big(Q\big)$]
 	\label{def:weak_solution_H}
 	Given $f\in \emph{L}^2\big(Q\big)$, we say that $w\in \emph{H}^1\big(Q\big)$ is an \emph{Weak Solution} of the PDE \eqref{eq:pde_second_order_elliptic_divergent_form} if 
 	\begin{equation}
 		\label{eq:variational_form_pde}
 		a\big(w,v\big)=F\big(v\big)\quad \forall v\in \emph{H}^1\big(Q\big). 
 	\end{equation}	
  \end{definition}
  
 \begin{theorem}[Interior Regularity \protect{\cite[Theorem 8.12]{Salsa}}]
 	\label{tm:interior_regularity}
 	Let $Q\subset\mathbb{R}^d$ be a domain and $w\in \emph{H}^1\big(Q\big)$ a weak solution of the PDE \eqref{eq:pde_second_order_elliptic_divergent_form}. Furthermore, consider that the coefficients of the uniformly elliptic operator \eqref{eq:differential_operator_second_order_non_divergent_form} (with ellipticity constant $\tilde{c}_0>0$) are of class $\emph{C}^{\emph{m+1}}\big(Q\big)$ with $m\geqslant 1$ and bounded by a constant $\widetilde{C}_{A,b,c}>0$. Then $w\in \emph{H}^{\emph{m+2}}_{\emph{loc}}\big(Q\big)$ and for $\widetilde{Q}\subset\subset Q$ a subdomain of $Q$ we get 
 	\begin{equation}
 		\label{eq:interior_regularity}
 		\|w\|_{\text{\emph{H}}^{\emph{m+2}}\big(\widetilde{Q}\big)}\leqslant \widetilde{\Lambda}\big(d;\tilde{c}_0;\widetilde{C}_{A,b,c}\big)\Bigg(\|f\|_{\emph{H}^{\emph{m}}\big(Q\big)}+\|w\|_{\emph{L}^2\big(Q\big)}\Bigg).
 	\end{equation}
 \end{theorem}

  We will also need another inequality which can be found in \cite{BogachevKrylovRockner} page 20. Let $Q_{r_1}:=B\big(0,r_1\big)\subset \mathbb{R}^d$ be a ball of radius $r_1>0$, and consider that the PDE \eqref{eq:pde_second_order_elliptic_divergent_form} admits a solution in $\text{W}^{\text{p,2}}\big(Q_{r_1}\big)$ (in which case we have a strong solution) with $f\in \text{L}^{\text{p}}\big(Q_{r_1}\big)$ and, consider that the hypotheses in item II are satisfied. Then, for $0<r<r_1$ and $p>d$, the following inequality holds
 \begin{equation}
 	\label{eq:auxiliary_inequality_sobolev}
 	\|w\|_{\text{W}^{\text{p,1}}\big(Q_r\big)}\leqslant \Lambda_{\text{d,r}}\big(c_0;C_{A,b,c};r_1\big)\Bigg(\|w\|_{\text{L}^1\big(Q_r\big)}+\|f\|_{\text{L}^{\text{p}}\big(Q_r\big)}\Bigg)
 \end{equation} 
 where $c_0>0$ and $C_{A,b,c}>0$ are constants such that:
 \begin{equation*}
 	A\geqslant c_0 I\quad\text{and}\quad\|A\|_{\text{W}^{\text{p,1}}\big(Q_r\big)}+\|b\|_{\text{L}^1\big(Q_r\big)}+\|c\|_{\text{L}^{\text{p}}\big(Q_r\big)}\leqslant C_{A,b,c}.
 \end{equation*} 
 
  \subsection{Continuity in Relation to the Slow Scale Parameter}
   We now define the notion of non-classical solution of the Fokker-Planck \eqref{eq:fokker_planck_PDE}. The definition adopted here is a particular case of Definition 1.4.1 in \cite[Chapter 1]{BogachevKrylovRocknerShaposhnikov} which can be used when we have enough regularity to work with probability densities instead of measures.
 \begin{definition}[Solution of the Fokker-Planck]
 	\label{def:non-classical_fokker_planck_solution}
 	A density $\rho:\mathbb{T}^{d_Y}\to \mathbb{R}_{+}$ satisfies the Fokker-Planck \eqref{eq:fokker_planck_PDE} if the following equality holds
 	\begin{equation*}
 		\int_{\mathbb{T}^{d_Y}}\mathcal{L}^{\bar{x}}\phi\big(y\big) \rho\big(y\big)\,dy=0\quad \forall \phi \in \emph{C}^{\infty}\big(\mathbb{T}^{d_Y}\big).
 	\end{equation*}
 \end{definition}
 
 Note that when $\rho$ is a classical solution of the Fokker-Planck \eqref{eq:fokker_planck_PDE}, we obtain, in $\text{L}^2\big(\mathbb{T}^{d_Y}\big)$, the following equality 
 \begin{equation*}
 	\int_{\mathbb{T}^{d_Y}}\mathcal{L}^{\bar{x}}\phi\big(y\big) \rho\big(y\big)\,dy=\int_{\mathbb{T}^{d_Y}}\phi\big(y\big) \mathcal{L}^{\bar{x},*}\rho\big(y\big)\,dy=0\quad \forall \phi \in \text{C}^{\infty}\big(\mathbb{T}^{d_Y}\big).
 \end{equation*}
 
 Moreover the operator $\mathcal{L}^{\bar{x},*}$ associated with Fokker-Planck \eqref{eq:fokker_planck_PDE} is in divergent form and therefore
 \begin{equation*}
 	\mathcal{L}^{\bar{x},*}=\mathcal{L}^{\bar{x}}_{\text{div}}.
 \end{equation*}

 \begin{lemma}[Continuity of Fokker-Planck in relation to the slow variable]
 	\label{lm:continuity_Fokker_Planck_relation_slow_variable}
 	Assume that hypotheses II and VI hold. Then the solution of the Fokker-Planck \eqref{eq:fokker_planck_PDE} is continuous with respect to the slow variable.  
 \end{lemma}
 \begin{proof}
 	Having fixed $\bar{x}\in \mathbb{X}$, from item II, it follows, by Theorem \ref{tm:ergodicity_markov_process}, that the PDE \eqref{eq:fokker_planck_PDE} admits a unique classical solution. We denote this solution by
 	\begin{equation}
 		\label{eq:fokker_planck_classical_solution}
 		\rho\big(\bar{x},y\big)=\bigg(\big[\mathcal{L}^{\bar{x},*}\big]^{-1}\big(0\big)\bigg)\big(y\big)\quad \forall y\in\mathbb{T}^{d_Y}.
 	\end{equation}
 	Note that \eqref{eq:fokker_planck_classical_solution} is the solution of the PDE \eqref{eq:fokker_planck_PDE} when we consider the unitary cube $\mathbb{Y}$ with the periodic boundary condition. 
 	In this case, we can extend (by periodicity) this solution to the whole $\mathbb{R}^{d_Y}$ in such a way that the new function continues to satisfy the PDE \eqref{eq:fokker_planck_PDE}. For notational convenience, we will continue to denote this extension by $\rho$.
 	
 	Given a sequence $h_n\downarrow 0$, such that $\bar{x}+vh_n\in \text{Int}\big(\mathbb{X}\big)$, where $v$ is a unity vector in $\mathbb{R}^{d_X}$, we take a ball $Q_r:=B\big(\big(\sqrt{d_Y}/2,...,\sqrt{d_Y}/2\big),r\big)\subset \mathbb{R}^{d_Y}$ of radius $r>4\sqrt{d_Y}$.  Thus, we obtain that  
 	\begin{equation*}
 	  \mathcal{L}^{\bar{x}+vh_n,*}\rho\big(\bar{x}+vh_n,y\big)=0\quad\forall y\in Q_r.
 	\end{equation*}
 	Consider $r_1>r$. From the inequality \eqref{eq:auxiliary_inequality_sobolev}, it follows that 
 	\begin{equation*}
 		\|\rho\big(\bar{x}+vh_n,\cdot\big)\|_{\text{W}^{\text{p,1}}\big(Q_r\big)}
 		\leqslant\Lambda_{d_Y,r}\big(c_0;C_{A,b,c};r_1\big)\|\rho\big(\bar{x}+vh_n,\cdot\big)\|_{\text{L}^1\big(Q_r\big)}.
 	\end{equation*}
 	Since $\rho$ is a density, we get that
 	\begin{equation*}
 		\|\rho\big(\bar{x}+vh_n,\cdot\big)\|_{\text{L}^1\big(Q_r\big)}=\int_{Q_r}\rho\big(\bar{x}+vh_n,y\big)dy\leqslant C_{\mathbb{Y}} \int_{\mathbb{Y}}\rho\big(\bar{x}+vh_n,y\big)dy=C_{\mathbb{Y}},
 	\end{equation*}
 	for a constant $C_{\mathbb{Y}}>0$ proportional to the number of cubes $\mathbb{Y}$ we use to cover the ball $Q_r$ such that the above inequality holds. Therefore,
 	\begin{equation}
 		\label{eq:auxiliary_inequality_24}
 		\|\rho\big(\bar{x}+vh_n,\cdot\big)\|_{\text{W}^{\text{p,1}}\big(Q_r\big)}\leqslant \Lambda_{d_Y,r}\big(c_0;C_{A,b,c};r_1\big)C_{\mathbb{Y}}.
 	\end{equation}
 	On the other hand, by Hölder inequality, for $1\leqslant q\leqslant p\leqslant\infty$
 	\begin{equation}
 		\label{eq:auxiliary_inequality_from_p_q}
 		\|\rho\big(\bar{x}+vh_n,\cdot\big)\|_{\text{L}^{\text{q}}\big(Q_r\big)}\leqslant \lambda\big(Q_r\big)^{\frac{1}{q}-\frac{1}{p}}\|\rho\big(\bar{x}+vh_n,\cdot\big)\|_{\text{L}^{\text{p}}\big(Q_r\big)}.
 	\end{equation}
 	Taking $q=2$ and $p\geqslant2$ in the inequality \eqref{eq:auxiliary_inequality_from_p_q} and combining it with the inequality \eqref{eq:auxiliary_inequality_24}, it follows that 
 	\begin{equation}
 		\label{eq:auxiliary_inequality_25}
 		\|\rho\big(\bar{x}+vh_n,\cdot\big)\|_{\text{L}^2\big(Q_r\big)}\leqslant \lambda\big(Q_r\big)^{\frac{1}{2}-\frac{1}{p}}\Lambda_{d_Y,r}\big(c_0;C_{A,b,c};r_1\big)C_{\mathbb{Y}}=:K_0.
 	\end{equation}
 	From item VI, the drift and dispersion coefficients are of class $\text{C}^{1,m+2}\big(\mathbb{X}\times\mathbb{T}^{d_Y}\big)$ with $m>d_Y/2$, so by the inequality \eqref{eq:interior_regularity} (in the theorem \ref{tm:interior_regularity}) together with the inequality \eqref{eq:auxiliary_inequality_25} it follows that 
 	\begin{equation}
 		\label{eq:auxiliary_inequality_26}
 		\|\rho\big(\bar{x}+vh_n,\cdot\big)\|_{\text{H}^{m+3}\big(Q_{\frac{r}{2}}\big)}\leqslant  \widetilde{\Lambda}\big(d_Y;\widetilde{c}_0;\widetilde{C}_{A,b,c}\big)K_0.
 	\end{equation}
 	Combining the inequality \eqref{eq:embedding_H_m_C_k} for $k=3$ and $\widetilde{m}=m+3$ (in the Theorem \ref{tm:embedded_H_m_C_k}) with the inequality \eqref{eq:auxiliary_inequality_26}, we have that
 	\begin{equation}
 		\label{eq:auxiliary_inequality_27}
 			\|\rho\big(\bar{x}+vh_n,\cdot\big)\|_{\text{C}_{\text{hol}}^{2,1}\big(\overline{Q}_{\frac{r}{2}}\big)}\leqslant \widetilde{\Lambda}_0\big(d_Y;2;Q_{\frac{r}{2}}\big) \widetilde{\Lambda}\big(d_Y;\widetilde{c}_0;\widetilde{C}_{A,b,c}\big)K_0=:K_1.
 	\end{equation}
 	From Theorem \ref{tm:arzela_ascoli_espaco_holder} (Arzelá-Ascoli), there exists a subsequence $\big(\rho\big(\bar{x}+vh_{n_i},\cdot\big)\big)_{i\geqslant0}$ and a function $\varrho\big(\bar{x}, \cdot\big) \in \text{C}_{\text{hol}}^{2,1}\big(\overline{Q}_{\frac{r}{2}}\big)$ such that this subsequence converges uniformly to $\varrho\big(\bar{x},\cdot\big)$ in $\text{C}^2\big(\overline{Q}_{\frac{r}{2}}\big)$. Since the convergence is uniform $\varrho\big(\bar{x},\cdot\big)$ is also non-negative and periodic. Therefore we can consider that $\varrho\big(\bar{x},\cdot\big)$ is defined on the $\mathbb{T}^{d_Y}$. On the one hand, by  Definition \ref{def:non-classical_fokker_planck_solution}, we have that 
 	\begin{equation*}
 		\int_{\mathbb{T}^{d_Y}}	\mathcal{L}^{\bar{x}+vh_{n_i}}\phi\big(y\big)\rho\big(\bar{x}+vh_{n_i},y\big)dy=0\quad\forall \phi \in \text{C}^{\infty}\big(\mathbb{T}^{d_Y}\big).
 	\end{equation*}
 	On the other hand, from item II, it follows that
 	\begin{equation*}
 		\lim_{i\to+\infty}	\mathcal{L}^{\bar{x}+vh_{n_i}}\phi\big(y\big)=\mathcal{L}^{\bar{x}}\phi\big(y\big)\quad \forall y\in \mathbb{T}^{d_Y}.
 	\end{equation*}
 	Furthermore, by item II and given that $\phi \in \text{C}^{\infty}\big(\mathbb{T}^{d_Y}\big)$, there exists a constant $C_{A,b,c,\phi}>0$ (independent of $\bar{x}+vh_{n_i}$) such that
 	\begin{equation*}
 		\big|\mathcal{L}^{\bar{x}+vh_{n_i}}\phi\big(y\big)\big|\leqslant C_{A,b,c,\phi}\quad \forall y\in \mathbb{T}^{d_Y}.
 	\end{equation*}
 	Combining this bound with the inequality \eqref{eq:auxiliary_inequality_27}, we obtain by the dominated convergence theorem that 
 	\begin{equation*}
 		\int_{\mathbb{T}^{d_Y}}	\mathcal{L}^{\bar{x}}\phi\big(y\big)\varrho\big(\bar{x},y\big)dy=0\quad\forall \phi \in \text{C}^{\infty}\big(\mathbb{T}^{d_Y}\big).
 	\end{equation*}
 	As $\varrho\big(\bar{x},\cdot\big) \in \text{C}^2\big(\mathbb{T}^{d_Y}\big)$ it follows, from the above equality, that 
 	\begin{equation*}
 		\int_{\mathbb{T}^{d_Y}}	\phi\big(y\big)\mathcal{L}^{\bar{x},*}\varrho\big(\bar{x},y\big)dy=0\quad\forall \phi \in \text{C}^{\infty}\big(\mathbb{T}^{d_Y}\big).
 	\end{equation*}
 	Which implies
 	\begin{equation*}
 		\mathcal{L}^{\bar{x},*}\varrho\big(\bar{x},y\big)=0\quad \forall y\in \mathbb{T}^{d_Y}.
 	\end{equation*}
 	By the inequality \eqref{eq:auxiliary_inequality_27} together with the dominated convergence theorem we obtain that 
 	\begin{equation*}
 		1=\lim_{i\to+\infty}\int_{\mathbb{T}^{d_Y}}\rho\big(\bar{x}+vh_{n_i},y\big)dy=\int_{\mathbb{T}^{d_Y}}\varrho\big(\bar{x},y\big)dy.
 	\end{equation*}
 	Therefore, $\varrho\big(\bar{x},\cdot\big)$ is a probability density and a classical solution of the Fokker-Planck \eqref{eq:fokker_planck_PDE}. Moreover, we obtain by uniqueness of classical solutions for the Fokker-Planck \eqref{eq:fokker_planck_PDE} on the torus $\mathbb{T}^{d_Y}$ that $\varrho\equiv\rho$. By a classical analysis argument, the sequence $\big(\rho\big(\bar{x}+vh_n,\cdot\big)\big)_{n\geqslant0}$ converges uniformly to $\rho\big(\bar{x},\cdot\big)$. Furthermore, the sequences $\big(\partial_{y_i}\rho\big(\bar{x}+vh_n,\cdot\big)\big)_{n\geqslant0}$ and $\big(\partial^2_{y_i,y_j}\rho\big(\bar{x}+vh_n, \cdot\big)\big)_{n\geqslant0}$ also converges uniformly to $\partial_{y_i}\rho\big(\bar{x},\cdot\big)$ and $\partial^2_{y_i,y_j}\rho\big(\bar{x},\cdot\big)$ respectively. We conclude that the applications $x\mapsto\rho\big(x,\cdot\big)$, $x\mapsto\partial_{y_i}\rho\big(x,\cdot\big)$ and $x\mapsto\partial^2_{y_iy_j}\rho\big(x,\cdot\big)$ are continuous.
 \end{proof}
 \subsection{Differentiability in Relation to the Slow Scale Parameter}
  Now we want to investigate the differentiability of the Fokker-Planck \eqref{eq:fokker_planck_PDE} with respect to the slow variable and for this we will consider the auxiliary PDE \eqref{eq:pde_second_order_elliptic_divergent_form} on the torus $\mathbb{T}^d$.  We want to prove the existence and uniqueness of periodic solutions to this PDE and to do this we need to define the space in which this solution will be searched. We will use Sobolev Space for periodic functions in an approach based on the works \cite{CioranescuDonato} and \cite{PavliotisStuart}.  Thus, the most appropriate space for proving the existence of solutions is the space  $\text{H}^{\text{k}}\big(\mathbb{T}^d\big)$ defined by
 \begin{equation*}
 	\text{H}^{\text{k}}\big(\mathbb{T}^d\big):=\overline{\text{C}^{\infty}\big(\mathbb{T}^d\big)}^{\|\cdot\|_{\text{W}^{\text{2,k}}}}.
 \end{equation*}
 More precisely, we will look for weak solutions and thus we consider $k=1$. In addition, note that for $p=2$ and $k=1$ the morm \eqref{eq:norm_sobolev} reduces to 
 \begin{equation}
 	\label{eq:periodic_sobolev_hilbert_space_norm_with_constants}
 	\|u\|_{\text{H}^{\text{1}}\big(\mathbb{T}^d\big)}=\sqrt{\|u\|^{\text{2}}_{\text{L}^{\text{2}}\big(\mathbb{T}^d\big)}+\|\nabla u\|^{\text{2}}_{\text{L}^{\text{2}}\big(\mathbb{T}^d\big)}}.
 \end{equation}
 where $\text{L}^{\text{2}}\big(\mathbb{T}^d\big):=\overline{\text{C}^{\infty}\big(\mathbb{T}^d\big)}^{\|\cdot\|_{\text{L}^2}}$.
 
As we are considering a problem with a periodic boundary condition, we need to consider a space in which the only constant function is the identically null one. It is therefore necessary to consider the following subspace of $\text{H}^{\text{1}}\big(\mathbb{T}^d\big)$
 \begin{equation*}
 	\mathfrak{H}^{\text{1}}\big(\mathbb{T}^d\big):=\Bigg\{u\in\text{H}^{\text{1}}\big(\mathbb{T}^d\big):\int_{\mathbb{T}^d}u\big(y\big)dy=0 \Bigg\}.
 \end{equation*}
 In this case, the only constant functions of $\mathfrak{H}^{\text{1}}\big(\mathbb{T}^d\big)$ are the identically zero ones. The dual of $\mathfrak{H}^{\text{1}}\big(\mathbb{T}^d\big)$ has the following representation
 \begin{equation}
 	\label{eq:dual_periodic_sobolev_hilbert_space_without_constants}
 	\big(\mathfrak{H}^{\text{1}}\big(\mathbb{T}^d\big)\big)^{*}=\Biggl\{F\in \big(\text{H}^{\text{1}}\big(\mathbb{T}^d\big)\big)^{*}: F\big(\textbf{1}\big)=0\Biggr\}.
 \end{equation}
 Finally, another space we will need is 
 \begin{equation*}
 	\mathfrak{L}^2\big(\mathbb{T}^d\big):=\Biggl\{u\in\text{L}^2\big(\mathbb{T}^d\big):\int_{\mathbb{T}^d}u\big(y\big)dy=0 \Biggr\}.
 \end{equation*}
 
  Our strategy for proving the existence of solutions to PDE \eqref{eq:fokker_planck_PDE} is to use the \emph{Fredholm Alternative} and the \emph{Weak Maximum Principle}.
 \begin{theorem}[Fredholm Alternative -  \protect{\cite[Theorem 7.9]{PavliotisStuart}}]
 	\label{tm:fredholm_alternative}
 	Assume that the operator $-\mathcal{L}_{\emph{div}}$ is uniformly elliptic and that its coefficients are of class $\emph{C}^1\big(\mathbb{T}^d\big)$. Then only one of the alternatives holds:
 	\begin{itemize}
 		\item For every $f\in \mathfrak{L}^2\big(\mathbb{T}^d\big)$, there is a unique weak solution $w\in \mathfrak{H}^1\big(\mathbb{T}^d\big)$ of the PDE
 	    \begin{equation}
 			\label{eq:non-homogeneous_auxiliary_PDE}
 			-\mathcal{L}_{\emph{div}}w\big(y\big)=f\big(y\big)\quad \forall y\in\mathbb{T}^d.
 		\end{equation} 
 		\item The homogeneous equation 
 		\begin{equation}
 			\label{eq:homogeneous_auxiliary_PDE}
 		-\mathcal{L}_{\emph{div}}w\big(y\big)=0 \quad \forall y\in \mathbb{T}^d
 		\end{equation}
 		has at least one non-trivial solution in $\mathfrak{H}^1\big(\mathbb{T}^d\big)$.
 	\end{itemize}
 \end{theorem}
 
  By theorem \ref{tm:fredholm_alternative}, in order to solve  \eqref{eq:pde_second_order_elliptic_divergent_form} we need to show that \eqref{eq:homogeneous_auxiliary_PDE} admits only the null solution in the space $\mathfrak{H}^1\big(\mathbb{T}^d\big)$. With this in mind, we present the Weak Maximum Principle below.
 \begin{theorem}[Weak Maximum Principle - \protect{\cite[Theorem 2 from Section 6.4]{Evans}}]
 	\label{tm:weak_maximum_principle}
 	Consider $Q\subset\mathbb{R}^d$ a bounded open set, $w\in \emph{C}^2\big(Q\big)\cap \emph{C}\big(\overline{Q}\big)$ and $\widetilde{c}\big(y\big)\leqslant 0$ for all $y\in Q$. Then we have the following cases:
 	\begin{itemize}
 		\item If $-\mathcal{L}_{\text{div}}w\big(y\big)\leqslant 0$ for all $y\in Q$, then 
 		\begin{equation*}
 			\max_{\overline{Q}}w\big(y\big)\leqslant \max_{\partial Q}w^{+}\big(y\big).
 		\end{equation*}
 		\item If $-\mathcal{L}_{\text{div}}w\big(y\big)\geqslant 0$ for all $y\in Q$, then 
 		\begin{equation*}
 			\min_{\overline{Q}}w\big(y\big)\geqslant -\max_{\partial Q}w^{-}\big(y\big).
 		\end{equation*}
 	\end{itemize}
 	In particular, if $-\mathcal{L}_{\text{div}}w\big(y\big)=0$ in $y\in Q$, then 
 	\begin{equation*}
 		\max_{\overline{Q}}\big|w\big(y\big)\big|=\max_{\partial Q}\big|w\big(y\big)\big|.
 	\end{equation*}
 \end{theorem}
  \begin{lemma}[Existence and Uniqueness of solutions for PDE \eqref{eq:pde_second_order_elliptic_divergent_form} in $\mathfrak{H}^1\big(\mathbb{T}^d\big)$]
 	\label{lm:existence_unicity_periodic_solutions_EDP}
 	Assume that the operator \eqref{eq:differential_operator_second_order_divergent_form} is uniformly elliptic and that the coefficients $A$ and $b$ are periodic and of class $\emph{C}^2\big(\mathbb{T}^d\big)$. Furthermore, consider that 
 	\begin{equation*}
 	   \widetilde{c}\big(y\big)\leqslant0 \quad\forall y\in\mathbb{T}^d.
 	\end{equation*}
 	 Then for every $f\in \mathfrak{L}^{2}\big(\mathbb{T}^d\big)$, there is a unique weak solution to the PDE \eqref{eq:pde_second_order_elliptic_divergent_form} in $\mathfrak{H}^1\big(\mathbb{T}^d\big)$. 
 \end{lemma}
 \begin{proof}
  	Consider $\psi\in \text{C}^2\big(\mathbb{T}^d\big)$ such that
 	\begin{equation}
 		\label{eq:toro_auxiliary_PDE_01}
 		-\mathcal{L}_{\text{div}}\psi\big(y\big)=0\quad \forall y\in\mathbb{T}^d.
 	\end{equation}
 	The PDE \eqref{eq:toro_auxiliary_PDE_01} can be rewritten as
 	\begin{subequations}
 		\label{subeq:PDE_auxiliary_cube_periodic_boundary_02}
 		\begin{align}
 			\label{eq:PDE_auxiliary_cube_02}
 			-\mathcal{L}_{\text{div}}\psi&\big(y\big)=0\hspace{0.18cm} \forall y\in\mathbb{Y}\\
 			\label{eq:PDE_auxiliary_periodic_boundary_cube 02}
 			\psi&\hspace{0.18cm} \mathbb{Y}\text{-Periodic}. 
 		\end{align}
 	\end{subequations}
 	 By Theorem \ref{tm:weak_maximum_principle} (Weak Maximum Principle) it follows that
 	\begin{equation*}
 		\max_{\mathbb{Y}}\big|\psi\big(y\big)\big|=\max_{\partial\mathbb{Y}}\big|\psi\big(y\big)\big|.
 	\end{equation*}
 	Given $a\in\mathbb{R}^d$, consider the set $\mathbb{Y}_a:=a+\mathbb{Y}$. For every $z\in \mathbb{Y}_a$ there exist $k_{z}\in\mathbb{Z}^d$ and $\gamma_{z}\in\mathbb{Y}$ such that $z=\gamma_{z}+k_{z}$. This means, for all $z\in\mathbb{Y}_a$, that
 	\begin{equation*}
 	-\mathcal{L}_{\text{div}}\psi\big(z\big)=-\mathcal{L}_{\text{div}}\psi\big(\gamma_{z}+k_{z}\big)=-\mathcal{L}_{\text{div}}\psi\big(\gamma_{z}\big)=0.
 	\end{equation*}
 	Therefore, once again, according to the weak maximum principle, we have
 	\begin{equation*}
 	\max_{\mathbb{Y}_a}\big|\psi\big(z\big)\big|=\max_{\partial\mathbb{Y}_a}\big|\psi\big(z\big)\big|.
 	\end{equation*}
 	Since we can choose $a\in\mathbb{R}^d$ in such a way that the points on a part of the boundary $\partial\mathbb{Y}_a$ correspond to points on $\text{Int}\big(\mathbb{Y}\big)$ it follows that $\psi$ also reaches a maximum inside $\mathbb{Y}$ and is therefore a constant function. Thus, we obtain that the only solutions to the problem \eqref{subeq:PDE_auxiliary_cube_periodic_boundary_02} are constant functions. 
 	
 	We are interested in solutions to the problem \eqref{subeq:PDE_auxiliary_cube_periodic_boundary_02} in $\mathfrak{H}^1\big(\mathbb{T}^d\big)$. Since the only constant function in $\mathfrak{H}^1\big(\mathbb{T}^d\big)$ is the identically zero ones, it follows that, in this space, the problem \eqref{subeq:PDE_auxiliary_cube_periodic_boundary_02} has only the trivial solution. On the other hand, this implies that the homogeneous equation
 	\begin{equation*}
 		-\mathcal{L}_{\text{div}}w\big(y\big)=0 \quad \forall y\in \mathbb{T}^d
 	\end{equation*}
 	has only the trivial solution in $\mathfrak{H}^1\big(\mathbb{T}^d\big)$. Therefore, by Theorem \ref{tm:fredholm_alternative} there is only one solution in $\mathfrak{H}^1\big(\mathbb{T}^d\big)$ for
 	 \begin{equation*}
 		-\mathcal{L}_{\text{div}}w\big(y\big)=f\big(y\big)\quad \forall y\in\mathbb{T}^d
 	\end{equation*}
 	for all $f\in \mathfrak{L}^2\big(\mathbb{T}^d\big)$.
 \end{proof}
 
  We now present some results on estimates of the solutions of the PDE \eqref{eq:pde_second_order_elliptic_divergent_form} in relation to the initial condition $f$.
 \begin{lemma}[\protect{\cite[Theorem 6 from Section 6.2]{Evans}}]
 	\label{lm:first_energy_estimate}
 	Assume that $\lambda\notin \sigma\big(	-\mathcal{L}_{\emph{div}}\big)$ (the spectrum of the operator $	-\mathcal{L}_{\emph{div}}$), then there is a constant $C>0$ such that 
 	\begin{equation*}
 		\|w\|_{\mathfrak{L}^2\big(\mathbb{T}^d\big)}\leqslant C \|f\|_{\mathfrak{L}^2\big(\mathbb{T}^d\big)}
 	\end{equation*}
 	whenever $f\in\mathfrak{L}^2\big(\mathbb{T}^d\big)$ and $w\in \mathfrak{H}^1\big(\mathbb{T}^d\big)$ is the only solution of the PDE 
 	\begin{equation}
 		\label{eq:auxiliary_PDE_01}
 		-\mathcal{L}_{\emph{div}}w\big(y\big)=\lambda w\big(y\big)+f\big(y\big)\quad \forall y\in \mathbb{T}^d.
 	\end{equation}
 	The constant $C>0$ only depends on $\lambda$, $\mathbb{T}^d$ and the coefficients of $\mathcal{L}_{\emph{af}}$.
 \end{lemma}
 
 To prove this result we need the following auxiliary lemma.
 \begin{lemma}[\protect{\cite[Lemma 7.10]{PavliotisStuart}}]
 	\label{lm:auxiliary_EDP_periodic_energy_estimate}
 	Assume that the operator \eqref{eq:differential_operator_second_order_divergent_form} is uniformly elliptic (with ellipticity coefficient $c_0>0$). Then there is a constant $\mu_{\lambda}>0$ such that the following \emph{energy estimate} holds for PDE \eqref{eq:auxiliary_PDE_01}
 	\begin{equation}
 		\label{eq:auxiliary_PDE_periodic_energy_estimate}
 		\frac{c_0}{2}\|w\|_{\mathfrak{H}^1\big(\mathbb{T}^d\big)}^2\leqslant a_{\lambda}\big(w,w\big)+\mu_{\lambda}\|w\|^2_{\mathfrak{L}^2\big(\mathbb{T}^d\big)}.
 	\end{equation}
 \end{lemma}
 
 We then move on to the proof of the lemma \ref{lm:first_energy_estimate}.
\begin{proof}[Proof of Lemma \ref{lm:first_energy_estimate}]
	The argument is by contradiction. Thus, there would be sequences $\big(f_k\big)_{k\geq1}$ and $\big(w_k\big)_{k\geq 1}$ such that 
	\begin{equation*}
	-\mathcal{L}_{\text{div}}w_k\big(y\big)=\lambda w_k\big(y\big)+f_k\big(y\big)\quad \forall y\in \mathbb{T}^d
	\end{equation*}
	but 
	\begin{equation*}
		\|w_k\|_{\mathfrak{L}^2\big(\mathbb{T}^d\big)}\geqslant k \|f_k\|_{\mathfrak{L}^2\big(\mathbb{T}^d\big)}.
	\end{equation*}
	Assuming, without loss of generality, that $\|w_k\|_{\mathfrak{L}^2\big(\mathbb{T}^d\big)}=1$ we have that $f_k\to 0$ in $\mathfrak{L}^2\big(\mathbb{T}^d\big)$. On the one hand, from the energy estimate \eqref{eq:auxiliary_PDE_periodic_energy_estimate}, it follows that 
	\begin{equation}
		\label{eq:auxiliary_inequality_02}
		\frac{c_0}{2}\|w_k\|_{\mathfrak{H}^1\big(\mathbb{T}^d\big)}^2\leqslant \big|a_{\lambda}\big(w_k,w_k\big)\big|+\mu_{\lambda} \|w_k\|^2_{\mathfrak{L}^2\big(\mathbb{T}^d\big)}.
	\end{equation}
	On the other hand, by definition of a weak solution, we have that 
	\begin{equation*}
		a_{\lambda}\big(w_k,w_k\big)=\int_{\mathbb{T}^d}f_k\big(y\big)w_k\big(y\big)dy.
	\end{equation*}
	Using Hölder's inequality in the above equality, it follows that 
	\begin{equation}
		\label{eq:auxiliary_inequality_03}
		\big|a_{\lambda}\big(w_k,w_k\big)\big|\leqslant\int_{\mathbb{T}^d}\big|f_k\big(y)\big|\big|w_k\big(y\big)\big|dy\leqslant \|f_k\|_{\mathfrak{L}^2\big(\mathbb{T}^d\big)}\|w_k\|_{\mathfrak{L}^2\big(\mathbb{T}^d\big)}.
	\end{equation}
	Combining the inequalities \eqref{eq:auxiliary_inequality_02} and \eqref{eq:auxiliary_inequality_03}, we get 
	\begin{equation*}
		\frac{c_0}{2}\|w_k\|_{\mathfrak{H}^1\big(\mathbb{T}^d\big)}^2\leqslant \|f_k\|_{\mathfrak{L}^2\big(\mathbb{T}^d\big)}\|w_k\|_{\mathfrak{L}^2\big(\mathbb{T}^d\big)}+\mu_{\lambda} \|w_k\|^2_{\mathfrak{L}^2\big(\mathbb{T}^d\big)}\leqslant \frac{1}{k}+\mu_{\lambda}\leqslant 1+\mu_{\lambda}.
	\end{equation*}
	That is,
	\begin{equation*}
		\|w_k\|_{\mathfrak{H}^1\big(\mathbb{T}^d\big)}\leqslant \sqrt{ \frac{2}{c_0}\big(1+\mu_{\lambda}\big)}.
	\end{equation*}
	Since the sequence $\big(w_k\big)_{k\geq 1}$ is bounded in $\mathfrak{H}^1\big(\mathbb{T}^d\big)$ (which is compactly contained in $\mathfrak{L}^2\big(\mathbb{T}^d\big)$), there exists a subsequence $\big(w_{k_j}\big)_{j\geq 1}$ and a function $w$ such that 
	\begin{subequations}
		\begin{align}
			\label{eq:weak_convergence}
			w_{k_j}&\rightharpoonup w \text{ weakly in }\mathfrak{H}^1\big(\mathbb{T}^d\big),\\
			\label{eq:strong_convergence}
			w_{k_j}&\to w \text{ strongly in } \mathfrak{L}^2\big(\mathbb{T}^d\big).
		\end{align}
	\end{subequations}
	So $w$ is a weak solution of the following PDE
	\begin{equation*}
		-\mathcal{L}_{\text{div}}w\big(y\big)=\lambda w\big(y\big)\quad \forall y\in\mathbb{T}^d.
	\end{equation*}
	Since $\lambda \notin \sigma\big(-\mathcal{L}_{\text{div}}\big)$, there exists $\big(-\mathcal{L}_{\text{div}}-\lambda\big)^{-1}$ and so
	\begin{equation*}
		w\big(y\big)=\big(-\mathcal{L}_{\text{div}}-\lambda\big)^{-1}\big(0\big)=0.
	\end{equation*}
	However, by strong convergence \eqref{eq:strong_convergence}, we have that $\|w\|_{\mathfrak{L}^2\big(\mathbb{T}^d\big)}=1$, in contradiction to the fact that $w=0$.
\end{proof}

 \begin{lemma}[Differentiability of Fokker-Planck in relation to the slow variable]
 	Assume that hypotheses II and VI hold. Then the solution of the Fokker-Planck \eqref{eq:fokker_planck_PDE} is differentiable with respect to the slow variable.
 \end{lemma}
 \begin{proof}
 	We will show that the application $x\mapsto \rho\big(x,\cdot\big)$ is differentiable in the direction $e_k$ on $\text{C}\big(\mathbb{T}^{d_Y}\big)$. Let $\rho^{h,k}:\mathbb{X}\times\mathbb{T}^{d_Y}\to\mathbb{R}$ be the quotient
 	\begin{equation}
 		\label{eq:auxiliary_quotient_01}
 		\rho^{h,k}\big(\bar{x},y\big):=\frac{\rho\big(\bar{x}+e_kh,y\big)-\rho\big(\bar{x},y\big)}{h}.
 	\end{equation}
 	Applying the operator $\mathcal{L}^{\bar{x},*}$ to both sides of the quotient \eqref{eq:auxiliary_quotient_01} and noting that $\mathcal{L}^{\bar{x},*}\rho\big(\bar{x},y\big)=0$, we obtain, by linearity of $\mathcal{L}^{\bar{x},*}$, that
 	\begin{equation}
 		\label{eq:auxiliary_equation_04}
 		\mathcal{L}^{\bar{x},*}\rho^{h,k}\big(\bar{x},y\big)=\frac{1}{h}\mathcal{L}^{\bar{x},*}\rho\big(\bar{x}+e_kh,y\big).
 	\end{equation}
 	 Consider the functions $A:\mathbb{X}\times\mathbb{T}^{d_Y}\to\mathbb{S}^{d_Y}\big(\mathbb{R}\big)$, $c: \mathbb{X}\times\mathbb{T}^{d_Y}\to\mathbb{R}$ and $b_i:\mathbb{X}\times\mathbb{T}^{d_Y}\to\mathbb{R}$ defined respectively by: 
 	\begin{equation*}
 		a_{ij}\big(x,y\big):=\big[\sigma_{Y}\sigma^{\top}_{Y}\big]_{ij}\big(x,y\big),\hspace{0.08cm}c\big(x,y\big):=\sum_{i,j}^{d_Y}\frac{1}{2}\partial^2_{y_iy_j}a_{ij}\big(x,y\big)-\sum_{i=1}^{d_Y}\partial_{y_i}\mu_{Y,i}\big(x,y\big),\hspace{0.08cm}\text{and}\hspace{0.08cm}b_i\big(x,y\big):=\sum_{j=1}^{d_Y}\partial_{y_j}a_{ij}\big(x,y\big)-\mu_{Y,i}\big(x,y\big).
 	\end{equation*}
 	With these functions we define:
 	\begin{subequations}
 		\label{subeq:coefficients_operator_h_01}
 		\begin{align}
 			\label{eq:coefficient_auxiliary_diffusion_h_01}
 			a^{h,k}_{ij}\big(x,y\big)&:=\frac{a_{ij}\big(x+e_kh,y\big)-a_{ij}\big(x,y\big)}{h},\\
 			\label{eq:coefficient_auxiliary_drift_h_01}
 			b^{h,k}_i\big(x,y\big)&:=\frac{b_i\big(x+e_kh,y\big)-b_i\big(x,y\big)}{h},\\
 			\label{eq:coefficient_auxiliary_function_c_h_01}
 			c^{h,k}\big(x,y\big)&:=\frac{c\big(x+e_kh,y\big)-c\big(x,y\big)}{h}.
 		\end{align}
 	\end{subequations}
 	Note that
 		\begin{align*}
 		c^{h,k}\big(\bar{x},y\big):=&\frac{c\big(\bar{x}+he_k,y\big)-c\big(\bar{x},y\big)}{h}\\
 		=&\sum_{i,j}^{d_Y}\frac{1}{2}\partial^2_{y_i,y_j}\Bigg(\frac{a_{ij}\big(\bar{x}+he_k,y\big)-a_{ij}\big(\bar{x},y\big)}{h}\Bigg)-\sum_{i=1}^{d_Y}\partial_{y_i}\Bigg(\frac{\mu_{Y,i}\big(\bar{x}+he_k,y\big)-\mu_{Y,i}\big(\bar{x},y\big)}{h}\Bigg)\\
 		=&\sum_{i,j}^{d_Y}\frac{1}{2}\partial^2_{y_i,y_j}a_{ij}^{h,k}\big(\bar{x},y\big)-\sum_{i=1}^{d_Y}\partial_{y_i}\Bigg(\frac{\mu_{Y,i}\big(\bar{x}+he_k,y\big)-\mu_{Y,i}\big(\bar{x},y\big)}{h}\Bigg)\\
 		=&\sum_{i,j}^{d_Y}\frac{1}{2}\partial^2_{y_i,y_j}a_{ij}^{h,k}\big(\bar{x},y\big)-\sum_{i=1}^{d_Y}\partial_{y_i}\Bigg(\frac{\big(\sum_{j=1}^{d_Y}\partial_{y_j}a_{ij}\big(\bar{x}+he_k,y\big)-b_i\big(\bar{x}+he_k,y\big)\big)-\big(\sum_{j=1}^{d_Y}\partial_{y_j}a_{ij}\big(\bar{x},y\big)-b_i\big(\bar{x},y\big)\big)}{h}\Bigg)\\
 		=&\sum_{i,j}^{d_Y}\frac{1}{2}\partial^2_{y_i,y_j}a_{ij}^{h,k}\big(\bar{x},y\big)-\sum_{i,j}^{d_Y}\partial^2_{y_i,y_j}a_{ij}^{h,k}\big(\bar{x},y\big)+\sum_{i=1}^{d_Y}\partial_{y_i}b^{h,k}_i\big(\bar{x},y\big)\\
 		=&\sum_{i=1}^{d_Y}\partial_{y_i}b^{h,k}_i\big(\bar{x},y)-\sum_{i,j}^{d_Y}\frac{1}{2}\partial^2_{y_i,y_j}a_{ij}^{h,k}\big(\bar{x},y\big).
 	\end{align*}
 	That is
 	\begin{equation}
 		\label{eq:auxiliary_equation_05}
 		c^{h,k}\big(\bar{x},y\big)=\sum_{i=1}^{d_Y}\partial_{y_i}b^{h,k}_i\big(\bar{x},y\big)-\sum_{i,j}^{d_Y}\frac{1}{2}\partial^2_{y_i,y_j}a_{ij}^{h,k}\big(\bar{x},y\big).
 	\end{equation}
 	
 	 On the other hand, consider the function $f^{h,k}:\mathbb{X}\times\mathbb{T}^{d_Y}\to\mathbb{R}$ defined by
 	\begin{equation}
 		\label{eq:auxiliary_function_F_h_k}
 		f^{h,k}\big(x,y\big):=\frac{1}{2}\sum_{i,j}^{d_Y}a^{h,k}_{ij}\big(x,y\big)\partial^2_{y_i,y_j}\rho\big(x+e_kh,y\big)+\sum_{i=1}^{d_Y}b^{h,k}_i\big(x,y\big)\partial_{y_i}\rho\big(x+e_kh,y\big)+c^{h,k}\big(x,y\big)\rho\big(x+e_kh,y\big).
 	\end{equation}
 	From Lemma \ref{lm:continuity_Fokker_Planck_relation_slow_variable} we have that $f^{h,k}\in \text{C}^{0,1}\big(\mathbb{X}\times \mathbb{T}^{d_Y}\big)$ and  combining the equation \eqref{eq:auxiliary_function_F_h_k} with the equality \eqref{eq:auxiliary_equation_04}, we get
 	\begin{equation}
 	\label{eq:PDE_derivative_parameter_h_k_01}
 		-\mathcal{L}^{\bar{x},*}\rho^{h,k}\big(\bar{x},y\big)=f^{h,k}\big(\bar{x},y\big)\quad \forall y\in \mathbb{T}^{d_Y}.
 	\end{equation}
 	Now notice that
 	\begin{equation*}
 		\int_{\mathbb{T}^{d_Y}}\rho^{h,k}\big(\bar{x},y\big)\,dy=0
 	\end{equation*}
    and	again from Lemma \ref{lm:continuity_Fokker_Planck_relation_slow_variable} we have that $\rho^{h,k}\in \text{C}^{0,2}\big(\mathbb{X}\times \mathbb{T}^{d_Y}\big)$ and therefore $\rho^{h,k}\big(\bar{x},\cdot\big)\in \mathfrak{H}^1\big(\mathbb{T}^{d_Y}\big)$. 
    
    Let's prove that $f^{h,k}\big(\bar{x},\cdot\big)\in \mathfrak{L}^2\big(\mathbb{T}^{d_Y}\big)$. To do this, it suffices to show that 
 	\begin{equation}
 		\label{eq:auxiliary_equation_06}
 		\int_{\mathbb{T}^{d_Y}}f^{h,k}\big(\bar{x},y\big)\,dy=0.
 	\end{equation}
 	First of all, we have
 	\begin{align*}
 		\int_{\mathbb{T}^{d_Y}}f^{h,k}\big(\bar{x},y\big)dy=\frac{1}{2}\sum_{i,j}^{d_Y}&\int_{\mathbb{T}^{d_Y}}a^{h,k}_{ij}\big(\bar{x},y\big)\partial^2_{y_i,y_j}\rho\big(\bar{x}+he_k,y\big)\,dy
 		+\sum_{i=1}^{d_Y}\int_{\mathbb{T}^{d_Y}}b^{h,k}_i\big(\bar{x},y\big)\partial_{y_i}\rho\big(\bar{x}+he_k,y\big)\,dy\\
 		+&\int_{\mathbb{T}^{d_Y}}c^{h,k}\big(\bar{x},y\big)\varrho\big(\bar{x}+he_k,y\big)\,dy.
 	\end{align*}
 	Integrating by parts, it follows that
 	 	\begin{align*}
 		\int_{\mathbb{T}^{d_Y}}a^{h,k}_{ij}\big(\bar{x},y\big)\partial^2_{y_i,y_j}\rho\big(\bar{x}+he_k,y\big)\,dy&=\int_{\mathbb{T}^{d_Y}}\partial^2_{y_i,y_j}a^{h,k}_{ij}\big(\bar{x},y\big)\rho\big(\bar{x}+he_k,y\big)\,dy\\
 		\int_{\mathbb{T}^{d_Y}}b^{h,k}_i\big(\bar{x},y\big)\partial_{y_i}\rho\big(\bar{x}+he_k,y\big)\,dy&=-\int_{\mathbb{T}^{d_y}}\partial_{y_i}b^{h,k}_i\big(\bar{x},y\big)\rho\big(\bar{x}+he_k,y\big)\,dy.
 	\end{align*}
 	Thus, we obtain that
 	\begin{equation*}
 		\int_{\mathbb{T}^{d_Y}}f^{h,k}\big(\bar{x},y\big)dy=\int_{\mathbb{T}^{d_Y}}\Bigg(\frac{1}{2}\sum_{i,j}^{d_Y}\partial^2_{y_i,y_j}a^{h,k}_{ij}\big(\bar{x},y\big)-\sum_{i=1}^{d_Y}\partial_{y_i}b^{h,k}_i\big(\bar{x},y\big)+c^{h,k}\big(\bar{x},y\big)\Bigg)\rho\big(\bar{x}+he_k,y\big)\,dy
 	\end{equation*}
 	 On the other hand, from equation \eqref{eq:auxiliary_equation_05}, we have that 
 	\begin{equation*}
 		\frac{1}{2}\sum_{i,j}^{d_Y}\partial^2_{y_i,y_j}a^{h,k}_{ij}\big(\bar{x},y\big)-\sum_{i=1}^{d_Y}\partial_{y_i}b^{h,k}_i\big(\bar{x},y\big)+c^{h,k}\big(\bar{x},y\big)=0.
 	\end{equation*}
 	Therefore the equation \eqref{eq:auxiliary_equation_06} holds.
 	
 	On the one hand, from Lemma \ref{lm:first_energy_estimate}  there is a constant $C>0$ depending only on the torus $\mathbb{T}^{d_Y}$ and the coefficients of the operator $\mathcal{L}^{\bar{x},*}$ such that 
  	\begin{equation*}
  		\|\rho^{h,k}\big(\bar{x},\cdot\big)\|_{\mathfrak{L}^2\big(\mathbb{T}^d\big)}\leqslant C \|f^{h,k}\big(\bar{x},\cdot\big)\|_{\mathfrak{L}^2\big(\mathbb{T}^d\big)}.
  	\end{equation*}
 	On the other hand, we can extend (by periodicity) the solution $\rho^{h,k}$ to the whole $\mathbb{R}^{d_Y}$ in such a way that the new function continues to satisfy the PDE \eqref{eq:PDE_derivative_parameter_h_k_01}. 
 	
 	Consider the ball $Q_r:=B\big(\big(\sqrt{d_Y}/2,...,\sqrt{d_Y}/2\big),r\big)\subset \mathbb{R}^{d_Y}$ of radius $r>2\sqrt{d_Y}$. Thus we obtain 
 	 \begin{equation}
 		\label{eq:PDE_derivative_parameter_h_k_02}
 		-\mathcal{L}^{\bar{x},*}\rho^{h,k}\big(\bar{x},y\big)=f^{h,k}\big(\bar{x},y\big)\quad \forall y\in Q_r.
 	\end{equation} 
 	Furthermore
 	\begin{equation}
 		\label{eq:auxiliary_inequality_04}
 	   \|\rho^{h,k}\big(\bar{x},\cdot\big)\|_{\text{L}^2\big(Q_r\big)}\leqslant C_{\mathbb{Y}}  \|\rho^{h,k}\big(\bar{x},\cdot\big)\|_{\text{L}_{\text{per}}^2\big(\mathbb{Y}\big)}\leqslant \overline{C}_{\mathbb{Y}} \|f^{h,k}\big(\bar{x},\cdot\big)\|_{\text{L}_{\text{per}}^2\big(\mathbb{Y}\big)}\leqslant\overline{C}_{\mathbb{Y}}\|f^{h,k}\big(\bar{x},\cdot\big)\|_{\text{L}^2\big(Q_r\big)}
 	\end{equation}
  	for a constant $C_{\mathbb{Y}}>0$ proportional to the number of cubes $\mathbb{Y}$ we use to cover the ball $Q_r$ such that the first inequality above holds.
 	
 	 Since $\rho^{h,k}\in \text{C}^{0,2}\big(\mathbb{X}\times \mathbb{T}^{d_Y}\big)$, we have $\rho^{h,k}\big(\bar{x},\cdot\big)\in \text{W}^{p,2}\big(Q_r\big)$ and as $f^{h,k}\in \text{C}^{0,1}\big(\mathbb{X}\times \mathbb{T}^{d_Y}\big)$, we also have $f^{h,k}\big(\bar{x},\cdot\big)\in \text{L}^{\text{p}}\big(Q_r\big)$.
 	 
 	By inequality \eqref{eq:auxiliary_inequality_from_p_q} we have that  
 	\begin{subequations}
 	  \label{subeq:auxiliary_inequalities_01}
 	   \begin{align}
 			&\|\rho^{h,k}\big(\bar{x},\cdot\big)\|_{\text{L}^{1}\big(Q_r\big)}\leqslant \lambda\big(Q_r\big)^{1-\frac{1}{2}}\|\rho^{h,k}\big(\bar{x},\cdot\big)\|_{\text{L}^{2}\big(Q_r\big)}\\
 			&\|f^{h,k}\big(\bar{x},\cdot\big)\|_{\text{L}^{2}\big(Q_r\big)}\leqslant \lambda\big(Q_r\big)^{\frac{1}{2}-\frac{1}{p}}\|f^{h,k}\big(\bar{x},\cdot\big)\|_{\text{L}^{\text{p}}\big(Q_r\big)}.
 		\end{align}
 	\end{subequations}

    Combining the inequalities \eqref{eq:auxiliary_inequality_sobolev}, \eqref{eq:auxiliary_inequality_04} and \eqref{subeq:auxiliary_inequalities_01} we obtain that 
 	\begin{equation}
 		\label{eq:auxiliary_inequality_30}
 		\|\rho^{h,k}\big(\bar{x},\cdot\big)\|_{\text{W}^{\text{p,1}}\big(Q_r\big)}\leqslant \Lambda_{d,r}\big(c_0;C_{A,b,c};r_1\big)\overline{C}_{\mathbb{Y}}\bigg(\lambda\big(Q_r\big)^{1-\frac{1}{2}}+\lambda\big(Q_r\big)^{\frac{1}{2}-\frac{1}{p}}\bigg)\|f^{h,k}\big(\bar{x},\cdot\big)\|_{\text{L}^{\text{p}}\big(Q_r\big)}
 	\end{equation} 
 	According to item VI, the following limits holds:
 	\begin{equation*}
 		\lim_{h\to0}a^{h,k}_{ij}\big(x,y\big)=\partial_{x_k}a_{ij}\big(x,y\big),\quad\lim_{h\to0}b^{h,k}_i\big(x,y\big)=\partial_{x_k}b_i\big(x,y\big)\quad\text{and}\quad
 		\lim_{h\to0}c^{h,k}\big(x,y\big)=\partial_{x_k}c\big(x,y\big)\quad \forall (x,y)\in \mathbb{X}\times\mathbb{T}^{d_Y}. 
 	\end{equation*}
    Furthermore, by the same item VI, the functions $\partial_{x_k}a_{ij}$, $\partial_{x_k}b_i$ and $\partial_{x_k}c$ are continuous in $\mathbb{X}\times\mathbb{T}^{d_Y}$. Thus, there exists a constant $C_0>0$, such that
 	\begin{equation}
 		\label{eq:auxiliary_bounded_01}
 		\sup_{h>0}\sup_{(x,y)\in \mathbb{X}\times\mathbb{T}^{d_Y}}\Biggl\{\bigg|a^{h,k}_{ij}\big(x,y\big)\bigg|+\bigg|b^{h,k}_i\big(x,y\big)\bigg|+\bigg|c^{h,k}\big(x,y\big)\bigg|\Biggr\}<C_0.
 	\end{equation}
 	In addition, from itens II and VI, we obtain the following inequality in the proof of Lemma \ref{lm:continuity_Fokker_Planck_relation_slow_variable}
 	\begin{equation}
 		\label{eq:auxiliary_bounded_02}
 		\|\rho\big(\bar{x}+e_kh,\cdot\big)\|_{\text{C}_{\text{hol}}^{2,1}\big(\overline{Q}_{\frac{r}{2}}\big)}\leqslant K_1.
 	\end{equation}
 	Combining inequalities \eqref{eq:auxiliary_bounded_01} and \eqref{eq:auxiliary_bounded_02}, it follows from the definition of $f^{h,k}$ (equation \eqref{eq:auxiliary_function_F_h_k}) that there exists a constant $\widetilde{C}>0$ such that 
 	\begin{equation}
 		\label{eq:auxiliary_bounded_03}
 		\|f^{h,k}\big(\bar{x},\cdot\big)\|_{\text{L}^{\text{p}}\big(\overline{Q}_{\frac{r}{2}}\big)}\leqslant \widetilde{C}.
 	\end{equation}
 	From the inequalities \eqref{eq:auxiliary_inequality_30} and \eqref{eq:auxiliary_bounded_03} we have 
 	\begin{equation}
 		\label{eq:auxiliary_bounded_04}
 			\|\rho^{h,k}\big(\bar{x},\cdot\big)\|_{\text{W}^{\text{p,1}}\big(\overline{Q}_{\frac{r}{2}}\big)}\leqslant \widetilde{\gamma},
 	\end{equation}
 	where 
 	\begin{equation*}
 		\widetilde{\gamma}:=\Lambda_{d,r}\big(c_0;C_{A,b,c};r_1\big)\overline{C}_{\mathbb{Y}}\bigg(\lambda\big(Q_r\big)^{1-\frac{1}{2}}+\lambda\big(Q_r\big)^{\frac{1}{2}-\frac{1}{p}}\bigg)\widetilde{C}.
 	\end{equation*}
 	Now from Theorem \ref{tm:rellich_kondrachov} we obtain that 
 	\begin{equation}
 		\label{eq:auxiliary_bounded_05}
 		\|\rho^{h,k}\big(\bar{x},\cdot\big)\|_{\text{C}_{\text{hol}}^{0,1-\frac{d_Y}{p}}\big(\overline{Q}_{\frac{r}{2}}\big)}\leqslant C\widetilde{\gamma}.
 	\end{equation}
 	By Theorem \ref{tm:arzela_ascoli_espaco_holder} there exists a subsequence $\big(\rho^{h_i,k}\big(\bar{x},\cdot\big)\big)_{i\geqslant0}$  and a function $\varrho^{k}\big(\bar{x},\cdot\big)\in \text{C}_{\text{hol}}^{0,1-\frac{d_Y}{p}}\big(\overline{Q}_{\frac{r}{2}}\big)$, such that this sequence converges uniformly to $\varrho^{k}\big(\bar{x},\cdot\big)$ in $\text{C}\big(\overline{Q}_{\frac{r}{2}}\big)$.
 	Since the convergence is uniform, $\varrho^{k}\big(\bar{x},\cdot\big)$ is periodic and,
 	therefore, we can consider that $\varrho^k\big(\bar{x},\cdot\big)$ is defined on the $\mathbb{T}^{d_Y}$. 
 	
 	From the inequality \eqref{eq:auxiliary_bounded_05} and the fact that the convergence is uniform, we have from the dominated convergence theorem that 
 	\begin{equation*}
 		\int_{\mathbb{T}^{d_Y}}\varrho^k\big(\bar{x},y\big)dy=\lim_{i\to+\infty}\int_{\mathbb{T}^{d_Y}}\rho^{h_i,k}\big(\bar{x},y\big)dy=0.
 	\end{equation*}
 	Thus, $\varrho^k\big(\bar{x},\cdot\big)\in \mathfrak{H}^{1}\big(\mathbb{T}^{d_Y}\big)$.

 	On the other hand, consider the function $f^{k}:\mathbb{X}\times\mathbb{T}^{d_Y}\to\mathbb{R}$ defined by
 	\begin{equation*}
 	   f^{k}\big(x,y\big):=\frac{1}{2}\sum_{i,j}^{d_Y}\partial_{x_k}a_{ij}\big(x,y\big)\partial^2_{y_i,y_j}\rho\big(x,y\big)+\sum_{i=1}^{d_Y}\partial_{x_k}b_i\big(x,y\big)\partial_{y_i}\rho\big(x,y\big)+\partial_{x_k}c\big(x,y\big)\rho\big(x,y\big).
 	\end{equation*}
 	We have that 
 	\begin{equation}
 		\label{eq:auxiliary_limit_01}
 		\lim_{h\to0}f^{h,k}\big(\bar{x},y\big)=f^{k}\big(\bar{x},y\big)\quad \forall (x,y)\in \mathbb{X}\times\mathbb{T}^{d_Y}. 
 	\end{equation}
 	From the inequality \eqref{eq:auxiliary_bounded_03} and the convergence \eqref{eq:auxiliary_limit_01}, we have from the dominated convergence theorem that 
 	\begin{equation*}
 		\int_{\mathbb{T}^{d_Y}}f^k\big(\bar{x},y\big)dy=\lim_{i\to+\infty}\int_{\mathbb{T}^{d_Y}}f^{h_i,k}\big(\bar{x},y\big)dy=0.
 	\end{equation*}
 	Thus, $f^k\big(\bar{x},\cdot\big)\in \mathfrak{L}^{2}\big(\mathbb{T}^{d_Y}\big)$. 
 	
 	From the inequalities \eqref{eq:auxiliary_bounded_03} and \eqref{eq:auxiliary_bounded_05} combined with the fact that $\rho^{h_i,k}\big(\bar{x},\cdot\big)$ and $f^{h_i,k}\big(\bar{x},\cdot\big)$ converge to $\varrho^k\big(\bar{x},\cdot\big)$ and $f^{k}\big(\bar{x},\cdot\big)$, respectively, it follows from the dominated convergence theorem and  the definition of weak solution (Definition \ref{def:weak_solution_H}) that
 	 \begin{equation}
 	 	\label{eq:PDE_derivative_parameter_limite}
 	 	-\mathcal{L}^{\bar{x},*}\varrho^{k}\big(\bar{x},y\big)=f^{k}\big(\bar{x},y\big)\quad \forall y\in \mathbb{T}^{d_Y}.
 	 \end{equation}
 	 From item VI and Lemma \ref{lm:existence_unicity_periodic_solutions_EDP} we know that PDE \eqref{eq:PDE_derivative_parameter_limite} has only one solution and therefore $\big(\rho^{h,k}\big(\bar{x},\cdot\big)\big)_{h\geqslant0}$ converges uniformly to $\varrho^{k}\big(\bar{x},\cdot\big)$ on $\text{C}\big(\mathbb{T}^{d_Y}\big)$. Then, the application $x\mapsto \rho\big(x,\cdot\big)$ is differentiable at the point $\bar{x}\in\mathbb{X}$ in the direction $e_k$ in $\text{C}\big(\mathbb{T}^{d_Y}\big)$ and $\varrho^k\big(x,\cdot\big)=\partial_{x_k}\rho\big(x,\cdot\big)$. Furthermore, by inequalitie \eqref{eq:auxiliary_bounded_05}, we obtain that
 	\begin{equation*}
 		\|\partial_{x_k}\rho\big(x,\cdot\big)\|_{\text{C}_{\text{hol}}^{0,1-\frac{d_Y}{p}}\big(\mathbb{T}^{d_Y}\big)}\leqslant C\widetilde{\gamma}\quad \forall x\in \mathbb{X}.
 	\end{equation*}
 	In particular
 	\begin{equation*}
 		\|D_x\rho\big(x,y\big)\|\leqslant C\sqrt{d_X}\quad \forall (x,y)\in \mathbb{X}\times \mathbb{T}^{d_Y}.
 	\end{equation*}
 \end{proof}
\bibliographystyle{siam} 
\bibliography{references} 

\end{document}